\documentclass[10pt,a4paper]{article}
\usepackage[a4paper]{geometry}
\usepackage[utf8]{inputenc}
\usepackage[english]{babel}
\usepackage[T1]{fontenc}
\usepackage{amsmath}
\usepackage{amsfonts}
\usepackage{amssymb}
\usepackage{mathrsfs}
\usepackage{bbold}
\usepackage{stmaryrd}
\usepackage{hyperref}
\usepackage{amsthm}
\usepackage{comment}
\usepackage{bm}
\usepackage[dvipsnames]{xcolor}

\title{Norm Convergence Rate for Multivariate Quadratic Polynomials of Wigner Matrices}
\usepackage[auth-sc-lg,affil-sl]{authblk}
\author[a,1]{Jacob Fronk}
\author[a,b,1]{Torben Kr\"{u}ger}
\author[c]{Yuriy Nemish}

\affil[a]{Department of Mathematical Sciences, University of Copenhagen}
\affil[b]{Department of Mathematics, FAU Erlangen-N\"{u}rnberg}
\affil[c]{Department of Mathematics, University of California, San Diego}

\newtheorem{theorem}{Theorem}
\newtheorem{lemma}[theorem]{Lemma}
\newtheorem{prop}[theorem]{Proposition}
\newtheorem{corollary}[theorem]{Corollary}
\newtheorem{definition}[theorem]{Definition}
\newtheorem*{assumption}{Assumption}
\newtheorem*{remark}{Remark}
\newtheorem{case}{Case}

\newcommand{\R}{\mathbb{R}}
\newcommand{\C}{\mathbb{C}}
\newcommand{\N}{\mathbb{N}}
\newcommand{\Z}{\mathbb{Z}}
\newcommand{\HH}{\mathbb{H}}

\newcommand{\dd}{\mathrm{d}}
\newcommand{\E}{\mathbb{E}}
\newcommand{\Pb}{\mathbb{P}}
\newcommand{\I}{\mathrm{i}}
\newcommand{\veps}{\varepsilon}

\newcommand{\mr}{\mathrm}
\newcommand{\mb}{\mathbf}
\newcommand{\mc}{\mathcal}

\newcommand{\IS}{\mathscr{S}}
\newcommand{\IL}{\mathscr{L}}
\newcommand{\IE}{\mathscr{E}}

\DeclareMathOperator{\Tr}{Tr}
\DeclareMathOperator{\diag}{diag}
\DeclareMathOperator{\supp}{supp}
\DeclareMathOperator{\dist}{dist}
\DeclareMathOperator{\rank}{rank}
\DeclareMathOperator{\Image}{Image}
\DeclareMathOperator{\Spec}{Spec}
\DeclareMathOperator{\Span}{span}

\renewcommand{\Im}{\operatorname{\mathrm{Im}}}
\renewcommand{\Re}{\operatorname{\mathrm{Re}}}
\renewcommand{\phi}{\varphi}

\counterwithin{equation}{section}
\counterwithin{theorem}{section}

\begin{document}

\maketitle

\footnotetext[1]{\hspace{0.15cm} Partially supported by VILLUM FONDEN research grant no. 29369\newline  E-mail addresses: \href{mailto:jf@math.ku.dk}{jf@math.ku.dk} (J. Fronk), \href{mailto:torben.krueger@fau.de}{torben.krueger@fau.de} (T. Kr\"{u}ger), \href{mailto:ynemish@ucsd.edu}{ynemish@ucsd.edu} (Yu. Nemish)}

\begin{abstract}
    We study Hermitian non-commutative quadratic polynomials of multiple independent Wigner matrices. We prove that, with the exception of some specific reducible cases, the limiting spectral density of the polynomials always has a square root growth at its edges and prove an optimal local law around these edges. Combining these two results, we establish that, as the dimension $N$ of the matrices grows to infinity, the operator norm of such polynomials $q$ converges to a deterministic limit with a rate of convergence of $N^{-2/3+o(1)}$. Here, the exponent in the rate of convergence is optimal.
    For the specific reducible cases, we also provide a classification of all possible edge behaviours.
\end{abstract}

\section{Introduction}
The empirical spectral distribution of a random matrix is typically well approximated by a deterministic measure as its dimension grows to infinity. A clear contender for the most famous example of such a convergence is the celebrated semi-circle law. It states that the spectral measure of a Wigner matrix, a Hermitian $N \times N$-matrix $\mb X$ with centered i.i.d. entries $x_{ij}$ above the diagonal and $\E |x_{ij}|^2= N^{-1}$,  converges to the semi-circle distribution, supported on the interval $[-2,2]$ \cite{Wigner1955}. In particular, the largest and smallest eigenvalues of $\mb X$ converge to the respective edges of the support, implying the convergence $\| \mb X\| \to 2$ of the operator norm, provided the fourth moments of the entries of $\sqrt{N}{\mb X}$ are finite \cite{Bai1988, Bai2010Spectral}. 

For non-commutative Hermitian polynomials $\mb Q=q(\mb X_1, \dots, \mb X_l)$ in several independent  Wigner matrices $\mb X_i$ an analogous statement holds.
 In this setup the limit of the eigenvalue distribution equals the distribution associated to the polynomial $\mathfrak{q}=q(\mathfrak{s}_1, \dots, \mathfrak{s}_l)$, where the matrices are replaced by free semicircular random variables within a non-commutative probability space, i.e. $\mathfrak{s}_i$ can be interpreted as operators acting on an infinite dimensional Hilbert space. This result was first established for  Gaussian random matrices in \cite{Voiculescu91Limit} and extended to Wigner matrices in \cite{Dykema1993Certain}. Similarly, the convergence of the norms $\| \mb Q\| \to \|\mathfrak{q}\|$ was first shown by Haagerup-Thorbj{\o}rnsen \cite{haagerup2005} in the Gaussian case and the Wigner case was proven by Anderson in \cite{Anderson2013Convergence}. Such results have also been shown when some  $\mb X_i$ are replaced by non-random matrices \cite{Belinschi2017Spectral, Male2012Norm}. For non-Hermitian polynomials convergence of the spectral measure to the limiting Brown measure, predicted by free probability theory, is known only for very specific cases, e.g. for products of random matrices \cite{Goetze2011Asymptotic, Orourke2011Products} and for quadratic polynomials \cite{cook2020spectrum}. 
 
Determining the limiting spectral measure $\rho$ of Hermitian polynomials $\mb Q$, or equivalently the distribution of $\mathfrak{q}$, becomes a nontrivial task beyond particular computable cases. Several works have been devoted to the analysis of the regularity properties of $\rho$.  It has been shown that $\rho$ has a single interval support \cite{Haagerup2006} and does not contain any atoms \cite{Mai2017Absence, Shlyakhtenko_2015}. The cumulative distribution function of $\rho$ is H\"older-continuous \cite{Banna2020Hoelder}, and if $q$ is a monomial or a homogeneous quadratic polynomial then $\rho$ is even absolutely continuous \cite{Charlesworth2016Free,erdos2019local}.

In the present paper we consider the case when $q$ is a general polynomial of degree two, i.e. we study self-adjoint polynomials in $l$ independent Wigner matrices $ \mb X_1,\ldots,\mb X_l$  of the form
\begin{equation}
q(\mb X_1,\ldots,\mb X_l)=\sum_{i,j=1}^l\mb X_iA_{ij}\mb X_j\mb + \sum_{i=1}^l b_i\mb X_i+c,
\label{eq polynomial}
\end{equation}
where $0\neq A=(A_{ij})\in\C^{l\times l}$ is a Hermitian matrix, $b=(b_i)\in\R^l$ and $c\in\R$. For these polynomials we classify the edge behaviour of $\rho$ and show (see Proposition~\ref{prop square root} below) that  at both edges the limiting spectral measure is absolutely continuous and apart from specific reducible cases its density exhibits a square root growth. The reducible cases are, up to a shift and change in sign, of the form $ \mb Y^*\mb Y$, where $\mb Y$ is an affine combination of the underlying Wigner matrices. Such polynomials still have a square root edge at the rightmost point of the spectrum, but  have a density blow-up at  the leftmost point if it is equal to zero.  All these cases are classified in Proposition~\ref{prop edge} below.

The square root growth of the limiting spectral distribution is a well-known phenomenon, that is already present in the semicircle law for Wigner matrices $\mb X$. In this setup the rate of convergence for the norm of $\mb X$ is  $\|\mb X\| = 2+\mc O(N^{-2/3+o(1)})$ with very high probability \cite{Erdos_2013Local} and several tail estimates have been established \cite{Augeri2016, Augeri2021, Erdos2023smalldeviation}. In fact, for Wigner matrices the  distribution of  the largest eigenvalue is known to be universal and given by the Tracy-Widom law \cite{ERDOSRigidity2012, Soshnikov1999Universality, Tao2010Random}. This distribution was first identified by Tracy and Widom for the Gaussian ensembles \cite{Tracy_1994, Tracy_1996} and necessary and sufficient conditions for its universality in the context of Wigner matrices identified  in \cite{Lee2014Necessary}. 
 
Edge universality has been extended to many other Hermitian random matrix models, including invariant ensembles \cite{Bourgade2014}, covariance matrices \cite{Feldheim2010, Pillai2014}, deformed Wigner matrices \cite{Lee2015}, deterministic matrices with Gaussian perturbations \cite{Landon2017} and  models with correlations \cite{Alt_2020edge, Pillai_2012}. Such universality results often rely on control of the  eigenvalue location on  mesoscopic scales between $O(1)$ and $O(N^{-1})$, i.e. on a local law.

Local laws arose in the context of Wigner matrices \cite{Erdos_2013Local, Tao13} and have subsequently been extended to the more complex models listed above. For non-commuting polynomials of several random matrices local laws are known only in specific cases, starting with Anderson's work on the anti-commutator $\mb X_1 \mb X_2 + \mb X_2 \mb X_1$ of Wigner matrices \cite{Anderson_2015} that controls the deviation of bulk eigenvalues from their expected position on the scale $\mc O(N^{-1/2})$. For general polynomials, that satisfy certain checkable conditions, an optimal  local law in the bulk regime was proved in \cite{erdos2019local}. Related results are \cite{Bao2017} and \cite{Bao2020}, where  it was shown that for two random matrices satisfying a local law and whose eigenvectors are in generic direction to each other also their sum satisfies a local law at the edge and in the bulk. This result covers e.g. $\mb X_1^2+ \mb X_2^3$ if one of these matrices is a Gaussian unitary ensemble. For  non-Hermitian polynomials the results \cite{Nemish2017Local} and \cite{Goetze2020Local} cover products of independent matrices with i.i.d. entries. and \cite{cook2020spectrum} quadratic polynomials.  

Currently the best estimate on the convergence rate of the norm for general polynomials of GUE matrices is $-N^{-\veps}\le \|\mb Q\| -\|\mathfrak{q}\| \le CN^{-1/4}$, for some $\veps<\frac{1}{3}$ and $C>0$ established in \cite{Collins2022Operator}.  Our  work improves this bound for polynomials $\mb Q =q(\mb X_1,\ldots,\mb X_l)$ of the form \eqref{eq polynomial} to the optimal rate of $N^{-2/3+o(1)}$ with square root growth at the edges of the spectral density and extends the result to Wigner matrices. The main novelty here is a detailed analysis of the Dyson equation, describing a generalized resolvent of the linearization matrix associated with the polynomial in the limit $N \to \infty$ and, consequently, the resolvent of $\mb Q$ itself. The idea of linearising polynomials of random matrices in this way stems from \cite{Haagerup2006, haagerup2005} and has been used in many works since, in particular in \cite{Anderson2013Convergence, Anderson2010Book, Belinschi2017, HELTON2018, HELTON2006}. In particular, we perform a comprehensive stability analysis of the Dyson equation that allows us (i) to prove a square root growth of the limiting spectral density $\rho$ and (ii) to establish an optimal bound on the difference between the solution to the Dyson equation and the generalized resolvent by using a modification of the bound on the random error matrix in the Dyson equation from \cite{ERDOS_2019SCD}. The main insight is that  the matrix  Dyson equation for the linearization, which has a linear self-energy term, can be reduced to a scalar equation for a function $m=m(z)$ of the form 
 \begin{equation}
            -\frac{1}{m}=z+\gamma(m),
\end{equation}
where the self-energy term $\gamma(m)$ is now a non-linear function of $m$. This representation allows us to identify the values of the spectral edges in terms of the coefficients of $q$ and study the quadratic singularity  at these edge points.

\subsection*{Notations\label{Sec notations}}

In this section, we introduce some definitions commonly used throughout the paper. The standard scalar product of vectors $v,w\in\C^n$ will be denoted by $\langle  v, w\rangle$ and the standard euclidean norm by $\Vert  v\Vert=\sqrt{\langle  v, v\rangle}$. A vector $ v$ is called normalized if $\Vert  v\Vert=1$.

Matrices $R\in\C^{k\times n}$ for (fixed) $k,n\in\N$ are usually denoted by non-boldfaced roman letters and matrices $\mb R\in\C^{kN\times nN}$ are usually denoted by boldfaced roman letters. In particular, we denote the identity matrix on $\C^{k\times k}$ by $I_k$ and the identity matrix on $\C^{kN\times kN}$ by $\mb I_{kN}$. 
For any $k,n\in\N$ we embed $\C^{k\times n}$ in $\C^{k\times n}\otimes\C^{N\times N}$ by identifying $R\in\C^{k\times n}$ with $R\otimes \mb I_N\in\C^{k\times n}\otimes\C^{N\times N}\cong\C^{kN\times nN}$ and write compactly 
    \begin{equation}
        R=R\otimes \mb I_N\in\C^{(l+1)N\times(l+1)N}.
        \label{eq embedding}
    \end{equation}
Matrices $R\in\C^{k\times n}$, which are embedded into $\C^{kN\times nN}$, get still denoted by non-boldfaced letters.

For $R,T\in\C^{n\times n}$ we denote the normalized trace by
$\langle R\rangle=\frac1n\Tr R$ and define a scalar product by $\langle R,T\rangle:=\langle R^* T\rangle$. We use the standard operator norm and the Hilbert-Schmidt norm, which are given by 
    \begin{equation}
        \Vert R\Vert=\sup_{\Vert x\Vert\leq1}\Vert R x\Vert
        \quad
        \text{and}
        \quad
        \Vert R\Vert_{\mr{hs}}=\sqrt{\langle R^*R\rangle}.
    \end{equation}
for vectors of matrices $V=(V_i)_{i\in\llbracket l\rrbracket}\in(\C^{n\times n})^l$ we denote by $\Vert V\Vert$ the maximum of the operator norms of the entries, i.e.
    \begin{equation}
     \Vert V\Vert=\max_{i\in\llbracket l\rrbracket}\Vert V_i\Vert.  
    \end{equation}
For random matrices $\mb R\in\C^{kN\times kN}$  the isotropic and averaged p-norms are defined by
    \begin{equation}
    \Vert \mb R\Vert_p:= \sup_{\Vert \mb x\Vert,\Vert\mb y\Vert\leq1}\left(\E|\langle\mb x,\mb R\mb y\rangle|^p\right)^{\frac1p}
    \quad
    \mr{and}
    \quad
    \Vert\mb R\Vert^{\mr{av}}_p:= \sup_{\Vert \mb B\Vert\leq1}\left(\E|\langle \mb B\mb R\rangle|^p\right)^{\frac1p}.
    \end{equation}
For a block matrix $\mb R\in\C^{kN\times nN}$ with blocks $\mb R_{ij}\in\C^{N\times N}$, $i\in\llbracket k\rrbracket$ and $j\in\llbracket n\rrbracket$, we define the blockwise (averaged) trace $\underline{\mb R}\in\C^{k\times n}$ by
    \begin{equation}
        \underline{\mb R}_{ij}=\langle\mb R_{ij}\rangle.
        \label{underline op}
    \end{equation}
A matrix $R$ is said to be positive definite if $\langle  v,R v\rangle>0$ for all $ v\in\C^{n\times n}\setminus\{0\}$ and we write $R>0.$ It is called positive semi-definite if $\langle  v,R v\rangle\geq0$ for all $ v\in\C^{n\times n}\setminus\{0\}$ and we write $R\geq0$. $R$ is called negative (semi-)definite, denoted by $R<0$ ($R\leq0$), if $-R$ is positive (semi)-definite. For $S\in\C^{n\times n}$ Hermitian we write $S\geq R$ if $S-R\geq0$ and $S>R$ if $S-R>0$.

For linear operators acting on matrix spaces, we denote by $\Vert\cdot\Vert_{\mr{sp}}$ the norm induced by the Hilbert-Schmidt norm, $\Vert\cdot\Vert_{\mr{hs}}.$ The identity map between matrix spaces is denoted by $\mathbb 1$, i.e. $\mathbb 1[R]=R$ for all $R\in\C^{k\times n}$.

The upper half-plane will be denoted by $\HH$, i.e. 
    \begin{equation}
        \HH=\{z\in\C: \Im z>0\}
    \end{equation}
and $\llbracket n\rrbracket:=\{1,2,\ldots, n\}$ is used for the natural numbers up to $n$. If $X$ and $Y$ are positive quantities, we use the notation $X\lesssim Y$ if there is some constant $c>0$ such that $X\leq cY$. The constant will in general depend on the coefficients of $q$. If it also depends on some other parameters $\alpha$, we write $X\lesssim_\alpha Y$. In particular, the constant will never depend on the dimension of our random matrices, $N$. If both $X\lesssim Y$ and $Y\lesssim X$ hold true , we write $X\sim Y$.

\subsection*{Acknowledgements}
We thank Peter Henning Thomsen Krarup for his valuable contributions in the early stages of the project.

\section{Main results \label{Sec main results}}
\begin{assumption}
Let $\zeta_0$ be a real-valued and $\zeta_1$ be a complex-valued random variable and let $\zeta_0$ and $\zeta_1$ be independent. For $i=0,1$, they are to satisfy
\begin{equation}
    \E[\zeta_i]=0,\quad
    \E[|\zeta_i|^2]=1\quad
    \text{and}
    \quad
    \E[|\zeta_i|^p]\leq C_p
    \label{distribution assumption}
\end{equation}
for all $p\in \N$ and some constants $C_p>0$, depending on $p$. Let $\mb W\in\C^{N\times N}$ be a Hermitian random matrix characterized by its entry distribution:
\begin{enumerate}
    \item The diagonal entries $\{w_{ii}:\,i\in\llbracket N\rrbracket\}$ and off-diagonal entries $\{(w_{ij},w_{ji}):\,i,j\in\llbracket N\rrbracket,\,i<j \}$ are independent;
    \item $\{w_{ii}:\,i\in\llbracket N\rrbracket\}$ consists of independent copies of $\frac{1}{\sqrt{N}}\zeta_0$,
    \item  $\{(w_{ij},w_{ji}):\,i,j\in\llbracket N\rrbracket,\,i<j \}$ consists of independent copies of $\frac{1}{\sqrt{N}}(\zeta_1,\bar{\zeta}_1)$.
\end{enumerate}
For a fixed $l\in\N$ we define $\mb X=(\mb X_i)_{i\in\llbracket l\rrbracket}\in(\C^{N\times N})^l$, a vector of random matrices, where each $\mb X_i$, $i\in\llbracket l\rrbracket$, is an independent copy of $\mb W$.
\end{assumption}

To present our results we first need to distinguish between shifted reducible and non-reducible quadratic polynomials.

\begin{definition}[Shifted reducible and non-reducible second degree polynomial]
    We call any non-commutative quadratic polynomial of the matrices $\mb X=(\mb X_i)_{i\in\llbracket l\rrbracket}$ which is of the form 
    \begin{equation}
        q_r(\mb X)=\alpha(v^*\mb X-\xi)(v^*\mb X-\xi)^*-\beta
        \label{eq shifted red}
    \end{equation}
    for some $\alpha,\beta,\xi\in\R$, $\alpha\neq0$, $\xi\geq0$ and $v\in\C^l$ with $\Vert v\Vert=1$ a \textbf{shifted reducible quadratic polynomial}.
    Any polynomial not of this form is called a \textbf{non-reducible quadratic polynomial}.
\end{definition}
\begin{remark}
    Shifted reducible quadratic polynomials are exactly the polynomials of the form \eqref{eq polynomial} with coefficients $A=\alpha  v v^*$, $ b=-\alpha\xi( v+\bar{ v})$ and $c=\alpha|\xi|^2-\beta$ for some $\alpha,\beta,\xi\in\R$ with $\alpha\neq0$, $\xi\geq0$ and normalised $v\in\C^l$. Note that our definition of a shifted reducible polynomial also allows for polynomials of the form \eqref{eq shifted red} with $\xi\in\C$ since $(v^*\mb X-\xi)(v^*\mb X-\xi)^*=(e^{\I\phi}v^*\mb X-e^{\I\phi}\xi)(e^{\I\phi}v^*\mb X-e^{\I\phi}\xi)^*$ for all $\phi\in\R$ and we used this invariance to restrict to $\xi\in\R_{\geq0}$ w.l.o.g..
\end{remark}
Shifted reducible polynomials are those where the edge characteristics reduce to understanding the singular value statistics of some (not necessarily Hermitian) first order polynomial in $\mb X$, whereas non-reducible polynomials are those where such a simplification is not possible. The main focus of this work are non-reducible polynomials, but we also characterize the limiting spectral measure of the shifted reducible polynomials.

For non-reducible polynomials, we prove the convergence of the norm of $q(\mb X)$ to a deterministic $\tau_*$ in the following sense:

 \begin{theorem}[Convergence of the matrix norm]
Let $q$ be a non-reducible quadratic polynomial of the form \eqref{eq polynomial}. There is a deterministic $\tau_*>0$, only depending on the coefficients $A, b,c$ of $q$, such that for all $\veps,D>0$ the operator norm of $q(\mb X)$ satisfies the estimate
\begin{equation*}
\Pb\left(|\Vert q(\mb X)\Vert-\tau_*|\geq N^{-\frac{2}{3}+\veps}\right)\leq C_{\veps,D} N^{-D}.
\end{equation*}
\label{theorem norm}
\end{theorem}
\begin{remark}
The deterministic value $\tau_*$ in Theorem~\ref{theorem norm} is the value of the norm as predicted by the limiting spectral measure
\begin{equation}
    \rho:=\lim_{N\to\infty}\frac1N\sum_{\mu\in\Spec(q(\mb X))}\delta_\mu,
    \label{eq SCDOS}
\end{equation}
where $\delta_\mu$ denotes the Dirac measure at point $\mu$ and the sum runs over all eigenvalues of $q(\mb X)$ accounting also for multiplicity.
That is, we have $\tau_*=\max\{|\tau_+|,|\tau_-|\}$, where $\supp(\rho)=[\tau_-,\tau_+]$ (see Definition~\ref{def edges} below). The points $\tau_+$ and $\tau_-$ can be obtained by solving an explicit polynomial equation, for details see Lemma~\ref{lemma h non-red} below.
\end{remark}

To obtain the main theorem we need several intermediate results. They  establish that the eigenvalue density of any non-reducible polynomial $q(\mb X)$ approximately shows a square root behaviour around its edge. For reducible polynomials, we classify the different edge behaviours. 

The central object of our interest is the Stieltjes transform of the limiting spectral measure $\rho$, which we denote by $m$. The function $m$ is uniquely defined by the following proposition.
\begin{prop}[Existence and uniqueness of the Stieltjes transform]
    Let $A\in\C^{l\times l}$, $A = A^*$, $A \neq 0$, $b\in\R^l$ and $c\in\R$.
    There is a unique function $m:\HH\to\HH$ such that
    \begin{enumerate}
        \item $m$ is complex analytic on all of $\HH$;\label{Cond 1 Stieltjes}
        \item $\lim_{z\to\infty}zm(z)=-1$;\label{Cond 2 Stieltjes}
        \item For all $z\in\HH$ the equation
        \begin{equation}
            -m^{-1}=z+\gamma(m),
            \label{scalar equation}
        \end{equation}
        with
        \begin{equation}
            \gamma(m):=-\Tr A(I_l+Am)^{-1}+m b^t\left((I_l+2m\widehat{A})^{-2}(I_l+m\widehat{A})\right) b-c,
            \label{eq gamma}
        \end{equation}
        is satisfied for $m=m(z)$. Here, the notation $\widehat{A}=\frac{1}{2}(A+A^t)$ was used.
    \end{enumerate}
    \label{prop SCE}
\end{prop}
The proof of Proposition~\ref{prop SCE} is deferred to Appendix~\ref{Appendix DE}.
\begin{definition}[Self-consistent density of states]
        By Conditions~\ref{Cond 1 Stieltjes} and \ref{Cond 2 Stieltjes} the function $m$ in Proposition~\ref{prop SCE} is the Stieltjes transform of a unique probability measure $\rho$ on the real line, i.e.
    \begin{equation}
        m(z)=\int_\R\frac{\rho(\dd x)}{x-z}
        \label{eq m as trafo}
    \end{equation}
    for all $z\in\HH$. We call $\rho$ the self-consistent density of states corresponding to $q(\mb X)$.
\end{definition}
\begin{remark}
    By the global law, \cite[Proposition~2.17]{erdos2019local}, the self-consistent density of states is indeed the limiting spectral measure.
\end{remark}
 Note that since $A$ is Hermitian, the matrix $\widehat{A}=\frac{1}{2}(A+A^t)$ denotes the entrywise real part of $A$. This should not be confused with the algebraic definition of the real part of a matrix, $\Re R=\frac{1}{2}(R+R^*)$.
 Let $I\subset \R$ be an interval and $E\in I$. Due to the Stieltjes inversion formula, the limiting spectral measure and its Stieltjes transform are related by the equation
\begin{equation}
    \rho(E)=\lim_{\eta\searrow 0}\frac1\pi\Im m(E+\I\eta),
    \label{rho from m}
\end{equation}
whenever that limit exists for all $E\in I$ (see e.g. \cite[Equation~(1.4)]{Gesztesy2000}).
It was shown in \cite[Theorem~1.1 (3)]{Shlyakhtenko_2015} that $\supp(\rho)$ is a single compact interval on the real line. In particular, this means that $\rho$ has no internal edges. We use the following definition.
\begin{definition}[Edges of the limiting spectral measure]
    Let $\tau_+$ denote the position of the right edge of $\rho$ and let $\tau_-$ denote the position of the left edge of $\rho,$ i.e., we have
    \begin{equation}
        \supp(\rho)=[\tau_-,\tau_+]
    \end{equation}
    \label{def edges}
\end{definition}
We also introduce the notion of a regular edge.
\begin{definition}
    Let $\tau_0\in\{\tau_-,\tau_+\}$. The limiting spectral measure $\rho$ is said to have a regular edge at $\tau_0$ if $\rho(\dd E)=\rho(E)\dd E$ has a Lebesgue density (also denoted by $\rho$) in a neighborhood of $\tau_0$ and
    \begin{equation}
        \lim_{\substack{E\in\mr{supp}(\rho)\\
        E\to\tau_0}} \frac{\rho(E)}{\sqrt{|\tau_0-E|}}
    \end{equation}
    exists and does not equal 0.
\end{definition}
In other words, regular edges are those that show a square root decay of the density $\rho$. Our next two results concern the edge characterization of reducible and non-reducible polynomials.
\begin{prop}[Edges of non-reducible polynomials]
Let $q$ be a non-reducible quadratic polynomial and $\rho$ its associated limiting spectral measure defined in \eqref{eq SCDOS}. Then the measure $\rho$ has regular edges both at $\tau_-$ and at $\tau_+$.
\label{prop square root}
\end{prop}
Next, we consider shifted reducible polynomials of the form
\begin{equation}
    q_r(\mb X)=(v^*\mb X-\xi)(v^*\mb X-\xi)^*
    \label{shifted normalized}
\end{equation}
for some $\xi\in\R_{\geq0}$ and normalized $v\in\C^l$.
\begin{remark}
    Compared to the general case defined in \eqref{eq shifted red} we restrict here to $\alpha=1$ and $\beta=0$. As these constants only constitute a scaling and a shift respectively, the following result, Proposition~\ref{prop edge}, generalizes in a straightforward manner to all $\alpha\neq0$ and $\beta\in\R.$ For $\alpha<0$ the roles of the left and the right edge are reversed.
\end{remark}
We introduce the quantities $\sigma=\Vert \Re v\Vert,$ $\mu=\langle \Re v,\Im v\rangle$ and for $\sigma\in(0,1)$ we define $s>0$ with
\begin{equation}
    s^2:=
    \begin{cases}
        \left((\sigma^2+a^2_+(1-\sigma^2)+2a_+\mu)(\sigma^2+a_+\mu)\right)^{-1}\\
        \qquad+\left((\sigma^2+a^2_-(1-\sigma^2)+2a_-\mu)(\sigma^2+a_-\mu)\right)^{-1}&\text{if }\mu\neq0 ,\\
        \sigma^{-4}&\text{if }\mu=0 ,\\
    \end{cases}
    \label{eq s}
\end{equation}
where for $\mu\neq0$ we also set the constant $a_\pm\in\R$ to
\begin{equation}
    a_\pm=
        \frac{1-2\sigma^2}{2\mu}\pm\sqrt{\left(\frac{1-2\sigma^2}{2\mu}\right)^2+1}.
    \label{eq a_pm}
\end{equation}
The edges of the shifted reducible polynomial are then characterized as follows.
\begin{prop}[Edges of shifted reducible polynomials]
    Let $q_r$ be a shifted reducible polynomial as in \eqref{shifted normalized}
    for some $\xi\in\R_{\geq0}$ and normalized $v\in\C^l$.
    Then the limiting spectral measure $\rho$  has a regular edge at $\tau_+=\max\supp(\rho)$.
        Furthermore there is a $\kappa>0$ such that the behaviour of $\rho$  on $(\tau_-,\tau_-+\kappa)$ with $\tau_-=\min\supp(\rho)$ is given by 
        \begin{equation}
            \rho(E)\sim
            \begin{cases}
              (E-\tau_-)^{-\frac12}\quad\text{if }\xi=0\\
              (E-\tau_-)^{-\frac12}\quad\text{if }v\in\R^l \text{ and } \xi<2\\
              (E-\tau_-)^{-\frac14}\quad\text{if }v\in\R^l \text{ and } \xi=2\\
              (E-\tau_-)^{-\frac12}\quad\text{if }v\in\C^l\setminus e^{\I\phi}\R^l \text{ for all } \phi\in(-\pi,\pi] \text{ and } s\xi<2\\
              (E-\tau_-)^{-\frac13}\quad\text{if }v\in\C^l\setminus e^{\I\phi}\R^l \text{ for all } \phi\in(-\pi,\pi] \text{ and } s\xi=2.
            \end{cases}
        \end{equation}
    In all other cases the left edge is also regular, i.e. $\rho(E)\sim(E-\tau_-)^{1/2}.$
    \label{prop edge}
\end{prop}

\begin{remark}
    The reason behind this result is that the edge behaviour near the left edge follows from the distribution of small singular values of $v^*\mb X-\xi$. It follows that $\rho$ has singularities precisely if $\xi$ is in the asymptotic spectrum of $v^*\mb X$, with stronger singularities being observed if $\xi$ is in the bulk of the spectrum and weaker ones if $\xi$ is on the edge of the spectrum.
\end{remark}

Secondly we prove that $m\mb I_N$, where $\mb I_N$ is the identity matrix on $\C^{N\times N}$, well approximates the resolvent $\mb g=(q(\mb X)-z\mb I_N)^{-1}\in\C^{N\times N}$ of $q(\mb X)$ around any regular edge and away from the spectrum. More precisely we prove uniform convergence of the quadratic form of $\mb g$ on the set
\begin{equation}
    \mathbb D_{\gamma}^{\kappa_0}:=\{E+\I\eta\in\HH:\, |E-\tau_0|\leq\kappa_0,\,N^{-1+\gamma}\leq\eta\leq1\}
\end{equation}
for some $\kappa_0\sim1$ and for all $\gamma>0$, as well as on the set
\begin{equation}
    \mathbb G_{\gamma}^{C,\eta_0}:=\{E+\I\eta\in\HH:\, C^{-1}\leq\dist(E,\supp(\rho))\leq C,\,N^{-1+\gamma}\leq\eta\leq\eta_0\}
\end{equation}
for all $C,\gamma>0$ and some $\eta_0$ depending on $C$.
\begin{theorem}[Local law for regular edges of polynomials]
Let $q$ be a polynomial of the form \eqref{eq polynomial} and let $\rho$ have a regular edge at $\tau_0\in\{\tau_-,\tau_+\}$. There is a $\kappa_0>0$, depending only on the coefficients of $q$, such that for all $\veps,\gamma,D>0$ and $z=E+\I\eta\in\mathbb{D}_{\gamma}^{\kappa_0}$ the isotropic local law 
\begin{equation}
\Pb\left(|\langle \mb x, (\mb g-m)\mb y\rangle|>\Vert\mb x\Vert\Vert\mb y\Vert N^\veps\left(\frac{\Im m}{\sqrt{N\eta}}+\frac{1}{N\eta}\right)\right)\lesssim_{\veps,\gamma,D}N^{-D}
\label{eq ll iso}
\end{equation}
holds for all deterministic $\mb x,\mb y\in\C^{N}$.
Moreover, the averaged local law
\begin{equation}
\Pb\left(\left|\frac{1}{N}\Tr(\mb B(\mb g-m))\right|>\Vert\mb B\Vert\frac{N^\veps}{N\eta}\right)\lesssim_{\veps,\gamma,D}N^{-D}
\label{eq ll av}
\end{equation}
holds for all deterministic $\mb B\in\C^{N\times N}$. If additionally  $E\notin\supp(\rho),$ an improved averaged local law of the form
\begin{equation}
\Pb\left(\left|\frac{1}{N}\Tr(\mb B(\mb g-m))\right|>\Vert\mb B\Vert N^\veps\left(\frac{1}{N|z-\tau_0|}+\frac{1}{(N\eta)^2\sqrt{|z-\tau_0|}}\right)\right)
\lesssim_{\veps,\gamma,D}N^{-D}
\label{eq ll imp}
\end{equation}
holds.
 \label{thm LL}
\end{theorem}
Away from the spectrum, we will also prove the following form of an averaged local law
\begin{prop}[Local law away from the spectrum]
    Let $q$ be a polynomial of the form \eqref{eq polynomial}. For all $C>0$ there is an $\eta_0>0$, depending only on the coefficients of $q$, such that an averaged local law of the form
    \begin{equation}
        \Pb\left(\left|\frac{1}{N}\Tr(\mb B(\mb g-m))\right|>\Vert\mb B\Vert N^\veps\left(\frac{1}{N}+\frac{1}{(N\eta)^2}\right)\right)
        \lesssim_{\veps,\gamma,D,C}N^{-D}
    \label{eq ll away}
    \end{equation}
     holds true for all deterministic $\mb B\in\C^{N\times N}$, $\gamma,\veps,D>0$ and $z\in\mathbb G_{\gamma}^{C,\eta_0}$.
     \label{prop LL away}
\end{prop}
The local laws, Theorem~\ref{thm LL} and Proposition~\ref{prop LL away}, are proved in Section~\ref{proof LL}. They give us very precise control over the spectral properties of $q(\mb X)$ close to any regular edge as can be seen from the following corollaries,
\begin{corollary}[Edge eigenvector delocalization]
Let the assumptions of Theorem~\ref{thm LL} be satisfied and $\mb v$ be a normalized eigenvector of $q(\mb X)$ with eigenvalue $\lambda$. Then there is a $\kappa_0>0$, only depending on the coefficients of $q$, such that if $|\lambda-\tau_0|\leq\kappa_0$ then we have for all $\veps,D>0$ that
\begin{equation}
\sup_{\mb x\in\C^{N\times N},\Vert \mb x\Vert=1}\Pb\left(|\langle \mb x,\mb v\rangle|\geq N^{-\frac12+\veps}\right)\lesssim_{D,\veps}N^{-D}
\label{eq deloc}
\end{equation}
\label{cor delocalization}
\end{corollary}
\begin{corollary}[Eigenvalue rigidity]
    Denote the classical index of the eigenvalue close to energy $E\in\mr{supp}(\rho)$ by
    \begin{equation}
        k(E):=\left\lceil N\int_{-\infty}^E\rho(E')\dd E'\right\rceil,
    \end{equation}
    with $\lceil\cdot\rceil$ being the ceiling function. Let $\rho$ have a regular edge at $\tau_0$. All eigenvalues around $\tau_0$ are close to their classical position in the following sense. There is a $\kappa_0>0$, only depending on the coefficients of $q$, such that
    \begin{equation}
        \Pb\left(\mr{sup}_E|\lambda_{k(E)}-E|\geq\min\left\{\frac{N^\veps}{N|E-\tau_0|},\frac{N^\veps}{N^{\frac23}}\right\}\right)\lesssim_{D,\veps} N^{-D}.
    \end{equation}
    holds for all $\veps,D>0$ and $E\in\mr{supp}(\rho)$ with $|E-\tau_0|\leq\kappa_0$.
    \label{cor rigidity}
\end{corollary}
\begin{proof}[Proof of Theorem~\ref{theorem norm}]
        As the norm of any Hermitian matrix $H$ with non-decreasing eigenvalues $\lambda_i$, $i\in\llbracket n\rrbracket$, is given by $\Vert H\Vert=\max\{|\lambda_1|,|\lambda_n|\}$, Theorem~\ref{theorem norm} follows as a special case of Corollary~\ref{cor rigidity} in conjunction with Proposition~\ref{prop square root}, which states that all edges of non-reducible polynomials are regular.
\end{proof}

\section{Properties of the spectral density}

We first state two propositions which describe the behaviour of the Stieltjes transform in the upper half-plane close to the edges of the spectrum. The first proposition concerns non-reducible polynomials, whereas the second one covers shifted reducible polynomials.
Subsequently, we conclude Propositions~\ref{prop square root} and \ref{prop edge} from Propositions~\ref{prop m non-red} and \ref{prop m red} and state Corollary \ref{cor m reg}. It summarizes important properties of the Stieltjes transform close to any regular edge and is used to prove the local law, Theorem~\ref{thm LL}. Afterwards, we prove Propositions~\ref{prop m non-red} and \ref{prop m red} in the remainder of the section.

\begin{prop}[Stieltjes transform for non-reducible polynomials]
    Let $q$ be a non-reducible quadratic polynomial. Then we have the following behaviour of $m$ close to the edges.
    \begin{enumerate}
        \item There are $m_+<0$, $c_+>0$ and $u>0$ such that for all $z\in\HH$ with $|z-\tau_+|\leq u$ the Stieltjes transform of the limiting spectral density $m=m(z)$ satisfies
        \begin{equation}
            m-m_+=c_+\sqrt{z-\tau_+}+\mc O(|z-\tau_+|).
        \end{equation}
        \item There are $m_->0$, $c_->0$ and $u>0$ such that for all $z\in\HH$ with $|z-\tau_+|\leq u$ the Stieltjes transform of the limiting spectral density $m=m(z)$ satisfies
        \begin{equation}
            m-m_-=-c_-\sqrt{\tau_--z}+\mc O(|\tau_--z|).
        \end{equation}
    \end{enumerate}
    Here, $\sqrt{\cdot}$ denotes the square root function that maps the positive real axis to itself and with a branch cut along the negative real axis.
    \label{prop m non-red}
\end{prop}

\begin{prop}[Stieltjes transform for shifted reducible polynomials]
        Let $q$ be a shifted reducible polynomial of the form \eqref{shifted normalized}. Then we have the following behaviour of $m$ close to the edges.
    \begin{enumerate}
        \item There are $m_+<0$, $c_+>0$ and $u>0$ such that for all $z\in\HH$ with $|z-\tau_+|\leq u$ the function $m=m(z)$ satisfies
        \begin{equation}
            m-m_+=c_+\sqrt{z-\tau_+}+\mc O(|z-\tau_+|).
        \end{equation}
        \item There are $c_->0$ and $u>0$ such that for all $z\in\HH$ with $|z-\tau_-|\leq u$ the function $m=m(z)$ satisfies
        \begin{equation}
            m=
            \begin{cases}
                c_-(\tau_--z)^{-\frac12}+\mc O(1) &\text{if }\xi=0\\
                c_-(\tau_--z)^{-\frac12}+\mc O(1) &\text{if }v\in\R^l \text{ and } \xi<2\\
                c_-(\tau_--z)^{-\frac14}+\mc O(1) &\text{if }v\in\R^l \text{ and } \xi=2\\
                c_-(\tau_--z)^{-\frac12}+\mc O(1) &\text{if }v\in\C^l\setminus e^{\I\phi}\R^l \text{ for all } \phi\in(-\pi,\pi] \text{ and } s\xi<2\\
                c_-(\tau_--z)^{-\frac13}+\mc O(1) &\text{if }v\in\C^l\setminus e^{\I\phi}\R^l \text{ for all } \phi\in(-\pi,\pi] \text{ and } s\xi=2,
            \end{cases}
        \end{equation}
        where $s$ was defined in \eqref{eq s}.
        If none of the above conditions is satisfied, there is additionally an $m_->0$ such that for all $z\in\HH$ with $|z-\tau_-|\leq u$ we have
        \begin{equation}
            m-m_-=-c_-\sqrt{\tau_--z}+\mc O(|\tau_--z|).
        \end{equation}
        The function $\zeta\mapsto\zeta^p$ is chosen such that the positive real axis is mapped to itself and with a branch cut along the negative real axis for all $p\in\R\setminus\Z.$
    \end{enumerate}
    \label{prop m red}
\end{prop}
\begin{corollary}
    Let $q(\mb X)$ have a regular edge at $\tau_0$ and $m_0=m(\tau_0)$. There is a $u>0$ such that the function $m=m(z)$ satisfies
    \begin{equation}
        |z-\tau_0|\sim|m-m_0|^2
    \end{equation}
    and
    \begin{equation}
        \Im m\sim
        \begin{cases}
            \sqrt{|E-\tau_0|+\eta} &\text{if }E\in\supp(\rho)\\
            \frac{\eta}{\sqrt{|E-\tau_0|+\eta}}&\text{if }E\notin\supp(\rho)
        \end{cases}
    \end{equation}
    for all $z=E+\I\eta\in\HH$ with $|z-\tau_0|\leq u$.
    \label{cor m reg}
\end{corollary}
\begin{proof}[Proof of Proposition~\ref{prop square root} and Proposition~\ref{prop edge}]
    By \eqref{rho from m} we have $\rho(E)=\lim_{\eta\searrow0}\pi^{-1}m(E+\I\eta)$ on any interval $I\subset\R$ on which the limit exists for all $E\in I$. For the cases in Proposition~\ref{prop m non-red} and Proposition~\ref{prop m red} where $m-m_0$ shows a square root behaviour around an edge $\tau_0$, we take the limit on $(\tau_0-u,\tau_0+u)$ to prove that the edge is regular. For the cases in Proposition~\ref{prop m red} where $m$ diverges at $\tau_-$, we take the limit on $(\tau_--u,\tau_-+u)\setminus\{\tau_-\}$ to obtain the respective asymptotic behaviour.
\end{proof}

Let $\mu_i\in\R$ denote the eigenvalues of $A$, $\widehat{\mu_i}\in\R$ the eigenvalues of $\widehat{A}$ and let $ w_i\in\R^l$ be an orthonormal set of eigenvectors of $\widehat{A}$ corresponding to the $\widehat{\mu_i}$. By \eqref{eq gamma} the function $\gamma$ is defined in terms of these quantities as
\begin{equation}
    \gamma(m)=-\sum_{i=1}^l\frac{\mu_i}{1+m\mu_i}+\sum_{i=1}^l\frac{|\langle  w_i, b\rangle|^2(1+m\widehat\mu_i)}{(1+2m\widehat\mu_i)^2}m 
    -c.
\end{equation}
From now on we consider $\gamma$ to be defined on its maximal domain, $\C\setminus\IS(\gamma)$, where 
\begin{equation}
    \IS(\gamma):=\{-\mu_i^{-1}\in\R: \,\mu_i\neq0\}
        \cup\{-(2\widehat{\mu}_i)^{-1}\in\R: \,\widehat{\mu}_i\neq0\text{ and }\langle w_i, b\rangle\neq0\}
\end{equation}
denotes the poles of $\gamma$. We also require the function $h$, which we define as follows.
\begin{definition}
    Let $h: \C\setminus\IS(h)$ be given by
    \begin{equation}
        h(m):=\frac{1}{m^2}-\gamma'(m)=\frac{1}{m^2}-\sum_{i=1}^l \frac{\mu_i^2}{(1+m\mu_i)^2}-\sum_{i=1}^l\frac{|\langle  w_i, b\rangle|^2}{(1+2m\widehat\mu_i)^3}.
        \label{function h}
    \end{equation}
    Here, $\IS(h):=\IS(\gamma)\cup\{0\}$ denotes the set of poles of $h$ and $\gamma'(m)$ is the derivative of $\gamma$ with respect to $m$.
    \label{def h}
\end{definition}
Proposition~\ref{prop SCE} asserts that $m$ is a Stieltjes transform of the measure $\rho,$ which has real and compact support. Thus $m$ is analytic and has positive imaginary part on $\HH$. We take the derivative of \eqref{scalar equation} with respect to $z$ on $\HH$ to find
\begin{equation}
    m'=\frac{1}{{h}(m)}.
\end{equation}
As $m$ is a Stieltjes transform of $\rho$, its analyticity extends to $\C\setminus\supp(\rho)$ and the above equation
holds on $\overline{\HH}\setminus\supp(\rho),$ where $\overline{\HH}=\HH\cup\R$.
Next, we define
\begin{align}
        m_-^*&=
        \begin{cases}
            \min((0,\infty)\cap\IS(h))&\text{if }(0,\infty)\cap\IS(h)\neq\emptyset\\
            \infty&\text{otherwise},
        \end{cases}\\
        m_+^*&=
        \begin{cases}
            \max((-\infty,0)\cap\IS(h))&\text{if }(-\infty,0)\cap\IS(h)\neq\emptyset\\
            -\infty&\text{otherwise.}
        \end{cases}
        \label{eq largest pole}
\end{align}
In other words, $(m_+^*,0)$ and $(0,m_-^*)$ are the maximal intervals to the left and the right of the origin, where $h$ is continuous.  The following lemmata describe the existence and characterization of roots of $h$ both for the shifted reducible and the non-reducible case on $(m_+^*,0)$ and $(0,m_-^*)$.

\begin{lemma}
    \label{lemma h non-red}
    Let $q$ be a non-reducible quadratic polynomial. Then we have the following. 
    \begin{enumerate}
        \item\label{lemma h non-red part 1} The function $h$ has a unique root $m_-$ in $(0,m_-^*)$ and $h$ has a unique root $m_+$ in $(m_+^*,0)$. Both of them are of first order. More precisely, they satisfy
    \begin{equation}
        \pm h'(m_\pm)>0.
    \end{equation}
    \item\label{lemma h non-red part 2} The positions of the edges $\tau_\pm$, defined in Definition \ref{def edges}, are given in terms of $m_\pm$ by
    \begin{equation}
        \tau_\pm=-m_\pm^{-1}-\gamma(m_\pm).
        \label{tau non-red}
    \end{equation}
    \end{enumerate}
\end{lemma}

\begin{lemma}
    \label{lemma h red}
    Let $h$ be the function introduced in Definition~\ref{def h} for a reducible polynomial $q=q_r$ of the form \eqref{shifted normalized} with normalized $v\in\C^l$ and $\xi\geq0$. Then the following holds true.
    \begin{enumerate}
        \item\label{lemma h red part 1} The function $h$ has a unique root $m_+$ in $(m_+^*,0)$, which is of first order and satisfies
        \begin{equation}
            h'(m_+)>0.
        \end{equation}
        \item\label{lemma h red part 2} The function $h$ is continuous on $(0,\infty)$. It has no root in $(0, \infty)$ if and only if one of the following hold true:
        \begin{enumerate}
            \item $\xi=0$;
            \item or $v\in\R^l$ and $\xi\leq2$;
            \item or $v\in\C^l\setminus e^{\I\phi}\R^l$ for all $\phi\in(-\pi,\pi]$ and $s\xi\leq2$.
        \end{enumerate}
        If none of the above conditions are satisfied, $h$ does have a unique root $m_-\in(0, \infty)$, which is of first order and satisfies
        \begin{equation}
            h'(m_-)<0.
        \end{equation}
        If $h$ does not have any roots in $(0,\infty)$, the following asymptotic behaviour is observed for $m\to \infty$:
        \begin{equation}
            h(m)\sim
            \begin{cases}
              m^{-3}\quad\text{if }\xi=0\\
              m^{-3}\quad\text{if }v\in\R^l \text{ and } \xi<2\\
              m^{-5}\quad\text{if }v\in\R^l \text{ and } \xi=2\\
              m^{-3}\quad\text{if }v\in\C^l\setminus e^{\I\phi}\R^l \text{ for all } \phi\in(-\pi,\pi] \text{ and } s\xi<2\\
              m^{-4}\quad\text{if }v\in\C^l\setminus e^{\I\phi}\R^l \text{ for all } \phi\in(-\pi,\pi] \text{ and } s\xi=2.
            \end{cases}
        \end{equation}
        Here $s$ was defined in \eqref{eq s}.
        \item\label{lemma h red part 3} The positions of the edges $\tau_\pm$, defined in Definition \ref{def edges}, are given by
    \begin{equation}
        \tau_+=-m_+^{-1}-\gamma(m_+)
        \label{tau_+ red}
    \end{equation}
    and
    \begin{equation}
        \tau_-=
        \begin{cases}
            -m_-^{-1}-\gamma(m_-) &\text{if }m_->0\text{ with }h(m_-)=0\text{ exists;}\\
            0 &\text{otherwise.}
        \end{cases}
        \label{tau_- red}
    \end{equation}
    \end{enumerate}
\end{lemma}
The proof of the lemmata is deferred until after the proof of Propositions~\ref{prop m non-red} and \ref{prop m red}.

\begin{proof}[Proof of Propositions~\ref{prop m non-red} and \ref{prop m red}]
In both the non-reducible and the shifted reducible cases the function $h$ has a unique root $m_+$ in $(m_+^*,0)$. 
For all $m\in(m_+,0)$ 
the derivative of the inverse $z=z(m)$ of $m=m(z)$ with respect to $m$ is given by
\begin{equation}
    \frac{\dd}{\dd m}z=h(m).
\end{equation}
Then, by Lemma~\ref{lemma h non-red} and Lemma~\ref{lemma h red}, there is an $r\sim1$ and a $c_+>0$, only depending on the coefficients of $q$, such that for all $m\in (m_+,m_++r)$ we have
\begin{equation}
    \frac{\dd}{\dd m}z=h(m)=2c_+^2(m-m_+)(1+\mc O(m-m_+))
\end{equation}
and thus
\begin{equation}
    z-\tau_+=c_+^2(m-m_+)^{2}(1+\mc O(m-m_+)).
\end{equation}
Taking the square root of the above equation, we obtain
\begin{equation}
    m-m_+=c_+\sqrt{z-\tau_+}+\mc O(|z-\tau_+|)
    \label{eq m to the right}
\end{equation}
for all $z\in(\tau_+,\tau_++u)$ for some $u>0.$ As $m$ is holomorphic on $\HH\cup(\tau_+,\infty)$ and has positive imaginary part on $\HH$, the relation \eqref{eq m to the right} also holds for all $z\in\HH$ with $|z-\tau_+|\leq u$ for some $u>0$.

In all cases, where $h$ has a root $m_-$ in $(0,m_-^*)$, we find by the same argument that
\begin{equation}
    m-m_-=-c_-\sqrt{\tau_--z}+\mc O(|\tau_--z|)
\end{equation}
for all $z\in\HH$ with $|z-\tau_+|\leq u$ for some $u,c_->0.$

Finally, in the cases where $h$ does not have a root in $(0,\infty)$, it is continuous on the whole interval and we have that for all $m\in(0,\infty)$
the derivative of $z$ with respect to $m$ is given by
\begin{equation}
    \frac{\dd}{\dd m}z=h(m).
\end{equation}
Then, by Lemma~\ref{lemma h red}, there is an $R\sim1$ and a $c_0>0$, only depending on the coefficients of $q$, such that for all $m\in (R,\infty)$ we have
\begin{equation}
    \frac{\dd}{\dd m}z=h(m)=c_0m^{-p}(1+\mc O(m^{-1}))
    \quad
    \text{and thus}
    \quad
    z=-\frac{c_0}{p-1}m^{-(p-1)}(1+\mc O(m^{-1})),
\end{equation}
with $p\in\{3,4,5\}.$ Inverting the relation, we obtain
\begin{equation}
    m=c_-(-z)^{-\frac{1}{p-1}}+\mc O(1)
    \label{eq m to the left}
\end{equation}
for all $z\in(-u,0)$ and some $u,c_->0.$ Again, as $m$ is holomorphic on $\HH\cup(-\infty,0)$ and has positive imaginary part on $\HH$, the relation \eqref{eq m to the left} also holds for all $z\in\HH$ with $|z|\leq u$ for some $u>0$.

\end{proof}

\begin{proof}[Proof of Lemma~\ref{lemma h non-red} and Lemma~\ref{lemma h red}]
We prove both lemmata simultaneously and remark, where it becomes necessary to distinguish the different polynomials $q.$ We first investigate the existence and order of roots of $h$.
Let $h_q$ denote the rational function \eqref{function h} with the dependence on $q$ being made explicit. We have $h_q(-m)=h_{-q}(m)$. Therefore any root of $h_q$ on $(0,\infty)$ corresponds to a root of $h_{-q}$ on $(-\infty,0).$
For simplicity, the proofs are thus only formulated for $m\in(-\infty,0)$. To obtain the corresponding result for $m\in(0,\infty)$ for the function $h_q$, we consider the function $h_{-q}$ on $(-\infty,0)$ instead.

We now prove for all polynomials $q$ as in \eqref{eq polynomial} that $h'(m)>0$ for all $m\in(m_+^*,0)$ with $h(m)\geq0$ and thus $h$ can only have first-order roots on the interval.

Note that on $(m_+^*,0)$ the inequality $h\geq0$ is equivalent to 
\begin{equation}
\sum_{i=1}^l \frac{(m\mu_i)^2}{(1+m\mu_i)^2}+\sum_{\substack{i=1,\\
\langle  w_i, b\rangle\neq0}}^l(m|\langle  w_i, b\rangle|)^2\frac{1}{(1+2m\widehat\mu_i)^3}\leq1
\label{ineq h positive}
\end{equation}
and that on $(m_+^*,0)$ the inequality $h'>0$ is equivalent to
\begin{equation}
\sum_{i=1}^l \frac{(m\mu_i)^3}{(1+m\mu_i)^3}+\sum_{\substack{i=1,\\
\langle  w_i, b\rangle\neq0}}^l(m|\langle  w_i, b\rangle|)^2\frac{3(m\widehat\mu_i)}{(1+2m\widehat\mu_i)^4}<1.
\label{ineq h' positive}
\end{equation}
For $m\in(m_+^*,0)$ we have also that $m\mu_i>-1$ for all $i\in\llbracket l\rrbracket$ and $m\widehat\mu_i>-\frac{1}{2}$ for all $i\in\llbracket l\rrbracket$ such that $\langle  w_i, b\rangle\neq0$ by the definition of $m_+^*$ in \eqref{eq largest pole}. Furthermore, Lemma~\ref{lemma hat A} ensures that $\max_{i\in\llbracket l\rrbracket}\{m\mu_i\}\geq\max_{i\in\llbracket l\rrbracket}\{m\widehat\mu_i\}$.
Thus the desired implication, $h\geq0$ implies $h'>0$,  follows directly from the following lemma by identifying \eqref{condition lemma} and \eqref{implication lemma} right below with \eqref{ineq h positive} and \eqref{ineq h' positive} respectively.
\begin{lemma}
 Let $n,k\in\N$ with $n\geq k$. Let $y_i$, $i\in\llbracket n\rrbracket$  and let $\widehat y_j$, $j\in\llbracket k\rrbracket$ be collections of real numbers, both of them sorted in non-increasing order, $y_1\geq y_2\geq\ldots\geq y_n$ and $\widehat{y}_1\geq\widehat{y}_2\geq\ldots\geq\widehat{y}_k$. Suppose they satisfy
\begin{enumerate}
    \item $y_1\geq\widehat y_1$;
    \item $y_n>-1$ and $\widehat{y}_k>-\frac12$.
\end{enumerate}
Furthermore let $c_j>0$, $j\in\llbracket k\rrbracket$. Then the inequality
\begin{equation}
\sum_{i=1}^n\frac{y_i^2}{(y_i+1)^2}+\sum_{j=1}^k c_j^2\frac{1}{(2\widehat y_j+1)^{3}}\leq1
\label{condition lemma}
\end{equation}
implies
\begin{equation}
\sum_{i=1}^n\frac{y_i^3}{(y_i+1)^3}+3\sum_{j=1}^k c_j^2\frac{\widehat y_j}{(2\widehat y_j+1)^{4}}\leq1-\nu
\label{implication lemma}
\end{equation}
for some $\nu>0$, depending only on $y_1.$
\label{lemma quad stab}
\end{lemma}
The proof of the lemma is deferred to the end of the section.

To characterize the existence of roots of $h$ we introduce the following notions. If $A$ is a rank one matrix with real entries, then $\widehat{A}=A$ and $\widehat A$ is a rank one matrix as well. We will denote its  non-zero eigenvalue by $\mu$ and the corresponding normalized eigenvector by $ w$. If, on the other hand, $A$ is a rank one matrix with $A\in\C^{l\times l}\setminus\R^{l\times l}$, then $\widehat{A}$ is a rank two matrix by Lemma~\ref{lemma hat A} in the appendix. We denote its non-zero eigenvalues by $\mu_\pm$ and the corresponding normalized eigenvectors by $ w_\pm.$ 

\begin{lemma}
Let $m_+^*$ be defined as in \eqref{eq largest pole}. The function $h$ has no root in $(m_+^*,0)$ if and only if $A$ is a negative semi-definite rank one matrix, $ b\in\Image(\widehat{A})$ and either
\begin{enumerate}
\item $A\in\R^{l\times l}$ and $\Vert  b\Vert\leq4\Vert A\Vert$;
\item or $A\in\C^{l\times l}\setminus\R^{l\times l}$ and 
\begin{equation}
\frac{|\langle b, w_+\rangle|^2}{r_+^3}+\frac{|\langle b, w_-\rangle|^2}{r_-^3}\leq(4\Vert A\Vert)^2,
\end{equation}
with $r_\pm:=-\Vert A\Vert^{-1}\mu_\pm.$
\end{enumerate}
If $h$ has no root in $(m_+^*,0)$, then we have $m_+^*=-\infty$ and the asymptotic behaviour of $h$ for $m\to-\infty$ is given by
\begin{equation}
h(m)\sim
\begin{cases}
  \frac{1}{(-m)^3}  & A\in\R^{l\times l} \text{and } \Vert  b\Vert<4\Vert A\Vert\\ 
  \frac{1}{(-m)^5}  & A\in\R^{l\times l} \text{and } \Vert  b\Vert=4\Vert A\Vert\\ 
  \frac{1}{(-m)^3} & A\in\C^{l\times l}\setminus\R^{l\times l}  \text{ and } \frac{|\langle b, w_+\rangle|^2}{|r_+|^3}+\frac{|\langle b, w_-\rangle|^2}{|r_-|^3}<(4\Vert A\Vert)^2\\
  \frac{1}{(-m)^4} & A\in\C^{l\times l}\setminus\R^{l\times l}  \text{ and } \frac{|\langle b, w_+\rangle|^2}{|r_+|^3}+\frac{|\langle b, w_-\rangle|^2}{|r_-|^3}=(4\Vert A\Vert)^2.
\end{cases}
\label{h asymptotics}
\end{equation}
\label{lemma root conditions}
\end{lemma}
The proof of the lemma is also deferred to the end of the section.

Note that $h$ can have at most one root on $(m_+^*,0)$. To see this, recall that $h(m)\geq0$ implies $h'(m)>0$. Therefore, if $m_+$ is a root of $h$ on $(m_+^*,0)$, we know that $h$ is strictly monotonously increasing on $(m_+,0)$. Thus, any root of $h$ on $(m_+^*,0)$ is also the largest one on the interval; therefore it must be the unique root.

Now let $q$ be a non-reducible polynomial. If $\rank A\geq2$, then $h$ has a root in $(-\infty,0)$ by Lemma~\ref{lemma root conditions}. If $\rank A=1$, we can express $A$ as $A=\alpha w w^*$ for some $\alpha\in\R\setminus\{0\}$ and normalized $ w\in\C^l$. The vector $ b$ cannot be of the form $ b=\alpha_1\Re  w+\alpha_2\Im  w$ for any $\alpha_1,\alpha_2\in\R$ as $q$ would be reducible otherwise. Thus $ b\notin\Image A$ and again $h$ has a root in $(-\infty,0)$ by Lemma~\ref{lemma root conditions}, completing the proof of Statement~\ref{lemma h non-red part 1} of Lemma~\ref{lemma h non-red}.

Next, consider a shifted reducible polynomial $q=q_r$ for $q_r$ as in \eqref{shifted normalized}. If we write $q$ in terms of \eqref{eq polynomial} we have $A= v v^*\geq0$. Applying Lemma~\ref{lemma h non-red} we obtain that $h$ has a root in $(-\infty,0)$ and thereby prove Statement~\ref{lemma h red part 1} of Lemma~\ref{lemma h red}.

Finally, consider  $q=-q_r$ for $q_r$ as in \eqref{shifted normalized}. Then the coefficients $A$ and $ b$ associated with $q$ are given by $A=- v v^*$ and $ b=2\xi\Re  v.$ The matrix $A$ is a negative semi-definite rank one matrix. For $\mu=\langle \Re  v,\Im  v\rangle\neq0$ the non-zero eigenvalues of $\widehat{A}=-\frac12( v^* v+ v v^*)=-(\Re  v)^*\Re  v-(\Im  v)^*\Im  v$ are given by 
\begin{equation}
    \mu_\pm=\sigma^2+\mu a_\pm,
\end{equation}
where $\sigma=\Vert \Re  v\Vert$ and $a_\pm$ was defined in \eqref{eq a_pm}.  The corresponding normalized eigenvectors are
\begin{equation}
     w_\pm=\frac{1}{\sigma^2+a_\pm(1-\sigma^2)+2a_\pm\mu}\left(\Re  v+a_\pm\Im  v\right).
\end{equation}
For $\mu=0$ the eigenvector-eigenvalue pairs are given by
\begin{equation}
    ( w_+,\mu_+)=\left(\frac{\Im  v}{1-\sigma^2},1-\sigma^2\right),\quad
    ( w_-,\mu_-)=\left(\frac{\Re  v}{\sigma^2},\sigma^2\right).
\end{equation}
Now recall that $h_q(m)=h_{-q}(-m)=h_{q_r}(-m)$ and Statement~\ref{lemma h red part 2} of Lemma~\ref{lemma h red} now follows from applying Lemma~\ref{lemma root conditions} to $h_q$.

It remains to prove the relations between $m_\pm$ and $\tau_{\pm}$ stated in \eqref{tau non-red}, \eqref{tau_+ red} as well as \eqref{tau_- red}.
Let $q$ be either a non-reducible polynomial or a reducible polynomial of the form \eqref{shifted normalized}. For $\widetilde m\in (m_+^*,0)$ define
\begin{equation}
    z(\widetilde m):=-\frac{1}{\widetilde m}-\gamma(\widetilde m).
    \label{def z function}
\end{equation}
The function $z$ is real analytic on $(m_+^*,0)$ with derivative $z'(\widetilde m)=h(\widetilde m)$. By Lemma~\ref{lemma h non-red}, Part~\ref{lemma h non-red part 1} and Lemma~\ref{lemma h red}, Part~\ref{lemma h red part 1}, the function $h$ has a unique root $m_+$ in $(m_+^*,0)$ and $h(\widetilde m)>0$ for all $\widetilde m\in(m_+,0)$. Therefore,the function $z$ is invertible on $(m_+,0)$ and its inverse function $\widetilde{m}:(z_+,\infty)\to(m_+,0)$, with $z_+:=z(m_+)$, is also real analytic and monotonically increasing. Since $z'(m_+)=h(m_+)=0$, the analyticity of $\widetilde m$ cannot be extended onto any neighbourhood of $z_+$.

By \eqref{def z function}, the function $\widetilde m$ satisfies \eqref{scalar equation} and has the asymptotic behaviour $\widetilde m=-z^{-1}+\mc O(z^2)$ for $z\to\infty$. At the same time, $m$, the function uniquely defined by Proposition~\ref{prop SCE}, is a Stieltjes transform of a probability measure with support $[\tau_-,\tau_+]$. As such, it can be analytically extended to $\C\setminus[\tau_-,\tau_+]$ but not to a neighbourhood of $\tau_+$ and the extension is real-valued on $\R\setminus[\tau_-,\tau_+]$. As an extension it also satisfies \eqref{scalar equation} on $\C\setminus[\tau_-,\tau_+]$ and has the asymptotic behaviour $m=-z^{-1}+\mc O(z^2)$. In particular, the restriction of the extension to $\R\setminus[\tau_-,\tau_+]$, called $m_{\R}$, is a real analytic function that also satisfies \eqref{scalar equation}, has the asymptotic behaviour $m_{\R}=-z^{-1}+\mc O(z^2)$ and cannot be analytically extended to a neighbourhood of $\tau_+$. As satisfying \eqref{scalar equation} and having the asymptotic behaviour $m=-z^{-1}+\mc O(z^2)$ uniquely define an analytic function on a neighbourhood of $\infty$, we have $\widetilde m=m_{\R}$ on $(C,\infty)$ for some $C>0$. If we assume that $\tau_+>z_+,$ then $\widetilde m$ would be an analytical continuation of $m_{\R}$ to some neighbourhood of $\tau_+$, which is a contradiction and vice versa. Therefore we must have $z_+=\tau_+.$ In particular, we have
\begin{equation}
    \tau_+=z(m_+)=-\frac{1}{m_+}-\gamma(m_+).
\end{equation}
By an analogous argument to the left of the spectrum we find 
\begin{equation}
    \tau_-=z(m_-)=-\frac{1}{m_-}-\gamma(m_-)
\end{equation}
if $m_-$ exists. Now, let $q=q_r$ be a shifted reducible polynomial of the form \eqref{shifted normalized} such that $h$ has no root on $(0,\infty)$. Then the function $z(\widetilde m)$ defined on $(0,\infty)$ by \eqref{def z function} is a monotonously increasing analytic function on the entirety of its domain. It is therefore invertible and its inverse function, $\widetilde m$, is analytic on $(-\infty, z_\infty)$ with $z_\infty:=\lim_{x\to\infty}z(x)$ but cannot be analytically extended to a neighbourhood of $z_\infty$. Along the lines of the above argument, $m_{\R}$ would be an analytical continuation of $\widetilde m$ in a neighbourhood of $z_\infty$ if $\tau_->z_\infty$ and $\widetilde m$ would be an analytical continuation of $m_{\R}$ in a neighbourhood of $\tau_-$ if $z_\infty>\tau_-$. Thus we have $\tau_-=z_\infty$ and in particular
\begin{equation}
    \tau_-=\lim_{x\to\infty}-\frac1x-\gamma(x).
    \label{limit tau_-}
\end{equation}
For reducible $q_r$ as in \eqref{shifted normalized}, we have $A=vv^*$, $b=-\xi(v+\bar v)$ and $c=\xi^2$. Taking the limit \eqref{limit tau_-} with these parameters we obtain $\tau_-=0$.
\end{proof}

\begin{proof}[Proof of Lemma~\ref{lemma quad stab}]
Let $y \in(-1,\infty)$ and $\widehat y\in(-\frac12,\infty)$.
We define
\begin{equation}
h_1(y)=\frac{y}{y+1}
\quad
\mr{and}
\quad
h_2(\widehat y)=\frac{1}{2\widehat y+1}.
\end{equation}
Using this notation \eqref{condition lemma} can be expressed as
\begin{equation}
g(\mathbf{y},\widehat{\mb y},\mathbf{c}):=\sum_{i=1}^n h_1(y_i)^2+\sum_{j=1}^k c_j^2h^3_2(\widehat{y}_j)\leq 1
\end{equation}
and \eqref{implication lemma} as
\begin{equation}
f(\mathbf{y},\widehat{\mb y},\mathbf{c}):=\sum_{i=1}^n h_1(y_i)^3+3\sum_{j=1}^k c_j^2\widehat{y}_jh^4_2(\widehat{y}_j)\leq1-\nu
\end{equation}
To prove the lemma it is then sufficient to show that $g\leq1$ implies $f\leq1-\nu$.

First assume $g=0$. Since all summands of $g$ are non-negative, it follows that $y_i=0$ for all $i\in\llbracket n\rrbracket$ and $c_j=0$ for all $j\in\llbracket k\rrbracket$ and so $f$ vanishes as well. Thus we assume $g\neq0$ from now on and we will prove that $g\leq1$ implies $f<g$.

We will prove the cases $\widehat{y}_1\leq \frac12$, $\frac12<\widehat{y}_1<1$ and $\widehat{y}_1\geq1$ separately and start with $\widehat{y}_1\leq\frac12$. Then $\widehat y_j\leq\frac12$ for all $j$. Note that $\widehat y_j h_2^4(\widehat y_j)\leq \widehat y_1 h_2(\widehat y_1) h_2^3(\widehat y_j)$ and $h_1^3(y_i)\leq h_1(y_1)h_1^2(y_i)$ hold for all $i$ and $j$ since $\widehat{y}\mapsto \widehat{y}h_2(\widehat{y})$ and $h_1$ are monotonously increasing and the $\widehat{y}_j$ and $y_i$ are sorted in non-increasing order. 
Thus 
\begin{equation}
\begin{split}
      f(\mathbf{y},\widehat{\mb y},\mathbf{c})
   &\leq h_1(y_1)\sum_{i=1}^n h_1^2(y_i)+3\widehat{y}_1h_2(\widehat{y}_1)\sum_{j=1}^kc_j^2h^3_2(\widehat{y}_j)\\
   &\leq h_1(y_1)\sum_{i=1}^n h_1^2(y_i)+\frac34\sum_{j=1}^kc_j^2h^3_2(\widehat{y}_j)
   \leq\max\left\{\frac{y_1}{1+y_1},\frac34\right\}g(\mathbf{y},\widehat{\mb y},\mathbf{c}).
   \label{ineq y<1/2}
\end{split}
\end{equation}
The second inequality holds since $3\widehat{y}_1h_2(\widehat{y}_1)\leq\frac34$ due to $\widehat{y}_1\leq\frac12$ and the third one because all summands in both sums are non-negative. The relation $f\leq1-\nu$ for $g\leq1$ then follows in this regime.

For $\frac12<\widehat y_1<1$ we estimate every summand in $f$ but the $y_1$ term by the corresponding term in $g$ and we find
\begin{equation}
f(\mathbf{y},\widehat{\mb y},\mathbf{c})\leq g(\mathbf{y},\widehat{\mb y},\mathbf{c})-h_1(y_1)^2+h_1(y_1)^3.
\end{equation}
This upper bound is valid since $\widehat y_j\leq\widehat y_1<1$ for all $j.$ Thus for $g\leq1$ we find
\begin{equation}
    f(\mathbf{y},\widehat{\mb y},\mathbf{c})\leq g(\mathbf{y},\widehat{\mb y},\mathbf{c})-\left(\frac{y_1}{y_1+1}\right)^2\frac{1}{1+y_1}<1-\frac{2}{27},
    \label{ineq y<1}
\end{equation}
where in the last step we used $y_1\geq\widehat y_1>\frac12.$

Now let $\widehat y_1\geq 1$.
Let $k_0$ be the largest integer such that $\widehat{y}_{k_0}\geq1.$
Then
\begin{equation}
f_0(\mathbf{y},\widehat{\mb y},\mathbf{c})
:=\sum_{i=2}^{n}h_1^3(y_i)+\sum_{j=k_0+1}^k3c_j^2\widehat{y}_jh^4_2(\widehat{y}_j)
\leq\sum_{i=2}^{n}h_1^2(y_i)+\sum_{j=k_0+1}^kc_j^2h^3_2(\widehat{y}_j)
=:g_0(\mathbf{y},\widehat{\mb y},\mathbf{c})
\end{equation}
is concluded by again estimating each term individually (either of the sums might be empty). Note that $g_0\geq0$. Therefore we only need to prove that
\begin{equation}
g_1(\mathbf{y},\widehat{\mb y},\mathbf{c})
:=(g-g_0)(\mathbf{y},\widehat{\mb y},\mathbf{c})=h_1^2(y_1)+\sum_{j=1}^{k_0}c_j^2h^3_2(\widehat{y}_j)\leq 1
\end{equation}
implies
\begin{equation}
f_1(\mathbf{y},\widehat{\mb y},\mathbf{c})
:=(f-f_0)(\mathbf{y},\widehat{\mb y},\mathbf{c})
=h_1^3(y_1)+\sum_{j=1}^{k_0}3c_j^2\widehat y_jh^4_2(\widehat{y}_j)
\leq g_1(\mathbf{y},\widehat{\mb y},\mathbf{c})-\nu
\end{equation}
for some $\nu>0$, only depending on $y_1$, to conclude the proof of the lemma.

For $\widehat y\in[1,\infty)$ the inequality $h^3_2(\widehat{y})\leq3\widehat yh^4_2(\widehat{y})$ is satisfied and consequently
\begin{equation}
    \sum_{j=1}^{k_0}h^3_2(\widehat{y}_j)\leq\sum_{j=1}^{k_0}3\widehat y_jh^4_2(\widehat{y}_j)
\end{equation} holds. Therefore for any fixed $(\mathbf{y},\widehat{\mb y},\mb{c})$ and $r\in\R_+$ the difference  $(g_1-f_1)(\mb{y},\widehat{\mb y},r\mb{c})$ decreases monotonically in $r$. Hence we only need to prove $(g_1-f_1)(\mb{y},\widehat{\mb y},r^*\mb{c})\geq\nu$ for
\begin{equation}
    r^*:=\sup\{r\geq0|\,g_1(\mb{y},\widehat{\mb y},r\mb{c})\leq1\}.
\end{equation}
On the other hand $g_1(\mb{y},\widehat{\mb y},r\mb{c})$ increases monotonically in $r$ and therefore $g_1(\mb{y},\widehat{\mb y},r^*\mb{c})=1$. It is thus sufficient to prove that $g_1=1$ implies $f_1\leq1-\nu$.
Hence we assume $g_1=1$. Then 
\begin{equation}
\sum_{j=1}^{k_0}c_j^2h^3_2(\widehat{y}_j)=1-h_1^2(y_1)
\end{equation}
and
\begin{equation}
\begin{split}
    f_1(\mathbf{y},\widehat{\mb y},\mathbf{c})
    &=h_1^3(y_1)+\sum_{j=1}^{k_0}3c_j^2\widehat y_jh^4_2(\widehat{y}_j)
    \leq h_1^3(y_1)+\max_{j\in\llbracket k_0\rrbracket}\{3\widehat y_jh_2(\widehat{y}_j)\}\sum_{j=1}^{k_0}c_j^2h^3_2(\widehat{y}_j)\\
    &\leq h_1^3(y_1)+3\widehat y_1h_2(\widehat{y}_1)(1-h_1^2(y_1))
    \leq h_1^3(y_1)+3 y_1h_2(y_1)(1-h_1^2(y_1))\\
    &=\frac{1}{(1+y_1)^3}(y_1^3+3y_1^2+3y_1)=1-\frac{1}{(1+y_1)^3}.
\end{split}
\end{equation}
The inequalities holds because $yh_2(y)$ increases monotonically in $y$ and $y_1\geq\widehat y_1\geq \widehat y_j$ for all $j$. Combining this with \eqref{ineq y<1/2} and \eqref{ineq y<1} concludes the proof of the lemma.
\end{proof}

\begin{proof}[Proof of Lemma~\ref{lemma root conditions}]
Note that $h(m)>0$ for $m$ close to 0, since the $m^{-2}$ term is dominating $h$ from \eqref{function h} in this regime. First, we consider the situation, where $h$ has a pole in $(-\infty,0)$. In other words, we have $m_+^*>-\infty$, where $m_+^*$ was defined in \eqref{eq largest pole}.
All poles of $h$ on $(-\infty,0)$ are either of order two or of order three so the pole at $m_+^*$ will be as well. If the pole is of order two, then the behaviour around $m^*$ is given by $\lim_{m\to m_+^*}h(m)=-\infty$ and if it is of order three $h$ behaves like $\lim_{m\searrow m_+^*}h(m)=-\infty$ and $\lim_{m\nearrow m_+^*}h(m)=\infty$. Since $h$ is continuous outside of its poles, the existence of a pole thus implies the existence of a root in $(m_+^*,0)$. 
Now let $h$ have no poles in $(-\infty,0),$ i.e. $m_+^*=-\infty.$ Then $h$ is analytic on $(-\infty,0)$ and $h$ has a root in $(-\infty,0)$ if and only if $h(m)$ is negative for some $m\in(-\infty,0).$

We separate multiple cases.
\begin{case}[$\rank(A)\geq2$]
    Either $h$ has a pole in $(-\infty,0)$ or $ b^t(1+2m\widehat{A})^{-3} b\geq0$ for all $m\in(-\infty,0),$ thus
    \begin{equation}
        h(m)\leq \frac{1}{m^2}-\Tr A^2(1+mA)^{-2}=-\frac{\rank(A)-1}{m^2}+\mathcal{O}(m^{-3})<0
    \end{equation}
    for sufficiently large $-m$ and $h$ has a root either way.
\end{case}
\begin{case}[$\rank(A)=1$ and $A\geq0$]
    The matrices $A$ and $\widehat{A}$ have a positive eigenvalue, therefore $h$ will have a pole in $(-\infty,0)$ and thus also a root. 
\end{case}
\begin{case}[$\rank(A)=1$, $A\leq0$ and $ b\notin\Image\widehat A$]
    In this case, $h$ has no pole on $(-\infty,0)$ and there is a vector $ w_0$ in the kernel of $\widehat A$ such that $\langle w_0, b\rangle\neq0.$ From \eqref{function h} we have that $h$ is estimated
    \begin{equation}
        h(m)\leq\frac{1}{m^2}-|\langle w_0, b\rangle|^2<0,
    \end{equation}
    where the last inequality holds for $-m$ sufficiently large and so $h$ has a root as well.
\end{case}
\begin{case}[$\rank(A)=1$, $A\leq0$, $ b\in\Image\widehat A$ and $A\in\R^{l\times l}$]
    Since $A$ is real-symmetric, we have $A=\widehat A$ and $\widehat A$ is also a rank one matrix and $h$ has no poles in $(-\infty,0)$. Either $ b=0$ or $ b$ is an eigenvector of $\widehat A$ with eigenvalue $\mu=-\Vert A\Vert$ since $A=\widehat A$, $\rank A=1$ and $A\leq0.$ In both cases $h$ becomes
    \begin{equation}
        h(m)=\frac{1}{m^2}\left(1-\frac{(\mu m)^2}{(1+\mu m)^2}\right)-\Vert b\Vert^2\frac{1}{(1+2\mu m)^3}
    \end{equation}
    To find the asymptotics for $m\to-\infty$ consider $0<\zeta<1$, with $\zeta=(\mu m)^{-1}$. We find
    \begin{equation}
        \begin{split}
            h(m)&=\Vert A\Vert^2\zeta^2\left(1-\frac{1}{(\zeta+1)^2}\right)-\frac{\Vert b\Vert^2}{8}\frac{\zeta^3}{(\frac12\zeta+1)^3}\\
            &=\zeta^3\left(\Vert A\Vert^2\sum_{k=0}^\infty (k+2)(-\zeta)^k-\frac{\Vert b\Vert^2}{8}\sum_{k=0}^\infty \frac{(k+1)(k+2)}{2}(-\frac12\zeta)^k\right)\\
            &=\left(2\Vert A\Vert^2-\frac{\Vert b\Vert^2}{8}\right)\zeta^3+3\left(\frac{\Vert b\Vert^2}{16}-\Vert A\Vert^2\right)\zeta^4+\left(4\Vert A\Vert^2-\frac{3\Vert b\Vert^2}{8}\right)\zeta^5+\mathcal{O}(\zeta^6).
        \end{split}
    \end{equation}
    Therefore $h$ has a root if $\Vert  b\Vert>4\Vert A\Vert$. If $\Vert  b\Vert\leq4\Vert A\Vert$ then the leading order coefficient of the expansion is positive. Since $h(m)\geq0$ implies $h'(m)>0$ this implies $h(m)>0$ for all $m\in(-\infty,0)$. Therefore there is no root and the first two cases in \eqref{h asymptotics} follow.    
\end{case}
\begin{case}[$\rank(A)=1$, $A\leq0$, $ b\in\Image\widehat A$ and $A\in\C^{l\times l}\setminus\R^{l\times l}$]
    Since $A$ is not real-symmetric, we have $\rank \widehat A=2$ by Lemma~\ref{lemma hat A} and $h$ has no poles in $(-\infty,0)$. Since $ b\in\Image\widehat A$, the function $h$ can be expressed as
    \begin{equation}
        h(m)=\frac{1}{m^2}\left(1-\frac{(\mu m)^2}{(1+\mu m)^2}\right)- |\langle b, w_+\rangle|^2\frac{1}{(1+2\mu_+ m)^3}-|\langle b, w_-\rangle|^2\frac{1}{(1+2\mu_- m)^3}
    \end{equation}
    To find the asymptotics for $m\to-\infty$ consider $0<\zeta<1$, with $\zeta=(-\Vert A\Vert m)^{-1}$. We find
    \begin{equation}
        \begin{split}
            h(m)
            =&\Vert A\Vert^2\zeta^2\left(1-\frac{1}{(\zeta+1)^2}\right)
            -\frac{|\langle b, w_+\rangle|^2}{8r_+^3}\frac{\zeta^3}{(\frac1{2r_+}\zeta+1)^3}
            -\frac{|\langle b, w_-\rangle|^2}{8r_-^3}\frac{\zeta^3}{(\frac1{2r_+}\zeta+1)^3}\\
            =&\left(2\Vert A\Vert^2-\frac{|\langle b, w_+\rangle|^2}{8r_+^3}-\frac{|\langle b, w_-\rangle|^2}{8r_-^3}\right)\zeta^3
            \\ & \qquad \qquad \qquad \qquad +3\left(\frac{|\langle b, w_+\rangle|^2}{16r_+^4}-\frac{|\langle b, w_-\rangle|^2}{16r_-^4}-\Vert A\Vert^2\right)\zeta^4
            +\mathcal{O}(\zeta^5).
        \end{split}
    \end{equation}
    By an analogous consideration to the above case, $h$ has a root if and only if 
    \begin{equation}
        \frac{|\langle b, w_+\rangle|^2}{r_+^3}+\frac{|\langle b, w_-\rangle|^2}{r_-^3}>(4\Vert A\Vert)^2
    \end{equation} and the remaining two cases in \eqref{h asymptotics} follow.    
\end{case}
\end{proof}

\section{Linearization}

Throughout the remainder of the paper, we suppress the index of the identity matrix if it is in dimension $l+1$ or in $(l+1)N$, i.e. $I:=I_{l+1}\in\C^{(l+1)\times (l+1)}$ and $\mb I:=\mb I_{(l+1)N}\in\C^{(l+1)N\times (l+1)N}$. In all other cases, we still write out the dimension in the index. 

To prove that $\mb g$ converges towards $m$ in the sense laid out in Theorem~\ref{thm LL}, we use the linearization method, which we briefly introduce here. For more details on the construction of linearizations in the context of random matrices, see e.g. \cite{erdos2019local, Haagerup2006, haagerup2005}. First, assume that $A$ is invertible. We define the linearization $\mb L$ of $q$ as
\begin{equation}
\mb L=K_0+\sum_{j=1}^lK_j\otimes \mb X_j\in\C^{(l+1)N\times(l+1)N},
\label{eq linearization}
\end{equation}
where
\begin{equation}
K_0
=\begin{pmatrix}
c&0\\
0&-A^{-1}
\end{pmatrix}
\in\C^{(l+1)\times(l+1)}
\quad
\text{and}
\quad
K_j
=\begin{pmatrix}
b_j&e_j^t\\
e_j&0
\end{pmatrix}
\in\C^{(l+1)\times(l+1)}
\label{eq K matrices}
\end{equation}
for $j\in\llbracket l\rrbracket$ and $e_j$ being the $j^{\mr{th}}$ standard Euclidean base vector. Here we made use of our convention to identify any matrix $R\in\C^{k\times n}$, $k,n\in\N$ with $R\otimes \mb I_N\in\C^{kN\times nN}$ introduced in Section~\ref{Sec notations}. Let $J\in\C^{(l+1)\times (l+1)}$ be the orthogonal projection onto the first entry. For $\delta\in[0,1]$ and $z=E+\I\eta\in\HH$ we define 
\begin{equation}
\mb G_\delta=(\mb L-zJ-\I\eta\delta(\mb I-J))^{-1}\in\C^{(l+1)N\times(l+1)N}
\end{equation}
$\mb G_\delta$ is a generalized resolvent and using the Schur complement formula we obtain
\begin{equation}
\mb G_\delta=
\begin{pmatrix}
\mb g_\delta & 
\mb g_\delta\mb X^tA_\delta\\
A_\delta\mb X\mb g_\delta&
-A_\delta+A_\delta\mb X \mb g_\delta\mb X^tA_\delta
\end{pmatrix}, 
\label{resolvent delta}
\end{equation}
where
\begin{equation}
A_\delta=A(I_l+\I\delta\eta A)^{-1}\in\C^{l\times l}
\quad
\text{and}
\quad
\mb g_\delta=\left(\mb X^tA(I_l+\I\delta\eta A)^{-1}\mb X+ b^t\mb X+c-z\right)^{-1}\in\C^{N\times N}.
\end{equation}
Note that $I_l+\I\delta\eta A$ is invertible for all $\eta>0$ and $\delta\in[0,1]$ as $A$ is Hermitian. In particular we find
\begin{equation}
   (\mb G_0)_{11}=\mb g\in\C^{N\times N}. 
\end{equation}
 This justifies, why $\mb L$ is called the linearization of $q$. The $\delta\neq0$ case adds an additional regularization, which we use to prove the local law, Proposition~\ref{prop LLL}, for $\delta=0$. $\mb G_\delta$ satisfies the equation
\begin{equation}
\mb I+(zJ+\I\eta\delta(\mb I-J)-K_0+\mc S[\mb G_\delta])\mb G_\delta=\mb D,
\label{perturbed DE}
\end{equation}
where
\begin{equation}
\mc S[\mb R]=\E[(\mb L-\E[\mb L])\mb R(\mb L-\E[\mb L])]
\end{equation}
is called the self-energy operator and $\mb D$ is the error term. It is defined by
\begin{equation}
\mb D:=\mc S[\mb G_\delta]\mb G_\delta+(\mb L-K_0)\mb G_\delta.
\label{error delta}
\end{equation}
We split the self energy term into $\widetilde\Gamma$ and $\mc S_{\mr o}$, i.e. $\mc S=\widetilde\Gamma+\mc S_{\mr o}$. 
The first part, $\widetilde\Gamma$, depends only on the averaged trace of its argument, i.e. for all $R\in\C^{(l+1)N\times (l+1)N}$ we have
\begin{equation}
    \widetilde\Gamma[\mb R]=\Gamma[\underline{\mb R}].
    \label{eq Gamma big small}
\end{equation}
The blockwise trace, $\underline{\mb R}$, was introduced in \eqref{underline op} and $\Gamma: \C^{(l+1)\times (l+1)}\to  \C^{(l+1)\times (l+1)}$ is given by
\begin{equation}
\Gamma\left[
\begin{pmatrix}
    \omega &  v^t\\
     w & T
\end{pmatrix}
\right]
=
\begin{pmatrix}
    \omega\Vert b\Vert^2+ b^t( v+ w)+\Tr T & \omega b^t+ w^t\\
    \omega  b+ v & \omega I_l
\end{pmatrix},
\label{Gamma formula}
\end{equation}
where $\omega\in\C$, $ v, w\in\C^l$ and $T\in\C^{l\times l}$. In \eqref{eq Gamma big small} we have made use of our convention to identify $\Gamma[\underline{\mb R}]$ with $\Gamma[\underline{\mb R}]\otimes \mb I_N$

The second term $\mc S_{\mr{o}}: \C^{(l+1)N\times (l+1)N}\to  \C^{(l+1)N\times (l+1)N}$ is given by
\begin{equation}
    \mc S_{\mr{o}}\left[
    \begin{pmatrix}
        \bm \omega & \mb V^{t}\\
        \mb W & \mb T
    \end{pmatrix}
    \right]
    =\frac 1N \E\zeta_1^2
    \begin{pmatrix}
        \Vert  b\Vert^2 \bm\omega^{(o)}+ b^t(\mb V^{(o)}+\mb W^{(o)})+\Tr_{\mr b}\mb T^{(o)}
        & \bm\omega^{(o)} b^t +\mb W^{(o)}\\
        \bm\omega^{(o)} b +  (\mb V^t)^{(o)} 
        & I_l\otimes\bm\omega^{(o)}
    \end{pmatrix},
    \label{eq SE small}
\end{equation}
where $\bm \omega\in\C^{N\times N}$, $\mb V=(\mb V_i)_{i\in\llbracket l\rrbracket}\in(\C^{N\times N})^l$, $\mb W=(\mb W_i)_{i\in\llbracket l\rrbracket}\in(\C^{N\times N})^l$ and $\mb{T}\in\C^{lN\times lN}$. For any $\mb{R}=(\mb{R}_{ij})_{i\in\llbracket k\rrbracket,j\in\llbracket n\rrbracket}\in\C^{kN\times nN}$, $n,k\in\N$, we define $\mb{R}^{(o)}$ by  $(\mb{R}^{(o)})_{ij}:=\mb{R}_{ji}^t-\diag(\mb{R}_{ji})\in\C^{N\times N}$. For $\mb{R}=(\mb{R}_{ij})_{i,j\in\llbracket l\rrbracket}\in\C^{lN\times lN}$ the block-trace $\Tr_{\mr b}$ is defined by 
\begin{equation}
    \Tr_{\mr b}\mb R=\sum_{i=1}^l\mb R_{ii}\in\C^{N\times N}.
\end{equation}

Equation \eqref{perturbed DE} without the error term and $\mc S_{\mr{o}}$ is called the Dyson equation (DE) and its solution  is denoted by $M_\delta=M_\delta(z)\in\C^{(l+1)\times(l+1)}$, i.e.
\begin{equation}
I+(zJ+\I\eta\delta(I-J)-K_0+\Gamma[M_\delta])M_\delta=0\in\C^{(l+1)\times(l+1)}.
\label{DE}
\end{equation}
In the next subsection, we will lay out in what sense $\mb G_\delta$ is close to $M_\delta$. \cite[Lemma~2.6]{erdos2019local} asserts the existence of a unique analytic solution to \eqref{DE} for $\delta=0$ and \cite[Theorem~2.1]{helton2007operatorvalued} guarantees a unique solution for $\delta\in(0,1]$.
For $m_\delta\in\C$, $v_\delta,w_\delta\in\C^l$ and $\widehat M_\delta\in \C^{l\times l}$ we partition $M_\delta$ as 
\begin{equation}
    M_\delta:=\begin{pmatrix}
        m_\delta & v_\delta^t\\
        w_\delta & \widehat M_\delta
    \end{pmatrix}.
    \label{M_delta components}
\end{equation}
Since $M_\delta$ solves \eqref{DE} it is invertible and by the Schur complement formula its inverse is given by
\begin{equation}
    M_\delta^{-1}=
    \begin{pmatrix}
        m^{-1}+m^{-2} v^t_\delta\left(\widehat M_\delta-w_\delta m_\delta^{-1}v_\delta^t\right)^{-1} w_\delta 
        & -m^{-1}v_\delta^t\left(\widehat M_\delta-w_\delta m_\delta^{-1}v_\delta^t\right)^{-1} \\
        -m^{-1}\left(\widehat M_\delta-w_\delta m_\delta^{-1}v_\delta^t\right)^{-1}w_\delta
        &\left(\widehat M_\delta-w_\delta m_\delta^{-1}v_\delta^t\right)^{-1}   
    \end{pmatrix}.
    \label{M inverse Schur}
\end{equation}
At the same time by \eqref{DE} the inverse of $M_\delta$ can be expressed as
\begin{equation}
    M_\delta^{-1}=
    \begin{pmatrix}
    -z-m_\delta b^tb- b^t( v_\delta+ w_\delta)-\Tr \widehat M_\delta 
    & -m_\delta b^t- w_\delta^t\\
    -m_\delta b-v &
    -A^{-1}-(m_\delta+\I\eta\delta)I_l
\end{pmatrix}.
    \label{M inverse DE}
\end{equation}
Comparing the two expressions for $M_\delta^{-1}$, we obtain a set of equations for $m_\delta,v_\delta,w_\delta$ and $\widehat M_\delta$. Solving them we find explicit expressions for $v_\delta,w_\delta$ and $\widehat M_\delta$ in terms of $m_\delta$.
In other words, $M_\delta=M_\delta[m_\delta]$ can be expressed purely in terms of its (1,1) entry with $M_\delta[x]$ given by
\begin{equation} 
M_\delta[x]:=
\begin{pmatrix}
x&-xb^tV_\delta(x) A_\delta\\
-x A_\delta V_\delta(x)  b&-A_\delta(I_l+xA_\delta)^{-1}+xA_\delta V_\delta(x) b b^tV_\delta(x) A_\delta
\end{pmatrix}
\label{sol DE}
\end{equation}
with
\begin{equation}
V_\delta(x):=x (I_l+2x\widehat{A}_\delta)^{-1}.
\label{eq B delta}
\end{equation}
Note that we use square brackets to denote $M_\delta$ as a function of its (1,1) entry. This is done to avoid confusion with $M_\delta(z)$, which denotes $M_\delta$ as a function of the spectral parameter $z.$ The two functions are related by 
$M_\delta[m_\delta(z)]=M_\delta(z)$.
The entry $m_\delta$ satisfies the equation
\begin{equation}
-m_\delta^{-1}=z+\gamma_\delta(m_\delta),
\label{scalar eq delta}
\end{equation}
where
\begin{equation}
\gamma_\delta(x):=-\Tr A_\delta(I_l+A_\delta x)^{-1}+ b^t(V_\delta(x)-V_\delta(x)\widehat{A}_\delta V_\delta(x)) b-c.
\end{equation}
and $\widehat{A}_\delta=\frac12(A_\delta+A_\delta^t)$.
For $\delta=0$, Equation~\eqref{scalar eq delta} corresponds precisely to \eqref{scalar equation}.
From here on we no longer assume that $A$ is invertible and instead, we use the Equations $\eqref{resolvent delta}$, \eqref{error delta} and \eqref{sol DE} as the definition for $\mb G_\delta$, $\mb D$ and $M_\delta=M_\delta[m_\delta]$ respectively. Note that this is also well defined for \eqref{error delta} as $K_0$ and $\mb L$, defined in \eqref{eq K matrices} and \eqref{eq linearization} respectively, depend on $A^{-1}$, but $\mb L-K_0$ does not. The function $m_\delta$ is uniquely defined by the following lemma.
\begin{lemma}
    Let $\rho$ have a regular edge at $\tau_\pm$ and let $\delta\in(0,1]$. The function $m_\delta$ is uniquely defined by the following criteria around the edge and away from the spectrum.
    \begin{enumerate}
        \item There is a $u>0$ only depending on the coefficients of $q$ such that for all $z\in\HH$ with $|z-\tau_\pm|<u$ there is a unique function $m_\delta=m_\delta(z)$ that solves \eqref{scalar eq delta} and satisfies
    \begin{equation}
        \Im m_\delta\sim
        \begin{cases}
            \sqrt{\kappa+\eta} &\text{if }E\in\supp(\rho)\\
            \frac{\eta}{\sqrt{\kappa+\eta}}&\text{if }E\notin\supp(\rho)
        \end{cases}
        \quad
        \text{as well as}
        \quad
        |m_\delta-m_\pm|^2\sim|z-\tau_\pm|,
        \label{eq m_delta asymptotics}
    \end{equation}
    with $\kappa=|E-\tau_\pm|$ and $m_\pm=\lim_{\eta\searrow0}m(\tau_\pm+\I\eta)$.
    \item For all $C>0$ there is an $\eta_0$ such that for all $z=E+\I\eta\in\HH$ with $C^{-1}<\dist(E,\supp(\rho))<C$ and $\eta<\eta_0$ there is a unique function $m_\delta=m_\delta(z)$ that solves \eqref{scalar eq delta} and satisfies
    \begin{equation}
        \Im m_\delta\sim_C\eta
        \quad
        \text{as well as}
        \quad
        |m_\delta(z)-m(E)|\sim\eta,
        \label{eq m_delta away from rho}
    \end{equation}
    with $m(E)=\lim_{\eta\searrow0}m(E+\I\eta)$.
    \end{enumerate}
    In each case $M_\delta=M_\delta[m_\delta]$ is also the unique solution to \eqref{DE} with $\Im M_\delta>0$ if the coefficient matrix $A$ is invertible.
    \label{lemma m delta asymptotics}
\end{lemma}
The lemma is proven in Appendix~\ref{Appendix DE}.

\subsection{Local law for the linearization\label{SS LLL}}
\begin{definition}[Shifted square of a Wigner matrix]
    A polynomial $q$ as in \eqref{eq polynomial} with $A\in\R^{l\times l},$ $\rank A=1$ and $b=0$ is called a \textbf{shifted square of a Wigner matrix.}
    \label{def shifted square}
\end{definition}
\begin{remark}
    If a polynomial $q$ satisfies the above definition, then there is some $v\in\R^l$ with $v\neq0$ and 
    \begin{equation}
        q(\mb X)=\pm(v^t\mb X)^2+c.
    \end{equation}
    Here, $\mb W:=\frac{1}{\Vert v\Vert} v^t\mb X$ is a Wigner matrix, normalised such that $\E |w_{ij}|^2 = \frac{1}{N}$ and we have 
    \begin{equation}
        q(\mb X)=\pm\Vert v\Vert^2\mb W^2+c.
    \end{equation}
    This justifies the terminology in the above definition. 
\end{remark}

In case $q$ is not a shifted square of a Wigner matrix, we will prove that $M_0$ and $\mb G_0$ are close to each other around any regular edge. 
As $m$ and $\mb g$ are sub-matrices of $M_0$ and $\mb G_0$ respectively, this also implies closeness of $\mb g$ and $m$. This result, a local law for the linearization, is presented below in Propositions~\ref{prop LLL} and \ref{prop LLL away}. For $q$ being a shifted square of a Wigner matrix, we will prove the local law, Theorem~\ref{thm LL} and Proposition~\ref{prop LL away}, directly without the use of a linearization. The reason why we prove this case separately is that the stability operator $\IL$, defined below in \eqref{eq stab operator}, does have an additional unstable direction.

\begin{prop}[Edge local law for the linearization]
Let $q$ be a polynomial of the form \eqref{eq polynomial} that is not a shifted square of a Wigner matrix and let the corresponding $\rho$ have a regular edge at $\tau_0$. There is a $\kappa_0>0$ depending only on the parameters of $q$ such that for all $\veps,\gamma,D>0$, $\delta\in\{0,1\}$ and $z=E+\I\eta\in\mathbb{D}_{\gamma}^{\kappa_0}$ the isotropic local law
\begin{equation}
\Pb\left(|\langle \mb x, (\mb G_\delta-M_\delta)\mb y\rangle|>\Vert\mb x\Vert\Vert\mb y\Vert N^\veps\left(\frac{\Im m}{\sqrt{N\eta}}+\frac{1}{N\eta}\right)\right)\lesssim_{\veps,\gamma,D}N^{-D}
\label{eq lll iso}
\end{equation}
holds for all deterministic $\mb x,\mb y\in\C^{(l+1)N}$.
Moreover, the averaged local law
\begin{equation}
\Pb\left(|\langle \mb B(\mb G_\delta-M_\delta)\rangle|>\Vert \mb B\Vert\frac{N^\veps}{N\eta}\right)\lesssim_{\veps,\gamma,D}N^{-D}
\label{eq lll av}
\end{equation}
holds for all deterministic $\mb B\in\C^{(l+1)N\times (l+1)N}$.
For $E\notin\supp(\rho)$, an improved averaged local law of the form
\begin{equation}
\Pb\left(|\langle \mb B(\mb G_\delta-M_\delta)\rangle|>\Vert \mb B\Vert N^\veps\left(\frac{1}{N(\kappa+\eta)}+\frac{1}{(N\eta)^2\sqrt{\kappa+\eta}}\right)\right)\lesssim_{\veps,\gamma,D}N^{-D},
\label{eq lll imp}
\end{equation}
with $\kappa=|E-\tau_0|$, is obtained.
 \label{prop LLL}
\end{prop}

\begin{prop}[Local law for the linearization away from the spectrum]
    Let $q$ be a polynomial of the form \eqref{eq polynomial} that is not a shifted square of a Wigner matrix. For all $C>0$ there is an $\eta_0>0$, depending only on the coefficients of $q$, such that an averaged local law of the form
    \begin{equation}
        \Pb\left(|\langle\mb B(\mb G_\delta-M_\delta)\rangle|>\Vert\mb B\Vert N^\veps\left(\frac{1}{N}+\frac{1}{(N\eta)^2}\right)\right)
        \lesssim_{\veps,\gamma,D,C}N^{-D}
    \label{eq lll away}
    \end{equation}
     holds true for all deterministic $\mb B\in\C^{(l+1)N\times (l+1)N}$, $\veps,\gamma,D>0$ and $z\in\mathbb G_{\gamma}^{C,\eta_0}$.
     \label{prop LLL away}
\end{prop}

To obtain Theorem~\ref{thm LL}, we only require the $\delta=0$ case, but to prove it we will require the $\delta=1$ case, thus we state both cases together in the proposition.
The proof of Proposition~\ref{prop LLL} has two major ingredients. For one we show that the error term $\mb D$ in \eqref{perturbed DE} is indeed small, this is done in Proposition~\ref{prop error scd}, which we import from \cite[Theorem~4.1]{ERDOS_2019SCD} and adjust to our setting. Additionally, we need to prove that \eqref{DE} is stable under small perturbations. This is done in Proposition~\ref{prop stab}.
\begin{prop}
Let $\veps>0$, $p\in\N$ and $\delta\in[0,1]$. Then there is a $c>0$ such that
\begin{equation}
\Vert\mb D\Vert_p\lesssim_{p,\veps}N^\veps\sqrt{\frac{\Vert \mb G_\delta\mb G_\delta^*\Vert_{p_0}}{N}}\left(1+\Vert \mb G_\delta\Vert_{p_0}\right)^c\left(1+N^{-\frac{1}{4}}\Vert \mb G_\delta\Vert_{p_0}\right)^{cp}
\label{bound error iso}
\end{equation}
and
\begin{equation}
\Vert \mb D\Vert_p^{\mr{av}}\lesssim_{p,\veps}N^\veps\frac{\Vert \mb G_\delta^*\mb G_\delta\Vert_{p_0}}{N}\left(1+\Vert \mb G_\delta\Vert_{p_0}\right)^c\left(1+N^{-\frac{1}{4}}\Vert \mb G_\delta\Vert_{p_0}\right)^{cp}
\label{bound error av}
\end{equation}
where we defined $p_0=c\frac{p^4}{\veps}$.
\label{prop error scd}
\end{prop}
\begin{proof}
    First, consider the case of $A$ being invertible. Then $\mb G_\delta=(\mb L-zJ-\I\eta\delta(\mb I-J))^{-1}$ and our proof follows the proof of \cite[Theorem~4.1]{ERDOS_2019SCD}  line by line with the exception that the Ward identities, (51a) and (51b), do not apply for $\delta=0$. Thus, we cannot replace the $\mb G_\delta^*\mb G_\delta$ terms by $\eta^{-1}\Im\mb G_\delta$ and we instead are left with the upper bounds \eqref{bound error iso} and \eqref{bound error av}.

    Now consider $A$ being non-invertible. Then $A+\veps$ is invertible for all $\veps\in(0,u)$ for some $u\sim1$. We denote the $\mb G_\delta$ associated with $A+\veps$ by $\mb G_\delta^\veps$ and we have $\lim_{\veps\to0}\mb G_\delta^\veps=\mb G_\delta$. We thus only need to prove that the constants in \eqref{bound error iso} and \eqref{bound error av} are uniform in $\veps$ to obtain the proposition in this case. This is non-trivial as \cite[Theorem~4.1]{ERDOS_2019SCD} states as a condition that $\E[\mb L]$ is bounded and this is clearly not the case in the $\veps\to0$ limit. The assumption is only used, however, to ensure that $\mb G_\delta^*\mb G_\delta$ satisfies the lower bound $\Vert \mb G_\delta^*\mb G_\delta\Vert_{p_0}\gtrsim 1$. In our case, this follows instead from
    $\Vert \mb G_\delta^*\mb G_\delta\Vert_{p_0}\geq \Vert \mb g_\delta^*\mb g_\delta\Vert_{p_0}\gtrsim_{p_0}\Vert \mb g^*\mb g\Vert_{p_0}$. Since $\mb g$ is the resolvent of $q(\mb X)$ we have $\Vert \mb g^*\mb g\Vert_q\geq\E[\Vert q(\mb X)\Vert^{-1}]\gtrsim 1$ uniformly in $z$ for bounded $z$. Here, the last inequality holds true since $q(\mb X)$ satisfies the inequality
    \begin{equation}
        \Vert q(\mb X)\Vert\leq \sum_{i,j=1}^l|A_{ij}|\Vert \mb X_i\Vert\Vert \mb X_j\Vert +\sum_{i=1}^l|b_i|\Vert \mb X_i\Vert +|c|
        \label{trivial bound}
    \end{equation}
    and $\E[\Vert \mb X_i\Vert]\lesssim 1$ for all $i\in\llbracket l\rrbracket$ since the $\mb X_i$ are Wigner matrices.
    Therefore the proposition also holds for non-invertible $A$.
\end{proof}

The following result concerns the stability of the Dyson equation. The stability operator $\IL:\C^{(l+1)N\times (l+1)N}\to\C^{(l+1)N\times (l+1)N}$ is given by 
\begin{equation}
\IL[\mb R]=\mb R-M_\delta\mc{S}[\mb R]M_\delta
\label{eq stab operator}
\end{equation}
and we will prove
\begin{prop}[Control of $\IL$]
Let $q$ be a polynomial of the form \eqref{eq polynomial} that is not a shifted square of a Wigner matrix and let the corresponding $\rho$ have a regular edge at $\tau_0$. There exists a $u\sim1$ such that for all $z=E+\I\eta\in\HH$ with $|z-\tau_0|<u$ and $\delta\in[0,1]$ there exists an eigenvalue $\beta$ with corresponding left and right eigenvectors $L,B\in\C^{(l+1)\times(l+1)}$ of $\IL$  such that
\begin{equation}
\begin{split}
\Vert\IL^{-1}\Vert_{\mr{sp}}\sim(\kappa+\eta)^{-\frac12},&\quad
\Vert(\mathbb 1-\mathscr P)\IL^{-1}\Vert_\mr{sp}\lesssim 1,\quad
|\beta|\sim(\kappa+\eta)^{\frac12},\\
|\langle L,B\rangle|\sim1,&\quad
\Vert L\Vert+\Vert B\Vert\sim1,\quad
|\langle L,M_\delta\mc S[B]B\rangle|\sim1,
\end{split}
\end{equation}
 with $\mathbb 1$ being the identity operator, $\mathscr{P}$ being the spectral projection onto $B$, i.e.
 \begin{equation}
     \mathscr{P}=(\langle L\otimes \mb I_N,B\otimes \mb I_N\rangle)^{-1}\langle L\otimes \mb I_N,\cdot\rangle (B\otimes \mb I_N)
     \label{eq mathscr P}
 \end{equation}
 and $\kappa=|E-\tau_0|$.

Furthermore, for any $C>0$ and $E=\Re z$ with $C^{-1}<\dist(E,\supp(\rho))<C$ there is an $\eta_0>0$ such that we have
\begin{equation}
    \Vert\IL^{-1}\Vert_{\mr{sp}}\sim_C1
    \label{eq IL away from rho}
\end{equation}
uniformly for all $\eta\leq\eta_0.$
 \label{prop stab}
\end{prop}
The proof will be given in Section \ref{proof stability}.
\begin{remark}
    $B$ and $L$ being right and left eigenvectors of $\IL$ with eigenvalue $\beta$ is understood in the sense of
    \begin{equation}
        \IL[B]=\beta B
        \quad
        \text{and}
        \quad
        \IL^*[L]=\bar\beta L.
    \end{equation}
    Here, we used the notation $R=R\otimes \mb I_N\in\C^{(l+1)N\times(l+1)N}$ introduced in \eqref{eq embedding}. The adjoint is defined with respect to the scalar product $\langle\mb R,\mb T\rangle=\langle \mb R^*\mb T\rangle$.
    \end{remark}

Corollary~\ref{cor m reg}, as well as Propositions~\ref{prop error scd} and \ref{prop stab} are the main ingredients to Proposition~\ref{prop LLL} and are in fact sufficient for $\delta=1.$ For $\delta=0$ however, extra care is needed as the Ward identity for resolvents $\mb G$,
\begin{equation}
    \mb G\mb G^*=\frac{\Im \mb G}{\eta},
\end{equation}
does not translate to generalized resolvents. Instead, we will estimate $\mb G_0\mb G_0^*$ by $\mb G_1\mb G_1^*$, which allows us to obtain the local law for $\delta=0$ from the $\delta=1$ case. The proof of Proposition~\ref{prop LLL} will be given in Section~\ref{proof LL} and follows the general strategy from \cite{Alt_2020edge}, modified to accommodate for the lack of a Ward identity.

\subsection{Proof of Proposition~\ref{prop stab}\label{proof stability}}
\begin{proof}[Proof of Proposition~\ref{prop stab}]
We split the stability operator into $\IL=\IL^{(0)}+\IL^{(1)}$ with
\begin{equation}
    \IL^{(0)}[\mb R]:=\IL[\underline{\mb R}]
    \quad
    \text{and}
    \quad
    \IL^{(1)}[\mb R]:=\IL[\mb R-\underline{\mb R}].
\end{equation}
By \eqref{Gamma formula} and \eqref{eq SE small}, we have $\tilde{\Gamma}[\mb R-\underline{\mb R}]=\mc S_{\mr o}[\underline{\mb R}]=0.$
Thus we have
\begin{equation}
    \IL^{(0)}[\mb R]=\mc L[\underline {\mb R}]\otimes \mb I_N
    \quad
    \text{and}
    \quad
    \IL^{(1)}[\mb R]=\mb R-\underline{\mb R}-M_\delta\mc S_{\mr{o}}[\mb R]M_\delta
\end{equation}
with $\mc L: \C^{(l+1)\times(l+1)}\to\C^{(l+1)\times(l+1)}$ defined as
\begin{equation}
    \mc L[R]:=R-M_\delta\Gamma[R]M_\delta.
\end{equation}
The image of $\IL^{(0)}$ is given by
\begin{equation}
    \{\mb R\in\C^{(l+1)N\times (l+1)N}:\,\underline{\mb R}=\mb R\}=:\mc U
\end{equation}
and its kernel is given by $\mc U^\perp$, the orthogonal complement of $\mc U.$ 
At the same time the image of $\IL^{(1)}$ is contained in $\mc U^\perp$ and $\mc U$ is contained in the kernel of $\IL^{(1)}$. That is, $\IL$ decomposes into $\IL^{(0)}$ acting on $\mc U$ and $\IL^{(1)}$ acting on its orthogonal complement.

The behaviour of $\mc L$ is summarized in the following lemma.
\begin{lemma}
    Let $q$ be a polynomial of the form \eqref{eq polynomial} that is not a shifted square of a Wigner matrix and let the corresponding $\rho$ have a regular edge at $\tau_0$. There exists an $u\sim1$ such that for all $z=E+\I\eta\in\HH$ with $|z-\tau_0|<u$ and $\delta\in[0,1]$ there exists an eigenvalue $\beta$ with corresponding normalized left and right eigenvectors $L$ and $B$ of $\mc L$  such that
    \begin{equation}
        \begin{split}
        \Vert\mc L^{-1}\Vert_{\mr{sp}}\sim(\kappa+\eta)^{-\frac12},&\quad
        \Vert(\mathbb 1-\mc P)\mc L^{-1}\Vert_{\mr{sp}}\lesssim 1,\quad
        |\beta|\sim(\kappa+\eta)^{\frac12},\\
        |\langle L,B\rangle|\sim1,&\quad
        |\langle L,M_\delta\Gamma[B]B\rangle|\sim1,
        \end{split}
    \end{equation}
    with $\mc P$ being the spectral projection onto $B$, i.e. $\mc P=(\langle L,B\rangle)^{-1}\langle L,\cdot\rangle B$ and $\kappa=|E-\tau_+|$.

    Furthermore, for any $C>0$ and $E=\Re z$ with $C^{-1}\leq\dist(E,\supp(\rho))\leq C$ there is an $\eta_0>0$ such that we have
\begin{equation}
    \Vert\mc L^{-1}\Vert_{\mr{sp}}\sim_C1
    \label{eq mc L from rho}
\end{equation}
uniformly for all $\eta\leq\eta_0.$
    \label{lemma stab small D}
\end{lemma}
The proof of Lemma~\ref{lemma stab small D} is deferred to the end of the section. For $\mc S_{\mr o}$  we find
\begin{equation}
    \Vert {\mc S_{\mr o}}[\mb R]\Vert_{\mr{hs}}\lesssim\frac1N\Vert \mb R\Vert_{\mr{hs}}
\end{equation}
for all $\mb R\in\C^{(l+1)N\times(l+1)N}$. Thus $\mc S_{\mr o}$ is bounded by 
\begin{equation}
    \Vert\mc S_{\mr o}\Vert_{\mr{sp}}\lesssim\frac1N.
    \label{eq bound So}
\end{equation}
By Corollary~\ref{cor m reg}, Lemma~\ref{lemma m delta asymptotics} and \eqref{sol DE} we have $\Vert M_\delta\Vert\lesssim1$ for all $z$ such that $|z-\tau_0|\leq u$ and  some $u>0$. Combined with \eqref{eq bound So} it follows that there is a $C>0$ such that
\begin{equation}
    \Vert\IL^{(1)}\Vert_{\mr{sp}}\leq1+CN^{-1}
    \quad
    \text{and}
    \quad
    \Spec\left(\IL^{(1)}|_{\mc U^\perp}\right)\subset B_{CN^{-1}}(1),
\end{equation}
where $B_\veps(x)$ denotes the $\veps$ neighborhood of $x$.
Thus for sufficiently large $N$ the smallest eigenvalue of $\IL$ equals that of $\mc L$ and the corresponding left and right eigenvectors of $\IL$ are given by $L\otimes \mb I_N$ and $B\otimes \mb I_N$. The norm of the inverse of $\IL$ is bounded by
\begin{equation}
    \Vert\IL^{-1}\Vert_{\mr{sp}}\leq\max\{1+CN^{-1},\Vert \mc L^{-1}\Vert_{\mr {sp}}\} 
    \quad
    \text{and}
    \quad
    \Vert(1-\mc P)\IL^{-1}\Vert_{\mr{sp}}\leq\max\{1+CN^{-1},\Vert \mc L^{-1}\Vert_{\mr {sp}}\},
\end{equation}
completing the proof of Proposition~\ref{prop stab}
\end{proof}

\begin{proof}[Proof of Lemma~\ref{lemma stab small D}]
   First, we prove that $\mc L$ has exactly one vanishing eigenvalue at $\tau_0$, and from there on we conclude the proof with the help of a perturbative argument.
We define 
\begin{equation}
    \mc C_J[R]=JRJ,
\end{equation}
i.e. $\mc C_J$ is the projection onto the $(1,1)$ entry and in particular $\mc C_J[M_0]=mJ$. We also set $\mc C_J^\perp R :=R-\mc C_JR$ and $\widetilde\C^{(l+1)\times(l+1)}:=\Image \mc C_J^\perp$. For any $R\in\C^{(l+1)\times(l+1)}$ we denote $r:=R_{11}$ and $\widetilde R:=\mc C_J^\perp[R]$, i.e. $R=rJ+\widetilde R$.
Let $T_z$ be the matrix
\begin{equation}
    T_z=
    \begin{pmatrix}
        1&0\\
        0 & zA
    \end{pmatrix}
    \in\C^{(l+1)\times(l+1)}
\end{equation}
as well as
\begin{equation}
    F:
    \begin{cases}
        \C\times\widetilde\C^{(l+1)\times(l+1)}\times\HH&\to \C^{(l+1)\times(l+1)}\\
        (r,\widetilde R,z)&\mapsto r J+\widetilde R+T_z(z+\Gamma[r J+\widetilde R]T_z)^{-1}.
    \end{cases}
    \label{eq F}
\end{equation}
Then $F(m,\widetilde M, z)=0$ for $M=M_0[m]$ defined in \eqref{sol DE}.
For $\delta=0$, the stability operator $\mc L$  is the derivative of $F$ in the sense that 
\begin{equation}
    \begin{split}
        \mc L[R]&=D_RF(m,\widetilde M, z),
    \end{split}
    \label{eq mc L as derivative}
\end{equation}
where $D_R F = \frac{\dd}{\dd\veps} F(m + \veps r , \widetilde{M} + \veps \widetilde{R},z)|_{\veps=0}$ is the directional derivative of $F$ in the direction $R$.
We first consider the case when $A$ is invertible. The case of non-invertible $A$ will be treated afterwards. We define $\mc B:=\mc C_{M^{-1}}\mc L$ on $z\in\R\setminus\supp(\rho)$, where $m(z)$ is defined as the unique analytical continuation to $\C\setminus\supp(\rho)$. Since $F(m,\widetilde M, z)=0$ we have
\begin{equation}
   \begin{split}
        \mc B[R]&=D_R[\mc C_{M^{-1}}F]=D_{rJ}\mc C_J[\mc C_{M^{-1}}F]+D_{\widetilde R}\mc C_J[\mc C_{M^{-1}}F]+D_{rJ}\mc C_J^\perp [\mc C_{M^{-1}}F]+D_{\widetilde R}\mc C_J^\perp [\mc C_{M^{-1}}F],
    \end{split}
    \label{eq mc B decomposition}
\end{equation}
where we used the linearity of the derivatives as well as the linearity of $\mc B$ in the second equality and we omitted the arguments of $F$. The above equation decomposes $\mc B$ into a two-by-two block operator with diagonal blocks $\mc B_{11}[rJ]:=D_{rJ}\mc C_J[\mc C_{M^{-1}}F]$, $\mc B_{22}[\widetilde R]:=D_{\widetilde R}\mc C_J^\perp [\mc C_{M^{-1}}F]$ and off-diagonal blocks $\mc B_{12}[\widetilde R]:=D_{\widetilde R}\mc C_J [\mc C_{M^{-1}}F]$, $\mc B_{21}[rJ]:=D_{rJ}\mc C^\perp_J[\mc C_{M^{-1}}F]$.

A tedious but straightforward calculation shows that $\mc B_{22}$ is invertible on the image of $\mc{C}_J^\perp$ and its inverse is given by
\begin{equation}
    (\mc B_{22})^{-1}
    \left[
    \begin{pmatrix}
        0 & r_{12}^t\\
        r_{21} & \widehat R
    \end{pmatrix}
    \right]=
    \begin{pmatrix}
        0 & h_{12}^t\\
        h_{21} & \widehat H
    \end{pmatrix}
    \label{eq B_22 invers 1}
\end{equation}
with
\begin{equation}
    \begin{split}
    h_{12}=&-A^tV_0(mA(r_{21}-m^{-1}\widehat Rw_0)+(1+mA)(r_{12}-m^{-1}\widehat R^tv_0))\\
    h_{21}=&-AV_0(mA^t(r_{12}-m^{-1}\widehat R^tv_0)+(1+mA^t)(r_{21}-m^{-1}\widehat Rw_0))\\
    \widehat H=&\frac{A}{1+mA}\widehat R\frac{A}{1+mA}-AV_0(mA^t(r_{12}-m^{-1}\widehat R^tv_0)+(1+mA^t)(r_{21}-m^{-1}\widehat Rw_0))m^{-1}v^t\\
    &\qquad-wm^{-1}((r_{21}^t-m^{-1}w_0^t\widehat R)mA^t+(r_{12}^t-m^{-1}v_0^t\widehat R)(1+mA^t))V_0A,
    \end{split}
    \label{eq B_22 invers 2}
\end{equation}
where $V_0=V_0(m)$ was introduced in \eqref{eq B delta} and $v_0$ and $w_0$ where defined in \eqref{M_delta components}. Their explicit form in terms of $m=m_0$ is given in \eqref{sol DE}. 
From $F=0$, we also have $\mc C_J^\perp\mc C_{M^{-1}}F=0$ and both $\widetilde M=\widetilde M_0$ and $z$ are uniquely defined by $m$ (see \eqref{sol DE} and \eqref{scalar equation}). Therefore the total derivative of $\mc C_J^\perp\mc C_{M^{-1}}F$ with respect to $m$ is well defined and vanishes as well, i.e.
\begin{equation}
    0
    =\frac{\dd}{\dd m}\mc C_J^\perp\mc C_{M^{-1}}F
    =\mc C_J^\perp D_J\mc C_{M^{-1}}F
    +\mc C_J^\perp D_{\widetilde M'}\mc C_{M^{-1}}F
    +z'(m)\frac{\partial}{\partial z}\mc C_J^\perp\mc C_{M^{-1}}F
    =\mc B_{21}[J]+\mc B_{22}[\widetilde M']
    \label{total defivative m}
\end{equation}
where $\widetilde M':=\frac{\partial\widetilde M}{\partial m}$ and $z'$ denote the derivative of $\mc C_J^\perp[M[m]]$ and $z$ with respect to $m$. In the last step we also used $\frac{\partial}{\partial z}\mc C_J^\perp\mc C_{M^{-1}}F=0$, which follows from \eqref{DE}.
Since $\mc B_{22}$ is invertible on its image, \eqref{total defivative m} is equivalent to
\begin{equation}
    (\mc B_{22})^{-1}\mc B_{21}[J]=-\widetilde M'.
\end{equation}
Therefore the Schur complement of $\mc B_{22}$ is given by
\begin{equation}
    (\mc B_{11}-\mc B_{12}(\mc B_{22})^{-1}\mc B_{21})[J]
    =\frac{\partial}{\partial m}C_J[\mc C_{M^{-1}}F]+D_{\widetilde M'}\mc C_J[\mc C_{M^{-1}}F].
\end{equation}
In other words, the Schur complement of $\mc B_{22}$ is the total derivative of $\mc C_J[\mc C_{M^{-1}}F]$ with respect to $m$ for fixed $z$. Calculating it, we find
\begin{equation}
    (\mc B_{11}-\mc B_{12}(\mc B_{22})^{-1}\mc B_{21})[J]
    =\left(\frac{1}{m^2}-\gamma'(m)\right)J=h(m)J,
\end{equation}
where $h$ was introduced in Definition~\ref{def h}. By Lemma~\ref{lemma h non-red}, Lemma~\ref{lemma h red} and Lemma~\ref{lemma m away from rho} we have $h(m)\neq0$ for all $z\in\R\setminus\supp(\rho).$ Therefore, the Schur complement of $\mc B_{22}$ is  invertible on the image of $\mc C_J$ for all $z\in\R\setminus\supp(\rho)$. Since both $\mc B_{22}$ and its Schur complement are invertible on their respective images for all $z\in\R\setminus\supp(\rho)$, the operator $\mc B$ for all $z\in\R\setminus\supp(\rho)$ is also invertible with its inverse given by the Schur complement formula
 \begin{equation}
 \begin{split}
     \mc B^{-1}[R]&=\left(\mathbb1-(\mc B_{22})^{-1}\mc B_{21}\right)\left(\mc B_{11}-\mc B_{12}(\mc B_{22})^{-1}\mc B_{21}\right)^{-1}\left[rJ-\mc B_{12}(\mc B_{22})^{-1}[\widetilde R]\right]+(\mc B_{22})^{-1}[\widetilde R]\\
     &=h(m)^{-1}\left(r-\left\langle J,\mc B_{12}(\mc B_{22})^{-1}[\widetilde R]\right\rangle\right)M'+(\mc B_{22})^{-1}[\widetilde R],
     \label{eq mc B invers}
 \end{split}
 \end{equation}
 where $M'=J+\widetilde M'$ is the derivative of $M$ with respect to $m$ and the inverse of the Schur complement is acting on $\Span(J)$. Therefore, $\mc L$ is also invertible with inverse $\mc L^{-1}=\mc B^{-1}\mc C_{M^{-1}}$ for $z\in\R\setminus\supp(\rho)$. Now let $A$ be non-invertible. Then, $A+\veps$ is invertible for all $\veps\in(0,u]$ for some $u>0$. We use an $\veps$ subscript to denote quantities with $A$ being replaced by $A^\veps:=A+\veps$, e.g. $\mc L^\veps$, $(\mc B^{-1})^\veps$ etc. We have $\mc L^\veps(\mc B^{-1}\mc C_{M^{-1}})^\veps=\mathbb 1$ for all $\veps\in(0,u]$ and both the limits $\lim_{\veps\searrow0}\mc L^\veps$ and  $\lim_{\veps\searrow0}(\mc B^{-1}\mc C_{M^{-1}})^\veps$ exist.
 Indeed, $\mc L^\veps$ is the derivative of $F^\veps$, which is smooth at $\veps=0$, proving the existence of $\lim_{\veps\searrow0}\mc L^\veps$.
 Furthermore we explicitly calculate the expression for $(\mc B^{-1}\mc C_{M^{-1}})^\veps$ in terms of $A^\veps$ and $m^\veps$ by using \eqref{M inverse DE}, \eqref{eq B_22 invers 1}, \eqref{eq B_22 invers 2} and \eqref{eq mc B invers}.
 In the resulting expression, $(A^\veps)^{-1}$, appearing in \eqref{M inverse DE}, cancels and therefore the limit $\lim_{\veps\searrow0}(\mc B^{-1}\mc C_{M^{-1}})^\veps$ exists. We leave the details to the reader.
 Therefore, $\mc L$ is also invertible for $z\in\R\setminus\supp(\rho)$ and we have $\mc L^{-1}=\lim_{\veps\searrow0}(\mc B^{-1}\mc C_{M^{-1}})^\veps$.

For any $C>0$ the norm of the operator $\mc L^{-1}$ is bounded on the compact subset $C^{-1}\leq\dist(z,\supp(\rho))\leq C$ of $\R\setminus\supp(\rho)$, i.e. $\Vert \mc L^{-1}\Vert_{\mr{sp}}\sim_C1$. By continuity, there is some $\eta_0>0,$ depending on $C$, such that $\Vert \IL^{-1}\Vert_{\mr{sp}}\sim_C1$ still holds for $z=E+\I\eta\in\HH$ with $C^{-1}\leq\dist(E,\supp(\rho))\leq C$ and $\eta\leq\eta_0$ and we have therefore shown \eqref{eq mc L from rho}.
 
Let $\mc Q_{M'}$ be the projection onto the orthogonal complement of $M'$. By \eqref{eq B_22 invers 1}, \eqref{eq B_22 invers 2} and \eqref{eq mc B invers} we have
\begin{equation}
    \Vert \mc Q_{M'}\IL^{-1}\Vert_{\mr{sp}}
    \lesssim \left\Vert\frac{1}{1+Am}\right\Vert+\Vert V_0\Vert
    \lesssim \dist\left(-1,\Spec(mA)\right)+\dist\left(-\frac12,\Spec(m\widehat A)\right)
    \label{Q_T projection}
\end{equation}
uniformly in $z\in\R\setminus\supp(\rho)$. 
At $z=\tau_0$ we have $h(m)=0$ and thus
\begin{equation}
    \left\Vert\frac{(mA)^2}{(1+mA)^2}\right\Vert\leq\Tr\left(\frac{(mA)^2}{(1+mA)^2}\right)+b^t\left(\frac{V_0}{m}\right)^3b=1.
\end{equation}
In the above equation, we have equality if and only if $\rank A=1$ and $b=0$. Therefore, we have $\Spec(mA)\subset[-\frac12,\infty)$ and $-\frac12\in\Spec(A)$ if and only if $\rank A=1$ and $b=0$. By Lemma~\ref{lemma hat A}, we have $\min \Spec (m\widehat A)=\min \Spec(mA)$ if and only if $mA\in\R^{l\times l}$. Therefore we have $-\frac12\in \Spec (m\widehat A)$ if and only if $\rank A=1$, $A\in\R^{l\times l}$ and $b=0$. That is, $-\frac12\in \Spec (m\widehat A)$ at $\tau_0$ if and only if $A$ is a shifted square of a Wigner matrix. If $A$ is not a shifted square of a Wigner matrix, then \eqref{Q_T projection} is uniformly bounded in some neighbourhood of $\tau_0.$
Thus $\mc L$ can at most have one vanishing eigenvalue at $\tau_0$. Indeed let
\[
H(m,z):=F(m,\widetilde M[m],z) = M[m]+T_z(z+ \Gamma[M[m]]T_z)^{-1},
\] 
where $M[m]$ is given by the right hand side of \eqref{sol DE}.
The spectral parameter $z$ is also uniquely defined by $m$ (see \eqref{scalar equation}), i.e. $H(m_,z(m))=0$
with 
\[
z(m) := -\frac{1}{m}-\gamma(m).
\]
Taking the derivative with respect to $m$ we find 
\begin{equation}
0=\frac{\mr d}{\mr dm}H(m,z(m))
=\partial_{m}H+\frac{\mr d z}{\mr d m}\partial_z H.
\end{equation}
By Corollary~\ref{cor m reg} the derivative $\frac{\mr d z}{\mr d m}$ vanishes at the edge. Therefore at $z=\tau_0$ the above equation simplifies to
\begin{equation}
\partial_{m}H=0.
\end{equation}
Calculating this we obtain 
\begin{equation} \label{critical direction} 
0=M'-M\Gamma [M']M=\mc L[M'],
\end{equation}
with $M'=M'[m]$ the derivative $M$ with respect to $m$ and we used $H=0$. Therefore $B=M'[m]$ is the critical right eigenvector at $\tau_0$.
Next, we obtain the critical left eigenvector $L$. The adjoint of the stability operator $\mc L$ is 
\begin{equation}
    \mc L^*[R]=R-\Gamma[M^*RM^*].
\end{equation}
At any regular edge, we have $\Im M=0$, i.e. $M=M^*.$ Therefore
\begin{equation}
    0=\Gamma [\mc L[B]]=\Gamma[B]-\Gamma[\mc C_{M}[\Gamma[B]]]=\mc L^*[\Gamma[B]]
\end{equation}
at $z=\tau_0$. 
Thus the critical left eigenvector is given by
\begin{equation} \label{L=GammaB}
    L=\Gamma[B].
\end{equation}
Note that $L\neq0$ since $L_{ii}=B_{11}=1$ for all $1<i\leq l+1$ by the expression for $\Gamma$ in \eqref{Gamma formula}.
As $L$ and $B$ belong to the same non-degenerate eigenvalue they cannot be orthogonal and since $\mc L$ does not depend on $N$ they satisfy
\begin{equation}
|\langle L,B\rangle|\sim1
\label{overlap}
\end{equation}
at $\tau_0$. Equation~\eqref{overlap} is also satisfied in a $u$ neighbourhood of $\tau_0$, with $u\sim1$, since $L$ and $B$ vary continuously in $z$.

To obtain $|\langle L,M\Gamma[B]B\rangle|$ we calculate the second total derivative
\begin{equation}
\begin{split}
    0&=\frac{\mr d^2}{\mr d {m}^2}H(m,z(m))\\
&=\partial_{m}^2H(m,z(m))+\frac{\mr d^2 z}{\mr d m^2}\partial_z H(m,z(m))\\
&\quad+\frac{\mr d z}{\mr d m}\left(\frac{\mr d z}{\mr d m}\partial_z^2H(m,z(m))+2\partial_{m}\partial_zH(m,z(m))\right).
\end{split}
\end{equation}
At the edge, $\frac{\mr d z}{\mr d m}$ once again vanishes, whereas $\frac{\mr d^2 z}{\mr d m^2}=c_0\neq0$ does not by Corollary~\ref{cor m reg} and Lemma~\ref{lemma m delta asymptotics}. Thus calculating the derivatives we arrive at
\begin{equation}
\begin{split}
    0&=-c_0M JM+M''-M\Gamma[M'']M-2M\Gamma[M']M\Gamma[M']M\\
    &=-c_0M JM+\mc L[M'']-2M\Gamma[M']M',
\end{split}
\end{equation}
where we used $H=0$ multiple times and \eqref{critical direction} in the last step.
Next, we solve for the last term, use $M'=B$ and take an inner product with $L$:
\begin{equation}
\begin{split}
\langle L,M\Gamma[B]B\rangle&=\frac12\left(\langle L,\mc L[M'']\rangle-c_0\langle L,M JM\rangle\right)\\
&=\frac12\left(\langle \mc L^*[L],M''\rangle-c_0\langle M LM,J\rangle\right)=-\frac{c_0}{2}B_{11}\neq0.
\end{split}
\label{eq LMGammaBB}
\end{equation}
In the last equality, $M LM = M \Gamma[B]M =B$ was used.
By continuity the relation $|\langle L,M\Gamma[B]B\rangle|\sim1$ also holds in a $u$ neighbourhood of $\tau_0$ with $u\sim1$. At $\tau_0$, the operator $\mc L$ is independent of $\delta$ and therefore its vanishing eigenvectors $L$ and $B$ are as well. Consequently, \eqref{overlap} and \eqref{eq LMGammaBB} also hold for all $\delta\in[0,1]$ and as both expressions are also continuous in $z$ for all $\delta$, they also hold in some order one neighborhood of $\tau_0.$

Next we study how $\mc L$ varies in $z$ around $\tau_0$ for arbitrary $\delta\in[0,1]$ and we make the dependence explicit by writing $\mc L=\mc L^z$, $M_\delta=M^z_\delta$, etc. and define $\mc E^z:=\mc L^z-\mc L^{\tau_0}.$ As the eigenvalues of $\mc L^z$ depend continuously on $z$, there is a $u>0$ such that $\mc L^z$ has an isolated small eigenvalue $\beta^z$ for all $|z-\tau_0|<u.$ 
By perturbation theory $\beta_z$ is given by 
\begin{equation}
     \beta^z
    =(\langle L^{\tau_0},B^{\tau_0}\rangle)^{-1}\langle L^{\tau_0},\IE^z(B^{\tau_0})\rangle+\mathcal{O}(\Vert\IE^z\Vert^2_{\mr{sp}}).
    \label{eigenvalue beta_z}
\end{equation}
To estimate the right hand side, we first evaluate $\mc E^z[B^{\tau_0}]$ and find
\begin{equation}
    \mc E^z[B^{\tau_0}]=\mc L^z[B^{\tau_0}]-\mc L^{\tau_0}[B^{\tau_0}]
    =\mc C_{M^z_\delta}[L^{\tau_0}]-C_{M^{\tau_0}_\delta}[L^{\tau_0}]
    =M^{\tau_0}_\delta L^{\tau_0}(M^{\tau_0}_\delta-M^z_\delta)+(M^{\tau_0}_\delta-M^z_\delta)L^{\tau_0}M^z_\delta.
\end{equation}
We take the scalar product with $L^{\tau_0}$ and use the cyclic invariance of the trace to obtain
\begin{equation}
\begin{split}
    |\langle L^{\tau_0},\mc E^z[B^{\tau_0}]\rangle|
    &=|\langle L^{\tau_0}(M^{\tau_0}_\delta+M^z_\delta)L^{\tau_0}(M^{\tau_0}_\delta-M^z_\delta)\rangle|\\
    &=2|\langle L^{\tau_0}M^{\tau_0}_\delta L^{\tau_0}(M^{\tau_0}_\delta-M^z_\delta)\rangle|+\mc O(|m^z_\delta-m_\delta^{\tau_0}|^2)\\
    &=2|\langle L^{\tau_0}M^{\tau_0}_\delta L^{\tau_0}(M^{\tau_0}_0)'\rangle(m_\delta^{\tau_0}-m_\delta^z)|+\mc O(|m_\delta^z-m_\delta^{\tau_0}|^2+\eta\delta)\\
    &=2|\langle L^{\tau_0},M^{\tau_0}_\delta\Gamma[B^{\tau_0}]B^{\tau_0}\rangle||m_\delta^{\tau_0}-m_\delta^z|+\mc O(|m_\delta^z-m_\delta^{\tau_0}|^2+\eta\delta)\\
    &\sim |m_\delta^{\tau_0}-m_\delta^z|\sim\sqrt{\kappa+\eta}.
\end{split}
\label{scalar product L, mcE}
\end{equation}
Here, the third equality follows from the fact that $M_\delta[m_\delta]=M_0(m_\delta)+\mc O(\eta\delta)$ and $M_0$ is analytic in $m_\delta$ at $m_\delta^{\tau_0}=m^{\tau_0}$. In the fourth line we used \eqref{L=GammaB}, $M'=B$  and in the fifth line that $|\langle L^{\tau_0},M^{\tau_0}_\delta\Gamma[B^{\tau_0}]B^{\tau_0}\rangle|$ is non-vanishing. In the last relation, we used Corollary~\ref{cor m reg} and Lemma~\ref{lemma m delta asymptotics}. Taking the absolute value \eqref{eigenvalue beta_z}, the asymptotic behaviour of $|\beta^z|$ follows from \eqref{scalar product L, mcE} and \eqref{overlap}, i.e.
\begin{equation}
     |\beta^z|
    \sim\sqrt{\kappa+\eta}+\mathcal{O}(\Vert\IE^z\Vert^2_{\mr{sp}})
    \sim \sqrt{\kappa+\eta}.
\end{equation}
In the last step, we used that
\begin{equation}
    \Vert\IE^z\Vert_{\mr{sp}}=\mc O(|m_\delta^z-m_\delta^{\tau_0}|).
\end{equation} 
\end{proof}

\section{Proof of the local law\label{proof LL}}

In this section, we prove Theorem~\ref{thm LL} and Proposition~\ref{prop LL away}. For shifted squares of Wigner matrices (see Definition~\ref{def shifted square}), we provide a direct proof below in Lemma~\ref{lemma Wigner square}. The Stieltjes transform $m$ and the resolvent $\mb g$ are submatrices of $M_0$ and $\mb G_0$, introduced in \eqref{sol DE} and \eqref{resolvent delta}, respectively. For polynomials that are not shifted squares of Wigner matrices, the local laws are therefore a direct consequence of Proposition~\ref{prop LLL} and Proposition~\ref{prop LLL away}, the local laws for the linearization around regular edges and away from the spectrum, and we will spend most of the section proving them. The main steps of our proof for the edge local law, Proposition~\ref{prop LLL}, follow the general strategy of \cite[Proposition~3.3]{Alt_2020edge}. Here, we briefly describe the main ideas of the proof. Throughout this section we will use the notation $\mb \Delta_\delta:=\mb G_\delta-M_\delta$. First, we establish a global law away from the spectrum, stated below in Proposition~\ref{prop global law}. Then, we use the global law as a starting point for a bootstrapping process. The bootstrapping proposition, Proposition~\ref{prop bootstrapping}, establishes a local law iteratively on scales ever closer to the optimal scale, $\eta\sim N^{-1+\veps}$. Lemmata~\ref{lemma a priori}, \ref{lemma quad eq} and \ref{lemma quad to lin} are auxiliary results used in the bootstrapping process. Lemma~\ref{lemma quad eq} establishes a bound for $\mb \Delta_\delta$ in terms of the error $\mb D$ and $\Theta_\delta$, the projection of $\mb \Delta_\delta$ onto its unstable direction, as well as an approximate quadratic equation for $\Theta_\delta$. Lemma~\ref{lemma quad to lin} transforms the quadratic bound on $\Theta_\delta$ into a linear bound. The most crucial difference between our proof and that of \cite[Proposition~3.3]{Alt_2020edge} is addressed in Lemma~\ref{lemma a priori}. It provides a naive upper bound for $\mb G_0^*\mb G_0$ and allows us to estimate $\mb G_0^*\mb G_0$ in terms of $\mb G_1^*\mb G_1$. Unlike $\mb G_0$, the matrix $\mb G_1$ is a resolvent and thus satisfies $\mb G_1^*\mb G_1=\eta^{-1}\Im \mb G_1$. As $\mb D$ is in turn bounded by $\mb G_\delta^*\mb G_\delta$ this step is necessary to obtain the correct upper bound for $\mb G_0-M_0$. The proof of the local law away from the spectrum, Proposition~\ref{prop LLL away}, makes use of a similar but more simple strategy since $\IL$ does not have an unstable direction away from $\supp\rho$.

\begin{lemma}
    Let  $q$ be a shifted square of a Wigner matrix as introduced in Definition~\ref{def shifted square} and let $\rho$ have a regular edge at $\tau_0$.
    \begin{enumerate}
        \item There is a $\kappa_0>0$, depending only on the coefficients of $q$, such that for all $\veps,\gamma,D>0$ and $z\in\mathbb{D}_{\gamma}^{\kappa_0}$ the isotropic local law \eqref{eq ll iso} holds for all deterministic $\mb x,\mb y\in\C^{N}$ and the averaged local law \eqref{eq ll av} holds for all deterministic $\mb B\in\C^{N\times N}$. If additionally $E\notin\supp(\rho),$ we also have the improved local law \eqref{eq ll imp}.
        \label{lemma Wigner square 1}
        \item For all $C>0$ there is an $\eta_0>0$, depending only on the coefficients of $q$, such that the averaged local law \eqref{eq ll away} holds true for all deterministic $\mb B\in\C^{N\times N}$, $\gamma,\veps,D>0$ and $z\in\mathbb G_{\gamma}^{C,\eta_0}$.
        \label{lemma Wigner square 2}
    \end{enumerate}
\label{lemma Wigner square}     
\end{lemma}
\begin{proof}
By assumption,  $q$ is a shifted square of a Wigner matrix. Thus it is of the form
\begin{equation}
    q(\mb X)=q_{a,c}(\mb W):=a\mb W^2+c,
    \label{eq Wigner square}
\end{equation}
with $a,c\in \R$, $a\neq0$ and $\mb W$ being a Wigner matrix. By $\mb g_{a,c}$ we denote the resolvent of \eqref{eq Wigner square} and by $m_{a,c}$ the solution of \eqref{scalar equation} for \eqref{eq Wigner square} with explicitly stated dependence on $a,c$. We have
\begin{equation}
    \mb g_{a,c}(z)=\frac{1}{a\mb W^2+c-z}=\frac1a\mb g_{1,0}\left(\frac{z-c}{a}\right)
\end{equation}
and
\begin{equation}
    m_{a,c}(z)=\frac1am_{1,0}\left(\frac{z-c}{a}\right).
\end{equation}
Therefore $q_{a,c}$ has a regular edge at $a\tau_0+c$ if and only if $q_{1,0}$ has a regular edge at $\tau_0$ and Theorem~\ref{thm LL} for general $a,c$ follows from $a=1$, $c=0$. Thus let w.l.o.g. $a=1$, $c=0$ and we drop the subscripts again from $m$ and $\mb g$.

For $\zeta\in \HH$  let $\mb g_{\mb W}:=\frac{1}{\mb W-\zeta}$ be the resolvent of $\mb W$ at spectral parameter $\zeta$.
We have
\begin{equation}
    \mb g(z)
    =\frac{1}{2\sqrt{z}}\left(\mb g_{\mb W}(\sqrt{z})+\mb g_{-\mb W}(\sqrt{z})\right),
    \label{eq square g}
\end{equation}
where again the square root function is chosen such that the positive real axis is mapped to itself and with a branch cut along the negative real axis. Let $\rho_{\mr{sc}}(x):=\frac{1}{2\pi}\sqrt{(4-x^2)_+}$ with $(y)_+:=\max\{y,0\}$ be the semi-circle density and $m_{\mr{sc}}$ its Stieltjes transform. By explicitly solving \eqref{scalar equation} for $q=\mb W^2$  and $m_{\mr{sc}}$ we find
\begin{equation}
    m(z)
    =\frac{1}{\sqrt{z}}\left(-\sqrt{z}+\sqrt{z-4}\right)
    =\frac{1}{\sqrt{z}}m_{\mr{sc}}(\sqrt{z}).
    \label{eq square m}
\end{equation}
for all $z\in\HH$. The semi-circle density $\rho_{\mr{sc}}$ has regular edges at $\pm2$, therefore $\rho$ has exactly one regular edge at $\tau_0=4$. Its left edge is a hard edge and has a singularity at zero. Combining \eqref{eq square g} and \eqref{eq square m} we find 
\begin{equation}
    \mb g(z)-m(z)=\frac{1}{2\sqrt{z}}\left((\mb g_{\mb W}(\sqrt{z})-m_{\mr{sc}}(\sqrt{z}))+(\mb g_{-\mb W}(\sqrt{z})-m_{\mr{sc}}(\sqrt{z}))\right).
\end{equation}
Note that $-\mb W$ is also a Wigner matrix. Therefore, Statement~\ref{lemma Wigner square 1} of Lemma~\ref{lemma Wigner square} around $\tau_0=4a+c$ follows from \cite[Theorem~2.6]{Alt_2020edge} and Statement~\ref{lemma Wigner square 2} of Lemma~\ref{lemma Wigner square} Proposition~\ref{prop LL away} from \cite[Theorem~2.1]{ERDOS_2019SCD}.
\end{proof}
From now on we assume that $q$ is not a shifted square of a Wigner matrix. To prove Proposition~\ref{prop LLL} and Proposition~\ref{prop LLL away}, we will require several intermediate results, which we state below.
Throughout the remainder of the section, we assume w.l.o.g. that there is some $\alpha>0$ such that
\begin{equation}
    \Vert\mb X_i\Vert\lesssim N^\alpha.
    \label{boundedness assumption}
\end{equation}
The assumption can be removed by a standard argument using Chebyshev's inequality and our moment assumption \eqref{distribution assumption}.

We also introduce stochastic domination, a commonly used notation used to state high probability bounds in a way well adapted to our needs. It has first been introduced in a slightly different form in \cite{Erdos_2013Averaging}, for the form stated here see e.g. \cite{Erdos_2013Local}.
\begin{definition}[Stochastic domination]
Let $X=(X^{(N)})_{N\in\N}$ and $Y=(Y^{(N)})_{N\in\N}$ be families of non-negative random variables. We say $X$ is stochastically dominated by $Y$ if for all (small) $\veps>0$ and (large) $D>0$
\begin{equation}
    \Pb\left(X^{(N)}\geq N^\veps Y^{(N)}\right)\lesssim_{D,\veps} N^{-D}.
\end{equation}
We denote this relation by $X\prec Y$. If the constant in the definition depends on any other parameters $\alpha$, we write $X\prec_\alpha Y$.
\end{definition}
Furthermore, we introduce the norm $\Vert\cdot\Vert_*:=\Vert\cdot\Vert_*^{K,\mb x,\mb y}$ for deterministic $\mb x,\mb y\in\C^{kN}$ and $k,K\in\N$. Our definition follows that of \cite{Alt_2020edge} and it goes back to \cite{ERDOS_2019SCD} and we refer to those two works for details. First, we define the set
\begin{equation}
    I_0:=\{\mb x,\mb y\}\cup\{\mb e_a,L^*_{a\cdot}:\,a\in\llbracket kN\rrbracket\},
\end{equation}
where $\mb e_a$ is the $a^{\mr{th}}$ standard base vector. Replacing a scalar index by a dot denotes the vector that runs over the entire range of the index, e.g. $\mb R_{a\cdot}:=(\mb R_{ab})_{b\in\llbracket kN\rrbracket}$ is the $a^{\mr{th}}$ row vector of a matrix $\mb R$. For $j\in\N$ we define the set $I_j$ recursively by
\begin{equation}
    I_{j+1}:=I_j\cup\{M_\delta\mb u:\,\mb u\in I_j\}\cup\{\kappa_c((M_\delta\mb u)a,b\cdot),\kappa_d((M_\delta\mb u)a,\cdot b):\,\mb u\in I_j,a,b\in\llbracket kN\rrbracket\},
    \label{def I_j}
\end{equation}
where $\kappa(ab,cd)$ denotes the cumulant of the $(a,b)$ entry and the $(c,d)$ entry of $\sum_{j=1}^lK_j \otimes \mb X_j$ and $\kappa_c$ and $\kappa_d$ denote the decomposition of $\kappa$ into its direct and its cross contribution according to the Hermitian symmetry. In \eqref{def I_j} we also use the shorthand notation $\kappa(\mb xb,cd):=\sum_a x_a\kappa(ab,cd)$.
Then $\Vert\cdot\Vert_*$ is defined as
\begin{equation}
    \Vert\mb R\Vert_*:=\sum_{0\leq j\leq K}N^{-\frac{j}{2K}}\Vert \mb R\Vert_{I_j}+N^{-\frac12}\max_{\mb u\in I_K}\frac{\Vert \mb R \mb u\Vert}{\Vert \mb u\Vert},
    \quad
    \Vert \mb R\Vert_{I_j}:=\max_{\mb u,\mb v\in I_j}\frac{|\langle \mb v,\mb R\mb u\rangle|}{\Vert \mb u\Vert\Vert \mb v\Vert}
\end{equation}

The notion of stochastic domination is closely related to bounds in p-norms as can be seen from the following lemma.
\begin{lemma}[{\cite[Lemma~5.4]{ERDOS_2019SCD}}]
Let $\mb R\in C^{kN\times kN}$, $k\in\N$, be a random matrix and $\Phi$ be a stochastic control parameter. Then the following holds true
\begin{enumerate}
\item If $\Phi\gtrsim N^{-C}$ and $\Vert\mb R\Vert\lesssim N^C$ for some $C>0$ and $|\langle\mb x,\mb R\mb y\rangle|\prec\Phi\Vert \mb x\Vert\Vert \mb y\Vert$ for all $\mb x,\mb y$, then $\Vert \mb R\Vert_p\lesssim_{\veps,p} N^\veps \Phi$ for all $\veps>0$, $p\in\N$.
\item If $\Vert \mb R\Vert_p\lesssim_{\veps,p} N^\veps \Phi$ for all $\veps>0$, $p\in\N$ then $\Vert \mb R\Vert_*^{K,\mb x,\mb y}\prec\Phi$ for any fixed $K\in\N$ and $\mb x,\mb y\in\C^{N}$.
\end{enumerate}
\label{lemma stoch dom}
\end{lemma}
    
\begin{lemma}[A priori bound] 
$\mb G_0$ satisfies the a priori bound
\begin{equation}
\mb G_0^*\mb G_0\lesssim\frac{1}{\eta^2}(1+\Vert \mb X\Vert^4)
\end{equation}
uniformly in $z\in\HH$ with $\eta\lesssim1$.
\label{lemma a priori}
\end{lemma}
\begin{proof}
 We recall that $\mb G_0$ is given by
\begin{equation}
\mb G_0=
\begin{pmatrix}
\mb g & \mb g\mb X^tA\\
A\mb X\mb g & -A+A\mb X\mb g\mb X^tA
\end{pmatrix}.
\end{equation}
We estimate $\mb G_0$ blockwise to obtain
\begin{equation}
\mb G_0\mb G_0^*\leq\Vert \mb G_0\Vert ^2
\lesssim\Vert \mb g\Vert^2(1+\Vert\mb X\Vert^4)
\leq\frac{1}{\eta^2}(1+\Vert\mb X\Vert^4).
\end{equation}
The last inequality holds because $\mb g$ is a resolvent.
\end{proof}

The lemma is used in the last step of the following estimate on $\mb G_0^*\mb G_0$,
\begin{equation}
\begin{split}
    \mb G_0^*\mb G_0
    &\leq2(\mb G_1^*\mb G_1+(\mb G_0-\mb G_1)^*(\mb G_0-\mb G_1))\\
    &=2(\mb G_1^*\mb G_1+\eta^2 \mb G_1^*(I_{l+1}-J)\mb G_0^*\mb G_0(I_{l+1}-J)\mb G_1)
    \lesssim (1+\Vert \mb X\Vert^4)\mb G_1^*\mb G_1.
\label{Eq G_0 est}
\end{split}
\end{equation}
In the first inequality, we applied Lemma~\ref{Lemma Matrix ineq} from the appendix with $R=\mb G_1$ and $T=\mb G_0-\mb G_1$. In the second step, we used
\begin{equation}
    \mb G_0-\mb G_1=\mb G_0(\mb G_1^{-1}-\mb G_0^{-1})\mb G_1=-\I\eta \mb G_0(I_{l+1}-J)\mb G_1.
\end{equation}
Taking the p-norm on both sides of \eqref{Eq G_0 est} and applying Hölder's inequality we obtain
\begin{equation}
\Vert \mb G_0^*\mb G_0\Vert_p
\lesssim_p (1+\Vert\Vert \mb X\Vert^4\Vert_{2p})\Vert \mb G_1^*\mb G_1\Vert_{2p}
\lesssim_p\Vert \mb G_1^*\mb G_1\Vert_{2p}=\frac{\Vert \Im \mb G_1\Vert_{2p}}{\eta}
\label{eq ward substitution}
\end{equation}
 for all $p\in\N$. Here we use $\Vert Y\Vert_p:=(\E[|Y|^p])^{1/p}$ for scalar random variables $Y$. In the second to last step, we used $\Vert \mb X\Vert_p\lesssim_p1$, which follows from Lemma~\ref{lemma stoch dom} since $\Vert \mb X\Vert\prec1$ and $\Vert \mb X\Vert\leq N^C$ by assumption \eqref{boundedness assumption}.

We define the sets
\begin{equation}
    \mathbb D^{\kappa_0}:=\{z\in\HH:\,|E-\tau_0|\leq\kappa_0\}
    ,\quad
    \mathbb G^{C,\eta_0}:=\{z\in\HH:\,C^{-1}<\dist(E,\supp(\rho))<C,\eta\leq\eta_0\}
\end{equation}
and the random variable
\begin{equation}
    \Theta_\delta:= \frac{\langle L\otimes \mb I_N,\mb \Delta_\delta\rangle}{\langle L\otimes \mb I_N,B\otimes \mb I_N\rangle}.
\end{equation}
With $\Theta_\delta$ we can write the projection of $\mb \Delta_\delta$ onto the critical direction (see \eqref{eq mathscr P}) as $\mathscr P[\mb \Delta_\delta] = \Theta_\delta B$. Furthermore, let $\chi(A)$ denote the indicator function on event $A$.
We import following lemmata from \cite[Proposition~3.3]{Alt_2020edge} and \cite[Lemma~3.9]{Alt_2020edge}.
\begin{lemma}[{\cite[Proposition~3.3]{Alt_2020edge}}]
Let $\delta\in\{0,1\}$ and $\Vert \cdot\Vert_*=\Vert\cdot\Vert_*^{K,\mb x,\mb y}$. There is a $\kappa_0\sim1$ and deterministic matrices $\mb R_i$ with $\Vert\mb R_i\Vert\lesssim1$ for $i=1,2$ such that 
\begin{equation}
\mb \Delta_\delta\chi(\Vert \mb \Delta_\delta\Vert_*\leq N^{-\frac3K})=(\Theta_\delta B-\IL^{-1}[(\mathbb 1-\mathscr P)[M_\delta\mb D]]+\mathcal{E})\chi(\Vert \mb \Delta_\delta\Vert_*\leq N^{-\frac3K}),
\label{Eq Delta from Theta}
\end{equation}
with an error function $\mathcal{E}$ of size 
\begin{equation}
\Vert \mathcal{E}\Vert_*=\mathcal{O} \left(N^{\frac2K}(|\Theta_\delta|^2+\Vert \mb D\Vert_*^2)\right)
\end{equation}
and $\Theta_\delta$ satisfying the approximate quadratic equation
\begin{equation}
\left(\xi_1\Theta_\delta+\xi_2\Theta_\delta^2\right)\chi(\Vert \mb \Delta_\delta\Vert_*\leq N^{-\frac3K})=\mathcal{O}\left(N^{\frac2K}\Vert \mb D\Vert_*^2+|\langle\mb R_1,\mb D\rangle|+|\langle\mb R_2,\mb D\rangle|\right)
\end{equation}
with $\xi_1\sim\sqrt{\kappa+\eta}$, $\xi_2\sim1$ uniformly in $\mb x,\mb y\in\C^{(l+1)N}$ and $z\in\mathbb{D}^{\kappa_0}$.
\label{lemma quad eq}
\end{lemma}
\begin{proof}
    The proof of \cite[Proposition~3.3]{Alt_2020edge} uses \cite[Proposition~3.1]{Alt_2020edge} as an input. After replacing \cite[Proposition~3.1]{Alt_2020edge} by our analogous Proposition~\ref{prop stab} the proof follows theirs line by line.
\end{proof}
\begin{lemma}[{\cite[Lemma~3.9]{Alt_2020edge}}]
Let $d=d(\eta)$ be a monotonically decreasing function in $\eta\geq N^{-1}$ and assume $0\leq d\lesssim N^{-\veps}$ for some $\veps>0$. Suppose there are $\kappa_0,\gamma>0$ such that
\begin{equation}
|\xi_1\Theta_\delta+\xi_2\Theta_\delta^2|\lesssim d
\text{ for all } z\in\mathbb{D}_\gamma^{\kappa_0}
\quad
\text{and}
\quad
|\Theta_\delta|\lesssim\min\left\{\frac{d}{\sqrt{\kappa+\eta}},\sqrt{d}\right\}
\text{for some } z_0\in\mathbb{D}_\gamma^{\kappa_0}
.
\end{equation}
Then also $|\Theta_\delta|\lesssim\min\{d/\sqrt{\kappa+\eta},\sqrt{d}\}$ for all $z'\in\mathbb{D}_\gamma^{\kappa_0}$ with $\Re z'=\Re z_0$ and $\Im z'\leq\Im z_0$.
\label{lemma quad to lin}
\end{lemma}
\begin{prop}[Global Law]
For all $C>0$ and some $\kappa_0>0$ with $\kappa_0\sim1$  there is an $\eta_0>0$ such that for all $\delta\in[0,1]$, $z\in\mathbb{D}^{\kappa_0}\cup\mathbb G^{C,\eta_0}$ and $\veps,D>0$ the isotropic global law,
\begin{equation}
\Pb\left(|\langle \mb x, \mb \Delta_\delta\mb y\rangle|>\Vert\mb x\Vert\Vert\mb y\Vert \frac{N^\veps}{\sqrt{N}}\right)
\lesssim_{\veps,D,\eta,C}N^{-D}
\end{equation}
holds for all deterministic $\mb x,\mb y\in\C^{(l+1)N}$. Additionally, we have a global averaged law,
\begin{equation}
\Pb\left(|\langle \mb B\mb \Delta_\delta\rangle|>\Vert\mb B\Vert\frac{N^\veps}{N}\right)\lesssim_{\veps,D,\eta,C}N^{-D},
\end{equation}
for all deterministic $\mb B\in\C^{(l+1)N\times(l+1)N}$.
\label{prop global law}
\end{prop}
\begin{proof}
    First, consider $\delta=1$ and $A$ being invertible. Then for any $z\in\HH$ the matrix $\mb G_1$ is a resolvent of the Hermitian random matrix $\mb L-EJ$ with spectral parameter $\I\eta$. In this regime the global law is covered by \cite[Theorem~2.1]{ERDOS_2019SCD}. 
    
    Now consider $A$ to be non-invertible. Then there is a $u>0$ s.t. $A+\veps$ is invertible for all $0<\veps< u$ and a global law for $A+\veps$ also follows from \cite[Theorem~2.1]{ERDOS_2019SCD}. For $C>0$ fix a $z\in\mathbb{D}^{\kappa_0}\cup\mathbb G^{C,\eta_0}$ with $\kappa_0,\eta_0$ sufficiently small for Proposition~\ref{prop stab} to be applicable. $\IL$ is continuous in $\veps\in[0,u]$ for sufficiently small $u\sim1$ and Proposition~\ref{prop stab} holds true uniformly for sufficiently small $\veps$. As $M_1$ and $\mb G_1$ are also continuous in $\veps\in[0,u]$ for some $u\sim1$, the proposition follows from a stochastic continuity argument.

    For $\delta\in[0,1)$ and arbitrary $A$ the global law follows from another stochastic continuity argument as $M_\delta$, $\mb G_\delta$ and $\IL$ are also continuous in $\delta$.
\end{proof}

Before starting our bootstrapping argument, we introduce the following notions for isotropic and averaged stochastic dominance for random matrices $\mb R$ and deterministic control parameters $\Lambda$,
\begin{equation}
    \begin{split}
    |\mb R|\prec\Lambda\text{ in } \mathbb D
    \Leftrightarrow \Vert\mb R\Vert_*^{K,\mb x,\mb y}\prec\Lambda \text{ uniform in }\mb x,\mb y\text{ and }z\in\mathbb D\\
    |\mb R|_{\mr{av}}\prec\Lambda\text{ in } \mathbb D
    \Leftrightarrow \frac{|\langle \mb B\mb R\rangle|}{\Vert \mb B\Vert}\prec\Lambda \text{ uniform in }\mb B\neq0\text{ and }z\in\mathbb D.
    \end{split}
\end{equation}
\begin{prop}[Bootstrapping]
Assume the following:
\\[5pt]
1. The isotropic and averaged local laws,
    \begin{equation}
    |\bm \Delta_\delta|\prec  N^{\frac2K}\left(\sqrt{\frac{\Im m}{N\eta}}+\frac{N^{\frac2K}}{N\eta}\right)
    ,\quad
    |\bm \Delta_\delta|_{\mr{av}}\prec 
    \begin{cases}
    \displaystyle{\frac{N^{\frac2K}}{N\eta}},& E\in\mr{supp}{\rho}, \\[10pt]
    \displaystyle{\frac{N^{\frac2K}}{N(\eta+\kappa)}+\frac{N^{\frac4K}}{N^2\eta^2\sqrt{\eta+\kappa}}},& E\notin\mr{supp}{\rho},
    \end{cases}
    \label{eq boot strapping}
    \end{equation}
    hold on $z = E + \I \eta \in \mathbb{D}_{\gamma_0}^{\kappa_0}$ for some $\gamma_0,\kappa_0, K$ and $\delta\in\{0,1\}$.
    \label{bootstrapping 1}
    \\[5pt]
2. The isotropic and averaged local laws,
    \begin{equation}
    |\bm \Delta_\delta|\prec_C N^{\frac2K}\left(\frac{1}{\sqrt{N}}+\frac{N^{\frac2K}}{N\eta}\right)
    ,\quad
    |\bm \Delta_\delta|_{\mr{av}}\prec_C
    \left(\frac{1}{N}+\frac{N^{\frac4K}}{(N\eta)^2}\right)
    \label{eq boot strapping away}
    \end{equation}
    hold on $z = E + \I \eta \in \mathbb{G}_{\gamma_0}^{C,\eta_0}$ for some $\gamma_0,C,\eta_0, K$ and $\delta\in\{0,1\}$.
    \label{bootstrapping 2}

 Then, for all $\gamma>0$ with $\gamma\geq\frac{100}{K}$ there is a $\gamma_s>0$ independent of $\gamma_0$ such that the local laws \eqref{eq boot strapping} also hold on $\mathbb{D}_{\gamma_1}^{\kappa_0}$ with $\gamma_1=\max\{\gamma,\gamma_0-\gamma_s\}$ and the local laws \eqref{eq boot strapping away} also hold on $\mathbb{G}_{\gamma_1}^{C,\eta_0}$ with $\gamma_1=\max\{\gamma,\gamma_0-\gamma_s\}$.
\label{prop bootstrapping}
\end{prop}
\begin{proof}
We first prove the local law \eqref{eq boot strapping} for $z \in \mathbb{D}_{\gamma_1}^{\kappa_0}$ and then comment on the modifications necessary to prove \eqref{eq boot strapping away} for $z \in \mathbb{G}_{\gamma_1}^{C,\eta_0}$.
We begin by proving that $\eta\mapsto\eta\Vert \mb G_\delta\Vert_p$ is monotonically non-decreasing in $\eta$.
For $\delta=1$ the proof is given in the proof  of \cite[Proposition~5.5]{ERDOS_2019SCD}. For $\delta=0$ the proof modifies as follows. Let $\veps>0$. For fixed $E$ we define $f(\eta):=\eta\Vert\mb G_0(E+\I\eta)\Vert_p$. It satisfies
\begin{equation}
\begin{split}
\liminf_{\veps\to0}\frac{f(\eta+\veps)-f(\eta)}{\veps}
&=\liminf_{\veps\to0}\Vert\mb G_0(E+\I(\eta+\veps))\Vert_p+\frac{\eta(\Vert\mb G_0(E+\I(\eta+\veps))\Vert_p-\Vert\mb G_0(E+\I\eta)\Vert_p)}{\veps}\\
&\geq\Vert\mb G_0(E+\I\eta)\Vert_p-\lim_{\veps\to0}\eta\left\Vert\frac{\mb G_0(E+\I(\eta+\veps))-\mb G_0(E+\I\eta)}{\veps}\right\Vert_p\\
&=\Vert\mb G_0(E+\I\eta)\Vert_p-\eta\Vert\mb G_0(E+\I\eta)J\mb G_0(E+\I\eta)\Vert_p
\end{split}
\end{equation}
To obtain a bound for the last term we estimate 
\begin{equation}
\eta|\langle \mb x,\mb G_0J\mb G_0\mb y\rangle|
\leq\frac\eta2(\langle \mb x,\mb G_0J\mb G_0^*\mb x\rangle+\langle\mb y,\mb G_0^*J\mb G_0\mb y\rangle)
=\frac{1}{2}(\langle \mb x,\Im \mb G_0 \mb x\rangle+\langle \mb y,\Im \mb G_0 \mb y\rangle)
\end{equation}
for $\mb x,\mb y\in\C^{(l+1)N}$. In the last step, we used the Ward identity for generalized resolvents,
\begin{equation}
\mb G_0J\mb G_0=\frac{\Im \mb G_0}{\eta}.
\end{equation}
Thus we find
\begin{equation}
\eta\Vert \mb G_0J\mb G_0\Vert_p\leq\sup_{\Vert \mb x\Vert,\Vert \mb y\Vert=1}\left(\E \left(\frac{1}{2}(\langle \mb x,\Im \mb G_0 \mb x\rangle+\langle \mb y,\Im \mb G_0 \mb y\rangle)\right)^p\right)^\frac1p\leq \Vert \mb G_0\Vert_p,
\end{equation}
where we used 
\begin{equation}
|\langle \mb x,\Im \mb R\mb x\rangle|\leq|\langle \mb x, \mb R\mb x\rangle|
\end{equation}
in the last step. Therefore 
\begin{equation}
\liminf_{\veps\to0}\frac{f(\eta+\veps)-f(\eta)}{\veps}\geq0
\end{equation}
and the claim follows.

For $\delta=0,1$, assume the local law \eqref{eq boot strapping} on $\mathbb{D}_{\gamma_0}^{\kappa_0}$. Then $\Vert \mb G_\delta\Vert_{p}\sim_{p}1$ on $\mathbb{D}_{\gamma_0}^{\kappa_0}$ and by monotonicity of $\eta\Vert \mb G_\delta\Vert_p$ we find
\begin{equation}
\Vert \mb G_\delta\Vert_p\lesssim_{p}N^{\gamma_s}.
\end{equation}
on $\mathbb{D}_{\gamma_1}^{\kappa_0}$.
We choose $\gamma_s<\frac14$. Then Proposition~\ref{prop error scd} yields 
\begin{equation}
\Vert \mb D\Vert_p\lesssim_{p,\veps}N^{\veps+c\gamma_s}\sqrt{\frac{\Vert \mb G_\delta \mb G_\delta^*\Vert_q}{N}}
\quad
\text{and}
\quad
\Vert \mb D\Vert_p^{\mr{av}}\lesssim_{p,\veps}N^{\veps+c\gamma_s}\frac{\Vert\mb G_\delta\mb G_\delta^*\Vert_q}{N}.
\end{equation}
For $\delta=0$ estimate the quadratic term by making use of \eqref{eq ward substitution} and then use the Ward identity on $\mb G_1\mb G_1^*$ for both $\delta=0,1$ to obtain
\begin{equation}
\Vert \mb D\Vert_p\lesssim_{p,\veps}N^{\veps+c\gamma_s}\sqrt{\frac{\Vert \Im \mb G_1\Vert_{2q}}{\eta N}}\lesssim_p\frac{N^{\frac{c'}{2}\gamma_s}}{\sqrt{\eta N}}
\quad
\text{and}
\quad
\Vert\mb D\Vert_p^{\mr{av}}\lesssim_{p,\veps}N^{\veps+c\gamma_s}\frac{\Vert\Im\mb G_1\Vert_{2q}}{\eta N}
\lesssim_p\frac{N^{c'\gamma_s}}{\eta N}.
\label{Rough bound D}
\end{equation}
on $\mathbb{D}_{\gamma_1}^{\kappa_0}$ for sufficiently small $\veps>0$ and some modified $c'>0$. Note that we are allowed to exchange the $q$ norm in the bound for $\delta=1$ by a $2q$ norm as the norm is increasing in $q$. After using Lemma~\ref{lemma stoch dom} to turn the p-norm bound in the first equation of \eqref{Rough bound D} into a bound on the *-norm bound, we get from Lemma~\ref{lemma quad eq}
\begin{equation}
|\xi_1\Theta_\delta+\xi_2\Theta_\delta^2|\chi(\Vert \bm \Delta_\delta\Vert_*\leq N^{-\frac3K})\prec \frac{N^{\frac2K+c'\gamma_s}}{\eta N}
\end{equation}
on $\mathbb{D}_{\gamma_1}^{\kappa_0}.$
The left hand side is also Lipschitz continuous with Lipschitz constant $\prec \eta^{-2}\leq N^2$. Thus the inequality
\begin{equation}
    |\xi_1\Theta_\delta+\xi_2\Theta_\delta^2|\chi(\Vert\bm\Delta_\delta\Vert_*\leq N^{-\frac3K})\leq N^{-\frac{10}{K}}
\end{equation}
holds with very high probability on all of $\mathbb{D}_{\gamma_1}^{\kappa_0}$ for sufficiently large $K$ and sufficiently small $\gamma_s$. By our assumption the local law, \eqref{eq boot strapping}, holds on $\mathbb{D}_{\gamma_0}^{\kappa_0}$. In conjunction with another stochastic continuity argument, we also get the bound 
\begin{equation}
    |\Theta_\delta|\leq\min\left\{\frac{N^{-\frac{10}{K}}}{\sqrt{\kappa+\eta}},N^{-\frac5K}\right\}
\end{equation}
on all of $\mathbb{D}_{\gamma_0}^{\kappa_0}$ with very high probability. Therefore Lemma~\ref{lemma quad to lin} can be applied with $d=N^{-10/K}$ and we obtain
\begin{equation}
    |\Theta_\delta|\chi(\Vert\bm\Delta_\delta\Vert_*\leq N^{-\frac3K})\prec N^{-\frac5K}.
    \label{rough bound Theta}
\end{equation}
Next, we turn this into a bound on $\bm \Delta_\delta.$ To do so we again use the *-norm bound obtained from applying Lemma~\ref{lemma stoch dom} to \eqref{Rough bound D} to find $\Vert \mb D\Vert_*\prec N^{-\frac7K}$ for sufficiently small $\gamma_s$ and sufficiently large $K$. Using both bounds on \eqref{Eq Delta from Theta} we get
\begin{equation}
\Vert\bm\Delta_\delta\Vert_*\chi(\Vert\bm\Delta_\delta\Vert_*\leq N^{-\frac3K})
\lesssim\left(|\Theta_\delta|+N^{\frac{2}{K}}\Vert\mb D\Vert_*\right)\chi(\Vert\bm\Delta_\delta\Vert_*\leq N^{-\frac3K})
\prec N^{-\frac5K}.
\label{rough bound Delta}
\end{equation}
We have thus found a ``forbidden'' area for $\Vert\bm\Delta_\delta\Vert_*$. With the aid of a standard stochastic continuity argument, we remove the indicator function from the bounds \eqref{rough bound Theta} and \eqref{rough bound Delta} to get to the rough bounds
\begin{equation}
    \Vert\bm\Delta_\delta\Vert_*\prec N^{-\frac5K}
    \quad
    \text{and}
    \quad
    |\Theta_\delta|\prec N^{-\frac5K}.
\end{equation}
Since $\mb x$ and $\mb y$ were arbitrary we the first bound becomes $|\bm\Delta_\delta|\prec N^{-5/K}.$ Now, assume that $\max\{|\bm\Delta_0|,|\bm\Delta_1|\}\prec\Lambda$ and $\max\{|\Theta_0|,|\Theta_1|\}\prec \theta$ for some deterministic $\theta\leq\Lambda\leq N^{-3/K}$. Then we know from Lemma~\ref{lemma quad eq}  and Proposition~\ref{prop error scd} as well as \eqref{eq ward substitution} that
\begin{equation}
    \max_i\{|\bm\Delta_i|\}\prec \theta +N^{\frac2K}\sqrt{\frac{\Im m+\Lambda}{N\eta}}
    \quad
    \text{and}
    \quad
    \max_i\{|\xi_1\Theta_i+\xi_2\Theta_i^2|\}\prec N^{\frac{2}{K}}\frac{\Im m+\Lambda}{N\eta}
    \label{eq self impro}
\end{equation}
are also deterministic bounds. Here, we also used $\Im m\sim\langle\Im M_1\rangle$ by Corollary~\ref{cor m reg}, Lemma~\ref{lemma m delta asymptotics} as well as \eqref{sol DE}. The first bound in \eqref{eq self impro} self-improving and applying it iteratively yields 
\begin{equation}
    \max_i\{|\bm\Delta_i|\}\prec \theta +N^{\frac2K}\left(\frac{N^{\frac2K}}{N\eta}+\sqrt{\frac{\Im m+\theta}{N\eta}}\right)
\end{equation}
and therefore the second bound in \eqref{eq self impro} improves to
\begin{equation}
    \max_i\{|\xi_1\Theta_i+\xi_2\Theta_i^2|\}\prec N^{\frac{2}{K}}\frac{\Im m+\theta}{N\eta}+N^{\frac4K}\frac{1}{(N\eta)^2}.
\end{equation}
Now we separately treat $\Re z\in\supp(\rho)$ and $\Re z\notin\supp(\rho)$ and we start with the former. Then by Corollary~\ref{cor m reg} we know that $\Im m\sim\sqrt{\kappa+\eta}$. For fixed $\theta$ we apply Lemma~\ref{lemma quad to lin} with 
\begin{equation}
    d=N^{\frac2K}\frac{\sqrt{\kappa+\eta}+\theta}{N\eta}+N^{\frac4K}\frac{1}{(N\eta)^2}
\end{equation}
to obtain
\begin{equation}
    \max_i\{|\Theta_i|\}\prec\min\left\{\frac{d}{\sqrt{\kappa+\eta}},\sqrt{d}\right\}.
\end{equation}
This is also a self-improving bound and iterating it gives
\begin{equation}
    \max_i\{|\Theta_i|\}\prec N^{\frac2K}\frac{1}{N\eta},
    \quad
    \text{hence}
    \quad
    \max_i\{\bm|\Delta_i|\}\prec N^{\frac2K}\left\{\sqrt{\frac{\Im m}{N\eta}}+\frac{N^{\frac2K}}{N\eta}\right\}.
\end{equation}
For $\Re z\notin\supp(\rho)$ we have $\Im m\sim \frac{\eta}{\sqrt{\kappa+\eta}}$, again by Corollary~\ref{cor m reg}. Analogously to the $\Re z\in\supp(\rho)$ we obtain
\begin{equation}
    \max_i\{|\Theta_i|\}\prec N^{\frac2K}\frac{1}{N(\eta+\kappa)}+N^\frac4K\frac{1}{(N\eta)^2\sqrt{\kappa+\eta}}.
\end{equation}
Finally we use \eqref{Eq Delta from Theta}, \eqref{bound error av} and the bounds on $\max_i\{|\Theta_i|\}$ to arrive at the averaged bounds
\begin{equation}
    \max_i\{|\bm\Delta_i|_{\mr{av}}\}\prec
    N^{\frac2K}
    \begin{cases}
    \frac{1}{N\eta}\text{ if } \Re z\in\mr{supp}{\rho_1}\\
    \frac{1}{N(\eta+\kappa)}+\frac{N^{\frac2K}}{N^2\eta^2(\eta+\kappa)^{\frac12}}\text{ if } \Re z\notin\mr{supp}{\rho_1}.
    \end{cases}
\end{equation}
This concludes the proof of \eqref{eq boot strapping} on $\mathbb D_{\gamma_1}^{\kappa_0}$.

Now, assume \eqref{eq boot strapping away} on $\mathbb G_{\gamma_0}^{C,\eta_0}$. From \eqref{eq stab operator} we have
\begin{equation}
    \IL[\mb \Delta_\delta]=-M_\delta\mb D+M_\delta\mc S[\mb \Delta_\delta]\mb\Delta_\delta.
    \label{IL of Delta}
\end{equation}
We apply $\IL^{-1}$ to the equation and take its *-norm for some deterministic $\mb x,$ $\mb y$ to find
\begin{equation}
\begin{split}
   \Vert \mb \Delta_\delta\Vert_*
   &\leq\Vert\IL^{-1}[M_\delta\mb D]\Vert_*+\Vert\IL^{-1}[M_\delta\mc S[\mb \Delta_\delta]\mb \Delta_\delta]\Vert_*\\
   &\leq\Vert\IL^{-1}\Vert_{*\to*}\left(\Vert M_\delta\mb D\Vert_*+\Vert M_\delta\mc S[\mb \Delta_\delta]\mb \Delta_\delta\Vert_*\right)\\
   &\lesssim_C N^{\frac2K}\Vert \mb D\Vert_*+N^{\frac2K}\Vert \mb \Delta_\delta\Vert_*^2.
   \label{Delta * est}
\end{split}
\end{equation}
Here, $\Vert\IL^{-1}\Vert_{*\to*}$ denotes the operator norm of $\IL^{-1}$ with respect to the *-norm.
In the last estimate, we have used $\Vert\IL^{-1}\Vert_{*\to*}\lesssim\Vert\IL^{-1}\Vert_{\mr{sp}}\lesssim_C1$ by \cite[Equation~(70c)]{ERDOS_2019SCD} and Proposition~\ref{prop stab} as well as $\Vert M_\delta\mb R\Vert_*\lesssim N^{2/K}\Vert \mb R\Vert_*$ and $\Vert M_\delta\mc S[\mb R]\mb R\Vert_*\lesssim N^{2/K}\Vert \mb R\Vert_*^2$ by \cite[Equation~(70a) and (70b)]{ERDOS_2019SCD}.
Equation \eqref{rough bound Delta} also holds on $\mathbb G_{\gamma_1}^{C,\eta_0}$ by the same argument as in the case $z\in \mathbb{D}_{\gamma_1}^{\kappa_0}$
 and for sufficiently large $K$ and sufficiently small $\gamma_s$ we have in particular $\Vert \mb D\Vert_*\prec N^{-\frac7K}$ for $\delta\in\{0,1\}$. We use this bound on with \eqref{Delta * est} to estimate
\begin{equation}
    \Vert \mb \Delta_\delta\Vert_*\chi(\Vert \mb \Delta_\delta\Vert_*\leq N^{-\frac3K})
    \prec_C N^{-\frac5K}.
\end{equation}
By a stochastic continuity argument, we establish the rough bound
\begin{equation}
    \Vert \mb \Delta_\delta\Vert_*
    \prec_C N^{-\frac5K}
    \quad
    \text{and therefore}
    \quad
    |\mb \Delta_\delta|\prec_C N^{-\frac5K}
\end{equation}
since $\mb x$ and $\mb y$ were arbitrary. Now, assume $\max_i\{|\bm\Delta_i|\}\prec_C\Lambda$ for some deterministic $\Lambda\leq N^{-3/K}$. Then we have another deterministic bound in the form of
\begin{equation}
    \max_i\{|\bm\Delta_i|\}\prec_C N^{\frac{2}{K}}\sqrt{\frac{\Im m +\Lambda}{N\eta}}.
\end{equation}
Iterating this self-improving bound and using $\Im m\sim_C\eta$ from Lemma~\ref{lemma m away from rho} we find 
\begin{equation}
    \max_i\{|\bm\Delta_i|\}\prec_C\frac{N^{\frac4K}}{N\eta}+N^{\frac2K}\frac{1}{\sqrt{N}}.
    \label{iso bound away}
\end{equation}
Finally from \eqref{iso bound away}, \eqref{IL of Delta} and \eqref{bound error av} we obtain the averaged bound
\begin{equation}
    \max_i\{|\mb \Delta_i|_{\mr{av}}\}\prec_C\frac{N^{\frac4K}}{(N\eta)^2}+\frac1N.
\end{equation}
Therefore we have shown that \eqref{eq boot strapping away} holds on $\mathbb G_{\gamma_1}^{C,\eta_0}$.
\end{proof}

\begin{proof}[Proof of Propositions~\ref{prop LLL} and \ref{prop LLL away}]
For all $K>0$ Equation~\eqref{eq boot strapping} holds on $\mathbb D_1^{\kappa_0}$ for some $\kappa_0$ and \eqref{eq boot strapping away} on $\mathbb G_1^{C,\eta_0}$ for all $C>0$ and some $\eta_0$ depending on $C$ due to the global law Proposition~\ref{prop global law}.
The local laws, Proposition~\ref{prop LLL} and Proposition~\ref{prop LLL away}, follow immediately from applying Proposition~\ref{prop bootstrapping} finitely many times in the respective domains.
\end{proof}
This concludes the proof of Theorem~\ref{thm LL} and Proposition~\ref{prop LL away} and we are left with proving their Corollaries, \ref{cor delocalization} and \ref{cor rigidity}.

\begin{proof}[Proof of Corollary~\ref{cor delocalization}]
Let $\lambda_i$, $i\in\llbracket N\rrbracket$, denote the eigenvalues of $q(\mb X)$ with eigenvectors $\mb v_i$. let $\mb x\in\C^{N\times N}$ be a deterministic and normalized vector and let $\mb v$ be an eigenvector of $q(\mb X)$ with eigenvalue $\lambda$ such that $|\lambda-\tau_0|<\kappa_0$ for $\kappa_0$ from Theorem~\ref{thm LL}. Evaluating $\mb g$ at spectral parameter $\lambda+\I\eta$ have with very high probability that
\begin{equation}
    1\gtrsim\Im\langle \mb x,\mb g(\lambda+\I\eta)\mb x\rangle
    =\sum_{i=1}^N\eta\frac{\eta}{\eta^2+(\lambda-\lambda_i)^2}|\langle \mb x,\mb v_i\rangle|^2
    \geq \frac{|\langle \mb x,\mb v_i\rangle|}{\eta}
\end{equation}
for all $\eta\geq N^{-1+\veps}$ and all $\veps>0$. As the deterministic vector $\mb x$ was arbitrary, the eigenvalue delocalization, \eqref{eq deloc}, follows.
\end{proof}

\begin{proof}[Proof of Corollary~\ref{cor rigidity}]
Let $q$ have regular edge at $\tau_0$. W.l.o.g. we can assume that $\tau_0=\tau_+$ is a right edge. Otherwise, consider the right edge of $-q$.
From \eqref{eq ll imp} and Proposition~\ref{prop LL away} it follows that
\begin{equation}
    \Pb\left(|\langle \mb B(\mb g-m)\rangle|>\Vert \mb B\Vert N^\veps\left(\frac{1}{N(\kappa+\eta)}+\frac{1}{(N\eta)^2\sqrt{\kappa+\eta}}\right)\right)
    \lesssim_{\veps,\gamma,D,C}N^{-D}
\end{equation}
for all deterministic $\mb B\in\C^{N\times N}$ and $z\in\mathbb{D}_\gamma^{\kappa_0}\cup\mathbb G_\gamma^{C,\eta_0}$ with $E\notin\supp(\rho)$. In particular, this holds true for $C\geq\kappa_0^{-1}$. Following \cite[Chapter 11.1]{Erdos_Yau_2017}, this implies for all $\veps>0$ that with very high probability there are no eigenvalues $\lambda$ such that 
$\lambda\notin\supp(\rho)$ and $N^{-2/3+\veps}\leq\lambda-\tau_+\leq C$, i.e.
\begin{equation}
    \Pb\left(\exists \lambda\in\Spec(q(\mb X)) \, : \,  N^{-\frac23+\veps}\leq\lambda-\tau_+<C
    \right)
    \lesssim_{\veps,D,C} N^{-D}.
    \label{eq no eigenvalues 1}
\end{equation}
Additionally, it follows from the trivial bound \eqref{trivial bound} that there is a $K>0$ such that 
\begin{equation}
    \Pb\left(\Vert q(\mb X)\Vert\geq K
    \right)
    \lesssim_{D} N^{-D}.
    \label{eq no eigenvalues 2}
\end{equation}
Combining \eqref{eq no eigenvalues 1} for some $C>0$ such that $C+\tau_0\geq K$ and \eqref{eq no eigenvalues 2} we have
\begin{equation}
    \Pb\left(\exists \lambda\in\Spec(q(\mb X)) \, : \, \lambda-\tau_+\geq N^{-\frac23+\veps}
    \right)
    \lesssim_{\veps,D} N^{-D}.
    \label{eq no eigenvalues 3}
\end{equation}
By a standard argument, \eqref{eq no eigenvalues 3} in conjunction with the averaged local law, \eqref{eq ll av}, implies eigenvalue rigidity around the edge as in Theorem~\ref{cor rigidity} (see e.g. \cite[Chapter 11.2-11.4]{Erdos_Yau_2017}).
\end{proof}

\appendix
\section{Appendix}
\subsection{The entrywise real part of a Hermitian matrix.}
Let $n\in\N$ and $H\in\C^{n\times n}$ be a Hermitian matrix and $\widehat{H}=\frac12(H+H^t)$ be its entrywise real part. The following lemma summarizes how the two matrices are related.
\begin{lemma}
    Let $h_i$ and $\widehat{h}_i$ be the eigenvalues of $H$ and $\widehat H$ respectively, both arranged in non-increasing order. The following holds true
    \begin{enumerate}
        \item $h_1\geq\widehat{h}_1$ and $h_n\leq\widehat{h}_n$.
        \item $\Vert H\Vert\geq\Vert\widehat{H}\Vert$.
        \item If $H\geq0$, then $\widehat H\geq0$ and conversely $H\leq0$ implies $\widehat H\leq0$.
        \item Let additionally $\rank H=1$. Then we have $\rank\widehat H=1$ if and only if $H\in\R^{l\times l}$. Otherwise we have $\rank \widehat{H}=2$.
    \end{enumerate}
    \label{lemma hat A}
\end{lemma}
\begin{proof}
Let $\widetilde H=\frac1{2\I}(H-H^t)$ be the entrywise imaginary part of $H,$ i.e. $H=\widehat{H}+\I\widetilde H$.
First, we will prove $h_1\geq\widehat{h}_1$. Note that for all $a\in\R$ the norms of $H+a$ and $\widehat{H}+a$ are given by $\Vert H+a\Vert=\max\{h_1+a,-h_n-a\}$ and $\Vert \widehat{H}+a\Vert=\max\{\widehat{h}_1+a,-\widehat{h}_n-a\}$. Thus for sufficiently large $a\in\R$ (depending on $H$) we have $\Vert H+a\Vert=h_1+a$ and $\Vert \widehat{H}+a\Vert=\widehat{h}_1+a$. For such an $a$ let $v_1\in\R^n$ be a normalized eigenvector of $\widehat H+a$ corresponding to the eigenvalue $\widehat{h}_1+a$ (it can be chosen purely real since $\widehat{H}$ is a real symmetric matrix). Then
\begin{equation}
h_1+a=\Vert H+a\Vert\geq\Vert (H+a)v_1\Vert=\sqrt{\Vert (\widehat{H}+a)v_1\Vert^2+\Vert\widetilde{H}v_1\Vert^2}\geq\Vert (\widehat{H}+a)v_1\Vert=\widehat{h}_1+a.
\end{equation}
Here the second equality holds since $\widetilde H v_1$ is a purely imaginary vector. The claim $h_1\geq\widehat{h}_1$ follows. Choosing $a$ sufficiently small, we find $h_n\leq\widehat{h}_n$ by a similar argument. Since $\Vert H\Vert=\max\{h_1,-h_n\}$ and $\Vert \widehat{H}\Vert=\max\{\widehat{h}_1,-\widehat{h}_n\}$, the inequality $\Vert H\Vert\geq\Vert \widehat{H}\Vert$ follows. Next, let $H\geq0$. Then $h_n\geq0$ and thus $\widehat{h}_n\geq h_n\geq0$. The inequality $\widehat H\geq0$ follows. Similarly, $H\leq0$ implies $\widehat H\leq0$.

Now let $\rank H=1$. Then $\rank \widehat H=1$ if $H\in\R^{n\times n}$ is clear. Let on the other hand $H\in\C^{n\times n}\setminus\R^{n\times n}$. Then $H=\alpha v v^*$ for some $\alpha\in\R\setminus\{0\}$ and normalized $ v\in\C^n\setminus e^{\I\phi}\R^n$ for all $\phi\in\R$. Then $\widehat H=\frac{\alpha}{2}( v v^*+\bar{ v}\bar{ v}^*)$ and since $ v$ and $\bar{ v}$ are only linearly dependent if $ v\in e^{\I\phi}\R^n$ the claim $\rank \widehat{H}=2$ follows.
\end{proof}

\subsection{Matrix inequality}
Here we provide a proof for the matrix inequality used in \eqref{Eq G_0 est}.
\begin{lemma}
    Let $R,T\in\C^{n\times n}$ be arbitrary matrices. Then the following inequality holds:
    \begin{equation}
        (R+T)^*(R+T)\leq 2(R^*R+T^*T).
    \end{equation}
    \label{Lemma Matrix ineq}
\end{lemma}
\begin{proof}
We first note that
\begin{equation}
    (R+T)^*(R+T)=R^*R+T^*T+R^*T+T^*R
\end{equation}
and it is thus sufficient to bound the last two terms. Let $v\in\C^n$ be arbitrary. Since $R^*T+T^*R$ is Hermitian, $\langle v,(R^*T+T^*R)v\rangle\in\R$ and we estimate
\begin{equation}
    \langle v,(R^*T+T^*R)v\rangle=2\Re(\langle Rv,Tv\rangle)\leq2\Vert Rv\Vert\Vert Tv\Vert\leq\Vert Rv\Vert^2+\Vert Tv\Vert^2=\langle v,(R^*R+T^*T)v\rangle,
\end{equation}
where we used the Cauchy-Schwarz inequality in the second step. As $v$ was arbitrary, $R^*T+R^*T\leq R^*R+T^*T$ follows and we conclude the lemma.
\end{proof}

\subsection{Some properties of the solution of the Dyson Equation\label{Appendix DE}}
\begin{proof}[Proof of Proposition~\ref{prop SCE}]
    There can at most be one analytic function in the upper half-plane that satisfies \eqref{scalar equation} and $\lim_{z\to\infty}zm(z)=-1$ since \eqref{scalar equation} is stable at infinity. Thus we are left with proving existence.
    First, consider the case of $A$ being invertible. Let $\{s_1,\ldots,s_l\}$ be a family of free semi-circular variables in a $C^*$ probability space $(\IS,\tau)$ (see \cite[Appendix~B]{erdos2019local}) and let $s:=(s_i)_{i\in\llbracket l\rrbracket}$. Define
    \begin{equation}
        \mb L_{\mr{sc}}=K_0\otimes\mathbb1+\sum_{j=1}^lK_j\otimes s_j.
        \label{linearization SC}
    \end{equation}
    Following the proof of \cite[Lemma~2.6]{erdos2019local}, the matrix-valued function $M_0:\HH\to\C^{(l+1)\times(l+1)},$ given by
    \begin{equation}
        M_0(z)=(\mr{id}\otimes\tau)(\mb L_{\mr{sc}}-zJ\otimes\mathbb 1)^{-1},
    \end{equation}
    is a solution to \eqref{DE} for $\delta=0$. Thus its (1,1) entry satisfies \eqref{scalar eq delta} for $\delta=0$ and therefore \eqref{scalar equation}. Using the Schur complement formula, we find
    \begin{equation}
        (M_0(z))_{11}=\tau(q(s)-z\otimes\mathbb1)^{-1}.
    \end{equation}
    The polynomial $q(s)$, defined in \eqref{eq polynomial} is self-adjoint, thus $(M_0)_{11}$ is analytic on $\HH$, has positive imaginary part and $\lim_{z\to\infty}z(M_0(z))_{11}=-1$. Therefore $m:=(M_0)_{11}$ is the unique function that satisfies all conditions of Proposition~\ref{prop SCE}.
    
    Now, let $A$ be non-invertible. Then there is a $u\sim1$ such that $A^\veps:=A+\veps$ is invertible for all $0<\veps<u.$ Let $q^\veps$ and $\gamma^\veps$ be the objects given by replacing $A$ with $A+\veps$ in the definitions of $q$ in \eqref{eq polynomial} and $\gamma$ in \eqref{eq gamma}, respectively. By the above argument, the function 
    \begin{equation}
        m^\veps(z):=\tau(q^\veps(s)-z\otimes\mathbb1)^{-1}
    \end{equation}
    is a solution to the equation
    \begin{equation}
        -\frac{1}{m^\veps}=z+\gamma^\veps(m^\veps).
        \label{scalar eq epsilon}
    \end{equation}
    For any fixed $z\in\HH$, both $m^\veps$ and $\gamma^\veps$ are continuous in $\veps$ at $\veps=0$. Thus $m:=m^0$ solves \eqref{scalar equation} and $m$ is analytic on $\HH$, has positive imaginary part and satisfies  $\lim_{z\to\infty}zm(z)=-1$. Therefore it is the unique function that satisfies all conditions of Proposition~\ref{prop SCE}.
\end{proof}
The function $m$ is a Stieltjes transform of a real-valued probability measure $\rho$ with compact support. As such it can be analytically extended to $\C\setminus\supp(\rho).$ The following lemma summarizes the properties of the extension, also called $m$, on and above the real axis outside of the spectrum
\begin{lemma}
    For the analytical extension of $m$ to $\C\setminus\supp(\rho)$ we have the following.
    \begin{enumerate}
        \item The function $m$ is real-valued on $\R\setminus\supp(\rho)$ and $m'(E)>0$ for all $E\in\R\setminus\supp(\rho)$.
        \item For all $C>0$ there is an $\eta_0>0$ such that for all $z=E+\I\eta$ with $C^{-1}\leq \dist(E,\supp(\rho))\leq C$ and $0<\eta\leq \eta_0$ we have $\Im m\sim_C\eta.$
    \end{enumerate}
    \label{lemma m away from rho}
\end{lemma}
\begin{proof}
    The fact that $m(E)\in\R$ for all $\R\setminus\supp(\rho)$ follows immediately for from \eqref{eq m as trafo} and taking the derivative of \eqref{eq m as trafo} we find
    \begin{equation}
        m'(E)=\int_\R\frac{\rho(\dd x)}{(x-E)^2}>0.
        \label{derivative m}
    \end{equation}
    Using $m(E)\in\R$ we have for $z\in\HH$
    \begin{equation}
        \Im m(z)=\Im(m(z)-m(E))=\eta m'(E)+\mc O(\eta^2).
    \end{equation}
    By \eqref{derivative m} and continuity of the derivative, we have that $m'(E)$ is bounded from above and bounded away from 0 on all compact subsets of $\R\setminus\supp(\rho)$. Therefore we have $m'(E)\sim_C1$ for all $C>0$ and $E$ such that $C^{-1}\leq \dist(E,\supp(\rho))\leq C$.
\end{proof}

\begin{proof}[Proof of Lemma~\ref{lemma m delta asymptotics}]
    The function $m$ is a Stieltjes transform and as such analytically extendable to $\C\setminus\supp(\rho)$. Since, by assumption, $\rho$ has a regular edge at $\tau_\pm$ it can also be continuously extended to $\tau_\pm$. For $x\in(\R\setminus\supp(\rho))\cup\{\tau_\pm\}$ we denote this extension by $m(x).$ By continuity it satisfies
    \begin{equation}
        -\frac{1}{m(x)}=x+\gamma(m(x)).
        \label{scalar eq on R}
    \end{equation}
    Consider \eqref{scalar eq delta} at $z\in\HH$ as a perturbation to \eqref{scalar eq on R}. More precisely, consider the following equation in the unknown function $\widetilde m_\delta(z)$,
    \begin{equation}
        \frac{1}{m(x)}-\frac1{\widetilde{m}_\delta(z)}=\gamma_\delta(\widetilde{m}_\delta(z))-\gamma_0(m(x))+(z-x),
        \label{eq m diff}
    \end{equation}
    where we have used $\gamma_0=\gamma.$
    Define by $\dot\gamma_\delta(y):=\partial_t\gamma_{\frac{t}{\I\eta}}(y)|_{t=\I\delta\eta}$ the partial derivative with respect to $\I\delta\eta$ and by $\gamma_\delta'(y):=\partial_y \gamma_\delta(y)$ the derivative with respect to its variable.

    Expanding \eqref{eq m diff} in $m_\delta(z)-m(x)$ and $\I\delta\eta$ we find 
\begin{equation}
\begin{split}
    &\left(\frac{1}{m(x)^2}-\gamma_0'(m(x))\right)(\widetilde{m}_\delta(z)-m(x))-\left(\frac{1}{m(x)^3}+\frac{1}{2}\gamma_0''(m(x))\right)(\widetilde{m}_\delta(z)-m(x))^2\\
    & \qquad\qquad =\I\dot\gamma_0(m(x))\delta\eta+(z-x)+\mc O(|\widetilde{m}_\delta(z)-m(x)|^3+|\widetilde{m}_\delta(z)-m(x)|\eta+\eta^2).
    \label{expansion m delta}
\end{split}
\end{equation}
We identify $\frac{1}{m^2}-\gamma_0'(m)=h(m)$ and $-\frac{2}{m^3}-\gamma_0''(m)=h'(m)$ with $h$ defined in \eqref{function h}. Now let $x=\tau_\pm$ and note that $m(\tau_\pm)=m_\pm.$ By Lemma~\ref{lemma h non-red} as well as Lemma~\ref{lemma h red} we have 
\begin{equation}
    h(m_\pm)=0,
    \quad
    |h'(m_\pm)|\sim1
    \quad
    \text{and}
    \quad
    \pm h'(m_\pm)>0.
    \label{h(z) asymptotics}
\end{equation}
Furthermore one finds $\dot\gamma_0(y)>0$ for all $y\in\R\setminus\IS(\gamma)$. Thus \eqref{expansion m delta} constitutes an approximate quadratic equation in $\widetilde{m}_\delta(z)-m_\pm$ and there are $u, u_1\sim1$ such that for all $z\in\HH$ with $0<|z-\tau_\pm|\leq u$, there are exactly two solutions $\widetilde{m}_\delta(z)$ with $|\widetilde{m}_\delta(z)-m_\pm|\leq u_1$. One of them has positive imaginary part and one of them has negative imaginary part. Combining \eqref{expansion m delta} and \eqref{h(z) asymptotics}, we find that the unique solution with positive imaginary part satisfies both parts of \eqref{eq m_delta asymptotics} and we denote it by $m_\delta$.

Now let $C>0$, $x=E=\Re z$ and $C^{-1}<\dist(E,\supp(\rho))<C$. Then we have 
\begin{equation}
    h(m(E))=(m'(E))^{-1}\sim_C1
    \label{eq h away from rho}
\end{equation}
by Lemma~\ref{lemma m away from rho} and the fact that $C^{-1}<\dist(E,\supp(\rho))<C$ constitutes a compact subset of $\R\setminus\supp(\rho)).$ Therefore \eqref{expansion m delta} is an approximate linear equation in $\widetilde{m}_\delta(z)-m(E)$. Hence there are $u,u_1\sim1$ such that for all $0<|z-E|=\eta\leq u$ there is a unique solution $\widetilde{m}_\delta(z)$ with $|\widetilde{m}_\delta(z)-m(E)|\leq u_1$ to \eqref{expansion m delta}. We denote this solution by $m_\delta$ and combining \eqref{h(z) asymptotics} with \eqref{eq h away from rho} we find that it satisfies both parts of \eqref{eq m_delta away from rho}.

From now on let $A$ be invertible and
\begin{equation}
    M_\delta(z)=(\mr{id}\otimes\tau)(\mb L_{\mr{sc}}-(zJ+\delta\eta(I-J))\otimes\mathbb 1)^{-1},
\end{equation}
where $\mb L_{\mr{sc}}$ was defined in \eqref{linearization SC}. For all $\eta>0$, the imaginary part of $zJ+\delta\eta(I-J)$ is positive. Therefore, by Lemma~\cite[Lemma~5.4]{haagerup2005}, the function $M_\delta$ is the unique solution to \eqref{DE} with $\Im M_\delta(z)>0$. In particular, this implies that its (1,1) entry satisfies \eqref{scalar eq delta}. By the Schur complement formula, it is given by
\begin{equation}
    (M_\delta(z))_{11}=\tau(q_\delta(s)-z\otimes\mathbb1)^{-1},
\end{equation}
where $q_\delta$ denotes the polynomial $q$, defined in \eqref{eq polynomial}, with $A$ replaced by $A_\delta$. For $E\notin\supp(\rho)$ consider
\begin{equation}
\begin{split}
    (M_\delta)_{11}-m
    &=\tau((q_\delta(s)-z\otimes\mathbb1)^{-1}-(q_0(s)-z\otimes\mathbb1)^{-1})\\
    &=\tau((q_\delta(s)-z\otimes\mathbb1)^{-1}(q_0(s)-q_\delta(s))(q_0(s)-z\otimes\mathbb1)^{-1}).
    \label{eq factorization m diff}
\end{split}
\end{equation}
The middle factor satisfies
\begin{equation}
    \Vert q_0(s)-q_\delta(s)\Vert=\left\Vert\sum_{i,j=1}^l(A-A_\delta)s_is_j\right\Vert\lesssim\eta\delta. 
    \label{bound factor 1}
\end{equation}
Since $E\notin\supp(\rho)$, the term $q_0(s)-E$ is invertible and by continuity in $\eta$ we have 
\begin{equation}
    \Vert (q_0(s)-z\otimes\mathbb1)^{-1}\Vert\lesssim_E 1
    \label{bound factor 2}
\end{equation}
uniformly in $\eta$ for $0<\eta\leq1$. In particular, for $\eta$ sufficiently small (depending on $E$), we have
\begin{equation}
    \Vert(q_0(s)-z\otimes\mathbb1)^{-1}(q_0(s)-q_\delta(s))\Vert\leq\frac12.
\end{equation}
Therefore the first factor in \eqref{eq factorization m diff} factorizes further into
\begin{equation}
    (q_\delta(s)-z\otimes\mathbb1)^{-1}=(1\otimes\mathbb1 +(q_0(s)-z\otimes\mathbb1)^{-1}(q_0(s)-q_\delta(s)))^{-1}(q_0(s)-z\otimes\mathbb1)^{-1}
    \label{bound factor 3}
\end{equation}
and
\begin{equation}
    \|(1\otimes\mathbb1 +(q_0(s)-z\otimes\mathbb1)^{-1}(q_0(s)-q_\delta(s)))^{-1} \| \leq 2.
    \label{bound factor 4}
\end{equation}
Combining \eqref{eq factorization m diff} with \eqref{bound factor 1},\eqref{bound factor 2},\eqref{bound factor 3} and \eqref{bound factor 4}, we find 
\begin{equation}
    |(M_\delta)_{11}-m|\lesssim_E\delta\eta.
    \label{eq diff M delta m}
\end{equation}
As shown in the first part of the proof, there are $u,u_1>0$ such that \eqref{eq m diff} has a unique solution $m_\delta$ with positive imaginary part and $|m_\delta(z)-m_\pm|\leq u_1$ for all $z\in\HH$ such that $|z-\tau_\pm|\leq u$. For such $u,u_1$ we choose a $u_2$ with $0<u_2<u$ such that $|m_z-m_\pm|\leq\frac{u_1}{2}$ for all $z\in\HH$ with $|z-\tau_\pm|\leq u_2$ (such a $u_2$ exists by Corollary~\ref{cor m reg}). Fix $E\notin \supp(\rho)$ with $|E-\tau_\pm|<u_2$. We choose $\eta=\Im z$ sufficiently small such that $|z-\tau_\pm|<u_2$ and $|(M_\delta)_{11}-m|\lesssim\frac{u_1}{2}$ (which is possible by \eqref{eq diff M delta m}). Thus there is a $z\in\HH$ with $|z-\tau_\pm|\leq u$ such that
\begin{equation}
    |(M_\delta)_{11}(z)-m_\pm|\leq|(M_\delta)_{11}(z)-m(z)|+|m(z)-m_\pm|\leq u_1.
\end{equation}
Since this condition is unique among solutions to \eqref{scalar eq delta} with positive imaginary part, we must have $(M_\delta)_{11}(z)=m_\delta(z)$ for some $z\in\HH$ with $|z-\tau_\pm|\leq u$. By continuity of both $m_\delta$ and $(M_\delta)_{11}$ we must have $(M_\delta)_{11}(z)=m_\delta(z)$ for all $z\in\HH$ with $|z-\tau_\pm|\leq u$.

For $C>0$ and $E$ with $C^{-1}<\dist(E,\supp(\rho))<C$ we find similarly from Lemma~\ref{lemma m away from rho} and \eqref{eq diff M delta m} that for all $u_1>0$ there is an $\eta>0$ such that
\begin{equation}
    |(M_\delta)_{11}(z)-m(E)|\leq|(M_\delta)_{11}(z)-m(z)|+|m(z)-m(E)|\leq u_1.
\end{equation}
Thus $(M_\delta)_{11}(z)$ must be the unique solution to \eqref{scalar eq delta} that satisfies \eqref{eq m_delta away from rho}, which concludes the proof of the lemma.
\end{proof}

\bibliographystyle{plain}
\bibliography{ReferencesPolynorm}

\end{document}